\documentclass[11pt]{amsart}

\usepackage{xcolor}
\definecolor{darkgreen}{rgb}{0,0.45,0}
\definecolor{darkred}{rgb}{0.75,0,0}
\definecolor{darkblue}{rgb}{0,0,0.545}
\usepackage[pdfborder=0,colorlinks,citecolor=blue,linkcolor=violet,urlcolor=teal]{hyperref}
\usepackage{fullpage}
\usepackage{amssymb}
\usepackage{comment}
\usepackage{enumitem}
\usepackage{graphicx}
\usepackage{ifthen}
\usepackage{stmaryrd}
\usepackage{microtype}
\usepackage[backend=bibtex,style=alphabetic,backref=true,natbib=true,maxalphanames=7,maxbibnames=99]{biblatex}
\DeclareFieldFormat{postnote}{#1}
\DeclareFieldFormat{multipostnote}{#1}

\usepackage{mathpartir}

\usepackage{tikz}
\usepackage{tikz-cd}
\usepackage{makebox}

\usetikzlibrary{decorations.markings}
\usetikzlibrary{decorations.pathmorphing}
\tikzset{
  module/.style={
    postaction={decorate},
    decoration={
      markings,
      mark=at position #1 with {\arrow{|}}}},
  module/.default=0.5,
  we/.style=
  { postaction={%
      decorate,
      decoration={
        markings,
        mark=at position #1 with {%
          \node[transform shape, yshift=.2em]{%
            \resizebox{0.5em}{!}{$\sim$}};}}}},
  we/.default=0.5,
  we'/.style=
  { postaction={%
      decorate,
      decoration={
        markings,
        mark=at position #1 with {%
          \node[transform shape, yshift=-.2em, rotate=180]{%
            \resizebox{0.5em}{!}{$\sim$}};}}}},
  we'/.default=0.5,
  iso/.style=
  { postaction={%
      decorate,
      decoration={
        markings,
        mark=at position #1 with {%
          \node[transform shape, yshift=.2em]{%
            \resizebox{0.5em}{!}{$\simeq$}};}}}},
  iso/.default=0.5,
  iso'/.style=
  { postaction={
      decorate,
      decoration={
        markings,
        mark=at position #1 with {%
          \node[transform shape, yshift=-.2em, rotate=180]{%
            \resizebox{0.5em}{!}{$\simeq$}};}}}},
  iso'/.default=0.5,
}

\tikzset{
  proarrow/.style={->, module},
  proequal/.style={-, double, module},
  prodotted/.style={->,dotted, module},
  prodashed/.style={->,dashed, module},
  wearrow/.style={->, we},
  wedashed/.style={->, dashed, we},
  wedotted/.style={->, dotted, we},
  tfibarrow/.style={->>, we=0.45},
  tfibdotted/.style={->>,dotted, we=0.45},
  tfibdashed/.style={->>,dashed, we=0.45},
  tcofarrow/.style={>->, we},
  tcofdashed/.style={>->, dashed, we},
  uwearrow/.style={->, we'},
  uwedashed/.style={->, dashed, we'},
  uwedotted/.style={->, dotted, we'},
  utfibarrow/.style={->>, we'=0.45},
  utfibdotted/.style={->>,dotted, we'=0.45},
  utfibdashed/.style={->>,dashed, we'=0.45},
  utcofarrow/.style={>->, we'},
  isoarrow/.style={->, iso},
  isodashed/.style={->, dashed, iso},
  uisoarrow/.style={->, iso'},
  uisodashed/.style={->, dashed, iso'},
  isocell/.style={=>, iso},
  isocelldashed/.style={=>, dashed, iso},
  uisocell/.style={=>, iso'},
  uisocelldashed/.style={=>, dashed, iso'}
}

\makeatletter
\def\makeslashed#1#2#3#4#5{#1{\mathpalette{\sla@{#2}{#3}{#4}}{#5}}}

\def\@mathlower#1#2#3{\setbox0=\hbox{$\m@th#2#3$}\lower#1\ht0\box0}
\def\mathlower#1#2{\mathpalette{\@mathlower{#1}}{#2}}
\makeatother

\newcommand\dhxrightarrow[2][]{%
  \mathrel{\ooalign{$\xrightarrow[#1\mkern4mu]{#2\mkern4mu}$\cr%
  \hidewidth$\rightarrow\mkern4mu$}}
}

\newcommand\tailxrightarrow[2][]{%
  \mathrel{\ooalign{$\xrightarrow[#1\mkern4mu]{#2\mkern4mu}$\cr%
  \hidewidth$\Yright\mkern14mu$}}
}

\newcommand{\fto}{\twoheadrightarrow}
\newcommand{\cto}{\rightarrowtail}
\newcommand{\wto}{\xrightarrow{{\smash{\mathlower{0.8}{\sim}}}}}
\newcommand{\cwto}{\tailxrightarrow{{\smash{\mathlower{0.8}{\sim}}}}}
\newcommand{\fwto}{\dhxrightarrow{{\smash{\mathlower{0.8}{\sim}}}}}

\DeclareMathAlphabet{\mathbbe}{U}{bbold}{m}{n}

\setlist{}
\setenumerate{leftmargin=*,labelindent=0\parindent}
\setitemize{leftmargin=\parindent}

\newtheorem{thm}{Theorem}[subsection]
\newtheorem{lem}[thm]{Lemma}
\newtheorem{prop}[thm]{Proposition}
\newtheorem{cor}[thm]{Corollary}

\theoremstyle{definition}
\newtheorem{defn}[thm]{Definition}
\newtheorem{con}[thm]{Construction}

\newtheorem{ex}[thm]{Example}

\newtheorem{nt}[thm]{Notation}

\theoremstyle{remark}
\newtheorem{rmk}[thm]{Remark}

\newtheorem{dig}[thm]{Digression}

\usepackage{bbm}
\usepackage[bbgreekl]{mathbbol}
\DeclareSymbolFontAlphabet{\mathbbl}{bbold}

\DeclareMathSymbol\DDelta \mathord{bbold}{"01}
\DeclareMathSymbol\SSigma\mathord{bbold}{"06}
\newcommand{\OOmega}{\mbox{$\,\mbox{\raisebox{0.45ex}{$\shortmid$}}\hspace{-0.37em}\Omega$}}
\DeclareMathSymbol\PPhi\mathord{bbold}{"08}

\newcommand{\CCube}{\mbox{$\,\mbox{\raisebox{0.3ex}{$\shortmid$}}\hspace{-0.35em}\square$}}

\DeclareMathSymbol{\LLat}\mathord{bbold}{"03}

\makeatletter
\let\c@equation\c@thm
\makeatother
\numberwithin{equation}{subsection}

\newcommand{\op}{\textup{op}}
\newcommand{\id}{\textup{id}}
\newcommand{\cod}{\textup{cod}}
\newcommand{\dom}{\textup{dom}}
\newcommand{\ob}{\textup{ob}}
\newcommand{\mor}{\textup{mor}}
\newcommand{\To}{\Rightarrow}
\newcommand{\ev}{\textup{ev}}

\renewcommand{\deg}{\textup{deg}}
\newcommand{\colim}{\textup{colim}}

\renewcommand{\AA}{\mathbb{A}}
\newcommand{\BB}{\mathbb{B}}
\newcommand{\CC}{\mathbb{C}}

\newcommand{\FF}{\mathbb{F}}

\newcommand{\II}{\mathbb{I}}
\newcommand{\JJ}{\mathbb{J}}
\newcommand{\KK}{\mathbb{K}}

\newcommand{\NN}{\mathbb{N}}
\newcommand{\PP}{\mathbb{P}}

\newcommand{\RR}{\mathbb{R}}
\newcommand{\TT}{\mathbb{T}}
\newcommand{\UU}{\mathbb{U}}
\newcommand{\VV}{\mathbb{V}}

\newcommand{\XX}{\mathbb{X}}
\newcommand{\YY}{\mathbb{Y}}
\newcommand{\ZZ}{\mathbb{Z}}

\newcommand{\1}{\mathbbe{1}}
\newcommand{\2}{\mathbbe{2}}

\newcommand{\mathscr}{\mathcal}
\usepackage{mathtools}
\newcommand{\coloneq}{\coloneqq}
\newcommand{\eqcolon}{\eqqcolon}

\newcommand{\cat}[1]{\textup{\textsf{#1}}}

\newcommand{\elf}{\cat{el}}

\newcommand{\hato}{\mathbin{\hat{\circ}}}

\newcommand{\cA}{\mathsf{A}}
\newcommand{\cB}{\mathsf{B}}
\newcommand{\cC}{\mathsf{C}}
\newcommand{\cD}{\mathsf{D}}
\newcommand{\cE}{\mathsf{E}}

\newcommand{\cG}{\mathsf{G}}
\newcommand{\cH}{\mathsf{H}}
\newcommand{\cI}{\mathsf{I}}

\newcommand{\cK}{\mathsf{K}}

\newcommand{\cM}{\mathsf{M}}
\newcommand{\cN}{\mathsf{N}}

\newcommand{\cX}{\mathsf{X}}

\newcommand{\Left}{\mathcal{L}}
\newcommand{\Right}{\mathcal{R}}
\newcommand{\Mono}{\mathcal{M}}
\newcommand{\NMono}{\mathcal{N}}
\newcommand{\Epi}{\mathcal{E}}

\newcommand{\sE}{\mathfrak{E}}
\newcommand{\sF}{\mathfrak{F}}

\newcommand{\Fin}{\cat{Fin}}
\newcommand{\Set}{\cat{Set}}
\newcommand{\sSet}{\cat{sSet}}
\newcommand{\cSet}{\cat{cSet}}
\newcommand{\lSet}{\ell\cat{Set}}

\newcommand{\Cat}{\cat{Cat}}

\newcommand{\Ladj}{\cat{Ladj}}
\newcommand{\Radj}{\cat{Radj}}

\newcommand{\TF}{\TT\FF}

\font\maljapanese=dmjhira at 2ex 
\def\yo{\textrm{\maljapanese\char"48}}

\newcommand{\Leftadj}{\mathrm{L}}

\newcommand{\isContr}{\textup{isContr}}

\newcommand{\Eq}{\textup{Eq}}

\newcommand{\leib}[1]{\mathbin{\hat{#1}}}

\newcommand{\sk}{\textup{sk}}

\newcommand{\latch}[1]{\widehat{\ell}_{#1}}

\newcommand{\lbound}[1]{\overleftarrow{\partial}\!{#1}}
\newcommand{\rbound}[1]{\overrightarrow{\partial}\!{#1}}

\newcommand{\utimes}{\underline{\times}}

\newcommand{\Aut}{\textup{Aut}}
\newcommand{\inj}{\textup{inj}}

\usepackage{xcolor}

\definecolor{darkorange}{rgb}{0.8,0.4,0}
\definecolor{darkblue}{rgb}{0.08,0.38,0.74}
\definecolor{darkpurple}{rgb}{0.6,0.2,0.8}
\definecolor{darkgreen}{rgb}{0,0.5,0}
\definecolor{darkteal}{rgb}
{0.0, 0.5, 0.5}
\definecolor{darkred}{rgb}{0.5,0.0, 0.13}

\begin{document}

\title{The equivariant model structure on cartesian cubical sets}
\author{Steve Awodey}
\address{Departments of Philosophy and Mathematics\\
Carnegie Mellon University\\
Pittsburgh, PA 15213, USA}
\email{awodey@cmu.edu}

\author{Evan Cavallo}
\address{Department of Computer Science and Engineering  \\ Chalmers University of Technology and University of Gothenburg \\ 405 30 G\"oteborg, Sweden}
\email{evan.cavallo@gu.se}

\author{Thierry Coquand}
\address{Department of Computer Science and Engineering \\ Chalmers University of Technology and University of Gothenburg \\ 405 30 G\"oteborg, Sweden}
\email{coquand@chalmers.se}

\author{Emily Riehl}
\address{Department of Mathematics\\Johns Hopkins University \\ Baltimore, MD 21218 USA}
\email{eriehl@jhu.edu}

\author{Christian Sattler}
\address{Department of Computer Science and Engineering \\ Chalmers University of Technology and University of Gothenburg \\ S-412 96 G\"oteborg, Sweden}
\email{sattler@chalmers.se}

\date{\today}

\begin{abstract} We develop a constructive model of homotopy type theory in a Quillen model category that classically presents the usual homotopy theory of spaces.  Our model is based on presheaves over the cartesian cube category, a well-behaved Eilenberg--Zilber category. The key innovation is an additional equivariance condition in the specification of the cubical Kan fibrations, which can be described as the pullback of an interval-based class of uniform fibrations in the category of symmetric sequences of cubical sets. The main technical results in the development of our model have been formalized in a computer  proof assistant.
\end{abstract}

\maketitle

\setcounter{tocdepth}{2}
\tableofcontents

\section{Introduction}

\subsection{Interpreting homotopy type theory}\label{subsec:cubical_models_of_HoTT}

Martin-L\"{o}f's dependent type theory \cite{ml75,nordstrom-petersson-smith:ml} provides a foundation for constructive mathematics.
It functions both as a formal language for mathematical arguments and as a programming language: proofs of mathematical statements in Martin-L\"of type theory can be regarded as functions or algorithms with computational content.
At the turn of the 21st century, higher-dimensional and ultimately homotopical interpretations of Martin-L\"of type theory were discovered \cite{HofmannStreicher:1997lg,awodey-warren:2009,BergGarner:2012,KapulkinLumsdaine:2021sm}.
The novelty of these interpretations is concentrated in their treatment of \emph{identities}, i.e.\ equalities: an identity between two elements of a type is interpreted as a \emph{path} or higher cell connecting them.
\emph{Homotopy type theory} (HoTT) or \emph{univalent foundations} (UF) \cite{UF:2013} refers to the formal system of Martin-L\"{o}f type
theory augmented by Voevodsky's \emph{univalence axiom}, which asserts that a certain canonical map is an equivalence
\begin{equation}\label{eq:UA}
(A =_U B)\ \simeq\  (A\simeq B)
\end{equation}
between the type $A =_U B$ of identities between types $A,B$ in a universe $U$ and the type $A \simeq B$ of homotopy equivalences between them.

To establish the consistency of the univalence axiom with the rules of Martin-L\"of type theory (relative to the consistency of the rest of mathematics), Voevodsky \cite{KapulkinLumsdaine:2021sm} built a model of homotopy type theory using the standard model of homotopy theory in simplicial sets.
The construction makes use of the \emph{Quillen model structure} on simplicial sets \cite{Quillen:1967ha}, which exhibits this category as a setting for abstract homotopy theory.
In particular, dependent type families are interpreted as the fibrations of this model structure (the \emph{Kan fibrations}), and the interpretations of type formers rely on established properties of the model structure; for example, the interpretation of $\Pi$-types rests on the fact that the model structure is \emph{right proper} \cite[2.3.1]{KapulkinLumsdaine:2021sm}.

Voevodsky's definition of the model relies on classical principles of reasoning such as the law of excluded middle and the axiom of choice, a surprising dependency given the constructive character of type theory itself.
Bezem, Coquand, and Parmann \cite{BezemCoquand:2015,BezemCoquandParmann:2015,Parmann:2018} showed that components of the model are in fact inherently non-constructive (though see \S\ref{sssec:constructive-simplicial} below).
Thus one is also interested in finding models that can be defined using only constructively valid reasoning. Such a model would, in particular, construct an explicit equivalence inverse to the map \eqref{eq:UA}, supplying computational content to proofs that invoke the univalence axiom.

\subsection{Cubical interpretations}

The first step towards a constructive model was taken by Bezem, Coquand, and Huber (BCH) \cite{BCH:2014mt}, who gave a partial constructive interpretation of homotopy type theory that was later completed in \cite{BCH:2019ua}.
Their interpretation replaces simplicial sets by a form of \emph{cubical} sets, i.e.~presheaves on a \emph{cube category},\footnote{There is a wide variety of cube categories in common use. In all instances, the objects are indexed by the natural numbers, defining an ``$n$-cube'' for each $n \in \NN$, and there are \emph{face} and \emph{degeneracy} maps that include exterior faces or project away from one of the $n$ dimensions of an $n$-cube. Other optional structure is given by maps that encode automorphisms of an $n$-cube in the form of either \emph{symmetries} or \emph{reversals}, \emph{diagonal} face maps, or extra degeneracies in the form of \emph{connections}.} thereby avoiding non-constructive elements of Voevodsky's model.
Cohen, Coquand, Huber, and M\"{o}rtberg (CCHM) \cite{CCHM:2018ctt} gave a complete constructive interpretation in a second, highly structured form of cubical sets; in this setting they were also able to interpret \emph{higher inductive types} \cite{CHM18}.
Angiuli, Favonia, and Harper \cite{AngiuliFavoniaHarper:2018} described a computational interpretation based on a third cube category proposed by Awodey
\cite{awodey:2018}, the \emph{cartesian cubes}, using a definition of fibration for these cubical sets proposed by Coquand \cite{Coquand:2014vc}.
This work cuts down the cube category from the CCHM model, retaining only the diagonal face maps that are apparently essential for interpreting higher inductive types.
The computational interpretation was then translated to a cubical set interpretation by Angiuli et al.~\cite{ABCFHL}.
The CCHM and cartesian interpretations can now be understood as instances of a single construction that works for any cube category with at least cartesian structure \cite{CMS20}.

The inspiration for these models traces back to Kan's early work on abstract homotopy theory in cubical sets \cite{Kan:1955}.
His \emph{E-complexes} (now called \emph{cubical Kan complexes}) are the fibrant objects of a Quillen model structure whose fibrations are the maps satisfying a \emph{box-filling} property \cites[\S3]{Jardine:2002}[8.4.38]{Cisinski:2006lp}.
Relative to simplicial sets, one essential feature of cubical sets for constructive interpretation of type theory seems to be the closure of cubes under a symmetric monoidal product---the product of the $m$-cube and $n$-cube is the $(m+n)$-cube---which plays a key role in the construction of universes.
The monoidal product of cubes in Kan's cubical sets is \emph{not} symmetric, so the techniques do not seem to yield a constructive interpretation in these cubical sets.
There is, however, a great variety of symmetric cube categories that one can consider (see, e.g., Grandis and Mauri \cite{GrandisMauri:2003} and Buchholtz and Morehouse \cite{BuchholtzMorehouse:2017vo}), hence the proliferation of cubical interpretations.

\subsection{Cubical model structures}
\label{ssec:cubical-model-structures}

Although the cubical interpretations of HoTT/UF were not introduced using Quillen model categories, a posteriori each can be seen \cite{Sattler:2017ee,CMS20,Awodey:2023} to determine a Quillen model structure in the following sense.%
\footnote{
  The existence of a model structure associated to the BCH interpretation has not appeared explicitly in the literature, to our knowledge, but it can be deduced from known results.
  Swan constructs the two factorization systems \cites{Swan:2016awfs}[\S7.5.3]{swan18lp}, and these form a cylindrical premodel structure in the sense of \S\ref{sec:cylindrical} with all objects cofibrant.
  The 2-out-of-3 property then follows from the existence of fibrant universes \cite[{\S}I.4]{huber16phd} via Lemma \ref{lem:fibU_fibext} and Proposition \ref{prop:cylindrical-model-category}.
}
Given a model of Martin-L\"{o}f type theory in the form of a suitably-structured category with families \cite{Dybjer:1996} or natural model \cite{Awodey:2018nm} (such as the aforementioned cubical interpretations), we consider its category of contexts.
On this category, we have a candidate class of \emph{fibrations}: the retracts of context projections $p_A \colon \Gamma.A \to \Gamma$ associated to semantic types $\Gamma \vdash A$.
We also have a candidate class of \emph{trivial fibrations}: those fibrations derived from \emph{contractible} semantic types in the sense of HoTT \cite[3.11]{UF:2013}.%
\footnote{
  We do not claim that this is the only sensible way to associate a candidate model structure to a model of type theory.
  Note that all objects are cofibrant in any model structure of this form, as every contractible type has a section.
  By contrast, in work on constructive simplicial models of homotopy theory and type theory (discussed in \S\ref{sssec:constructive-simplicial} below), one works with a model structure in which not all objects are cofibrant and uses the full subcategory of cofibrant objects as the category of contexts for the model of type theory.
}
A Quillen model structure is completely characterized by its classes of fibrations and trivial fibrations, but not every choice of classes forms a Quillen model structure.
First, each class should form the right class of a weak factorization system.
If this is the case, we have a \emph{premodel structure} in the sense of Barton \cite{Barton:2019pm}.
A premodel structure determines a candidate class of weak equivalences, and a premodel structure is a Quillen model structure exactly if this class satisfies the 2-of-3 property.
When said property is satisfied, we may speak of the \emph{Quillen model structure associated to} the model of type theory.

As the semantic types of the cubical interpretations are defined by right lifting properties, it is not so surprising that the induced classes of fibrations and trivial fibrations indeed define a premodel structure.
The main technical challenge lies in checking the 2-of-3 property.
In a reversal of the history of Voevodsky's simplicial model, the property is verified using components of the model of type theory.
In particular, the \emph{fibration extension property} and \emph{equivalence extension property}, used to interpret universes and univalence respectively, play a direct role \cite{Sattler:2017ee}, as does the \emph{Frobenius condition} \cites[3.3.3]{BergGarner:2012}{GambinoSattler:2017fc} used to interpret $\Pi$-types.

\subsection{Standard homotopy theory}

Having associated a Quillen model structure to each cubical interpretation, we are in a position to ask what \emph{homotopy theory} each presents, i.e., to characterize the $(\infty,1)$-categories they present.
In particular, we would like to know if any constructive interpretation of HoTT presents the $(\infty,1)$-category of \emph{spaces} (or \emph{homotopy types} or \emph{$\infty$-groupoids}), as does Voevodsky's classical model in simplicial sets.
While we might ultimately hope to interpret HoTT in all $\infty$-toposes, as accomplished by Shulman in the classical setting \cite{Shulman:2019ai}, interpretation in spaces is a fundamental motivation for synthetic homotopy theory in HoTT.

In model-categorical language, we seek a \emph{Quillen equivalence} (or zigzag of Quillen equivalences) between a model structure associated to a cubical interpretation and some model category known to present spaces, such as the classical Kan--Quillen model structure on simplicial sets.
In fact, we can compare more directly with existing classical model structures on cubical sets.
Buchholtz and Morehouse \cite{BuchholtzMorehouse:2017vo} observe that each of the cube categories used to model type theory is a so-called \emph{test category}.
The theory of test categories, initiated by Grothendieck \cite{Grothendieck:pursuing} and continued by Maltsiniotis \cite{Maltsiniotis:2005} and Cisinski \cite{Cisinski:2006lp}, guarantees (using classical logic) that the category of presheaves on any test category admits a Quillen model structure presenting the homotopy theory of spaces. Thus, we may also ask whether the Quillen model structure associated to a cubical interpretation coincides with the test model structure. This is a stronger condition, as multiple equivalent but non-identical Quillen model structures can exist on the same underlying category.

These questions were first discussed in 2018, at the Hausdorff Institute Trimester, and a number of negative results became folklore, discussed on the Homotopy Type Theory mailing list \cite{CounterexampleDiscussion:2018} and sketched in \cite{Sattler:2018dc}.
The upshot is that many cubical interpretations do not present spaces, and a fortiori do not coincide with the corresponding test model structures.
In particular, the BCH model in the minimal symmetric monoidal cube category does not.
The later model constructions yield interpretations in any cube category with cartesian products, so there are many candidates to consider here.
However, neither the \emph{De Morgan cube category} with connections and reversals, which is the focus of \cite{CCHM:2018ctt}, nor the minimal \emph{cartesian cube category} considered in \cite{ABCFHL,Awodey:2023}, gives a model of spaces.
It is an open question whether the interpretation in the \emph{Dedekind cube category} (with cartesian structure and connections) presents spaces.

This brings us to the main result of this article, the construction of a cubical interpretation that \emph{does} classically present the homotopy theory of spaces.

\subsection{The equivariant cubical model}

We define a new model of HoTT with an associated Quillen model structure, the \emph{equivariant cartesian model}, by modifying the original cartesian cubical set model of Angiuli et al., replacing its fibrations with a more restrictive class of \emph{equivariant fibrations}.

\subsubsection{The problem}

Our definition of equivariant fibration is motivated by a specific pathology in the original Quillen model structure associated to the model of type theory on cartesian cubical sets \cite{CMS20,Awodey:2023}, namely the \emph{non-contractibility of automorphism quotients of cubes}.
In cartesian cubical sets, the group of automorphisms of the representable $n$-cube $I^n$ is the symmetric group $\Sigma_n$: the only automorphisms are the permutations of the axes of the cube.
For any subgroup $H \subset \Sigma_n$, we then have a quotient $I^n_{/H} \in \cSet$, the colimit of the $H$-indexed diagram sending a permutation to the corresponding automorphism of $I^n$.
When $H$ is non-trivial, $I^n_{/H}$ is not contractible in this model structure.

First, to see why this is problematic, let us consider a natural comparison to a model category presenting the homotopy theory of spaces: the adjunction
\[
  \begin{tikzcd} \cSet \arrow[rr, bend left, "T"] \arrow[rr, phantom, "\bot"] && \sSet \arrow[ll, bend left]
  \end{tikzcd}
\]
between cartesian cubical and simplicial sets whose left adjoint, \emph{triangulation}, sends the $n$-cube to the $n$-ary cartesian product of the 1-simplex.
This adjunction is in fact a Quillen adjunction, but the triangulations $TI^n_{/H} \in \sSet$ are all contractible; for example, the quotient $I^2_{/\Sigma_2}$ is isomorphic to $\Delta^2$.
As the left adjoint of a Quillen equivalence reflects contractibility of cofibrant objects \cite[1.3.16]{Hovey:1999}, this adjunction cannot be a Quillen equivalence if $I^n_{/H}$ is not contractible.
Of course, the model structure on $\cSet$ could present spaces without this particular adjunction being a Quillen equivalence.
However, it is worth noting that triangulation \emph{does} define a Quillen equivalence from the \emph{test} model structure on $\cSet$ to the Kan--Quillen model structure on $\sSet$ (see \S\ref{ssec:test}).

Second, let us give some intuition as to why $I^n_{/H}$ is not contractible in the model structure of \cite{CMS20,Awodey:2023}.
We recall the ``uniform unbiased box-filling'' characterization of its fibrations alluded to above.
Briefly, a map $f \colon Y \twoheadrightarrow X$ is a fibration when it admits a choice of lifts
\[
  \begin{tikzcd}
    I^n \cup_{C} C \times I^1 \arrow[d,tail,"{\langle [\xi], c \times I^1 \rangle}" left] \ar[r] & Y \ar[->>,d,"f"] \\
    I^n \times I^1 \ar[r] \ar[ur,dashed] & X
  \end{tikzcd}
\]
for each lifting problem against an \emph{open box inclusion} (determined by a subobject $c \colon C \rightarrowtail I^n$ and generalized point $\xi \colon I^n \to I^1$) in such a way that for every morphism of cubes $\alpha \colon I^m \to I^n$, the resulting triangle
\[
  \begin{tikzcd}
    I^m \cup_{D} D \times I^1 \arrow[dr, phantom, "\lrcorner" very near start] \ar[d,tail] \ar[r] & I^n \cup_{C} C \times I^1 \arrow[d,tail] \ar[r] & Y \ar[->>,d,"f"] \\
    I^m \times I^1 \ar[urr,dashed] \ar[r, "\alpha \times I^1" below] & I^n \times I^1 \ar[r] \ar[ur,dashed] & X
  \end{tikzcd}
\]
formed by the two chosen lifts commutes.

In the language of \emph{algebraic weak factorization systems (awfs)}, the class of fibrations is generated by the category of open box inclusions and pullback squares between them.
In fact, this open box category generates \emph{categories} of \emph{trivial cofibration coalgebras} and \emph{fibration algebras}, which by Garner's algebraic small object argument \cite{Garner:USOA} constitute an awfs. There is then an underlying weak factorization system whose left and right maps are the retracts of maps admitting trivial cofibration coalgebra and fibration algebra structures respectively.\footnote{Since the awfs under consideration is cofibrantly generated by a category, the monad algebras are already retract closed.
  Equivalently, the left and right maps are those admitting coalgebra structures for the copointed endofunctor underlying the comonad and algebra structures for the pointed endofunctor underlying the monad, respectively.}
The forgetful functor sending a trivial cofibration coalgebra to its underlying map creates colimits, and
the open box category embeds in the category of coalgebras; thus, in particular, the colimit in $\cSet^\2$ of any diagram of open box inclusions and pullback squares is a trivial cofibration.

With this definition, it is immediate that the 1-cube is contractible: the endpoint $0 \colon 1 \rightarrowtail I$ is the open box formed by the subobject $\emptyset \rightarrowtail I^0$ and point $0 \colon I^0 \to I^1$, thus a trivial cofibration.
That the 2-cube is contractible is slightly less immediate: we can write $\vec{0} \colon 1 \to I^2$ as a composite of generating trivial cofibrations
\[
  \begin{tikzcd}[column sep=large]
    1 \ar[r,tcofarrow,"0" below] & I^1 \ar[r,tcofarrow,"I^1 \times 0" below]{I^1} & I^2
  \end{tikzcd}
\]
where the second map is the open box formed by $\emptyset \rightarrowtail I^1$ and the constant map $0 \colon I^1 \to I^1$.
We can continue inductively to see that $\vec{0} \colon 1 \to I^n$ is a trivial cofibration for all $n$, the composite of $n$ generating trivial cofibrations.
Observe, however, that this construction is inherently \emph{asymmetric}: we collapse a 2-cube by collapsing first along one axis and then along the other.
This prevents us, for example, from deriving a trivial cofibration coalgebra structure on $\vec{0} \colon 1 \to I^2_{/\Sigma_2}$ by taking a colimit: writing $\SSigma_2$ for the one-object groupoid corresponding to $\Sigma_2$, the diagram $\SSigma_2 \to \cSet^\2$ sending the object to $\vec{0} \colon 1 \cwto I^2$ and $\sigma \in \Sigma_2$ to
\[
  \begin{tikzcd}
    1 \ar[d,tcofarrow,"\vec{0}" left] \ar[r] & 1 \ar[d,tcofarrow,"\vec{0}" left] \\
    I^2 \ar[r,"\sigma"] & I^2
  \end{tikzcd}
\]
does \emph{not} lift to a diagram of trivial cofibration coalgebras.
In fact, one can show that if $A \cwto B$ is a trivial cofibration and $B$ contains a non-trivial (in an appropriate sense) copy of $I^2_{/\Sigma_2}$, then so does $A$ \cite[\S4]{Coquand:2018tcf}: trivial cofibrations cannot collapse copies of $I^2_{/\Sigma_2}$.
It follows that $I^2_{/\Sigma_2}$ is not contractible \cite[\S5]{Coquand:2018tcf}; the same argument applies to quotients of higher cubes.

\subsubsection{The solution} \label{sssec:equivariant-solution}

Our solution to this problem is to require a more general \emph{equivariant} uniform box-filling structure on our fibrations.
First, we generalize the open box inclusions, replacing generalized points $\xi \colon I^n \to I^1$ on the 1-cube with points $\xi \colon I^n \to I^k$ in arbitrary $k$-cubes, so that we ask for lifts
\[
  \begin{tikzcd}
    I^n \cup_{C} C \times I^k \arrow[d,tail,"{\langle [\xi], c \times I^k \rangle}" left] \ar[r] & Y \ar[->>,d,"f"] \\
    I^n \times I^k \ar[r] \ar[ur,dashed] & X \rlap{.}
  \end{tikzcd}
\]
This generalization alone does not change the class of fibrations.
The key is in our generalization of the uniformity condition: for every morphism of cubes $\alpha \colon I^m \to I^n$ \emph{and automorphism $\sigma \colon I^k \cong I^k$}, the resulting triangle of lifts
\[
  \begin{tikzcd}
    I^m \cup_{D} D \times I^k \arrow[dr, phantom, "\lrcorner" very near start] \ar[d,tail] \ar[r] & I^n \cup_{C} C \times I^k \arrow[d,tail] \ar[r] & Y \ar[->>,d,"f"] \\
    I^m \times I^k \ar[urr,dashed] \ar[r, "\alpha \times \sigma" below] & I^n \times I^k \ar[r] \ar[ur,dashed] & X
  \end{tikzcd}
\]
must commute.

With this definition, the vertex inclusion $\vec{0} \colon 1 \to I^n$ is immediately a trivial cofibration: it is the open box formed by $\emptyset \rightarrowtail 1$ and the point $\vec{0} \colon 1 \to I^n$.
Moreover, for any $H \subset \Sigma_n$, the diagram $\cH \to \cSet^2$ sending the object to $\vec{0} \colon 1 \cwto I^n$ and $\sigma \in H$ to
\[
  \begin{tikzcd}
    1 \ar[d,tcofarrow,"\vec{0}" left] \ar[r] & 1 \ar[d,tcofarrow,"\vec{0}" left] \\
    I^n \ar[r,"\sigma"] & I^n
  \end{tikzcd}
\]
now \emph{does} lift to a diagram of trivial cofibration coalgebras; its colimit exhibits the point $\vec{0} \colon 1 \to I^n_{/H}$ as a trivial cofibration, making $I^n_{/H}$ contractible.

These observations led us to a construction of the generating categories of cofibrations and trivial cofibrations for the equivariant model structure in Summer 2019.
While we felt confident that these categories were canonical---since we had arrived at their definition simultaneously through two different constructions, one category theoretic and one type theoretic---the corresponding model structure felt somewhat ad hoc, not fitting into known paradigms for constructions of model categorical models of homotopy type theory.
A few years later, we realized that the equivariant premodel structure could be transferred from a premodel structure on the category $\cSet^{\SSigma}$ of \emph{cubical species} (i.e., symmetric sequences of cubical sets), where there exists a canonical equivariant interval object $\II = (I^n)_{n \geq 1}$.
There the generating cofibrations and trivial cofibrations fit into a known paradigm where the latter are defined from the former using the generic point of the interval $\II$ (as in \cite{ABCFHL,CMS20,Awodey:2023}).\footnote{Interestingly, while the equivariant premodel structure is lifted along a right adjoint, the constant functor $\Delta \colon \cSet \to \cSet^{\SSigma}$, the model structure itself is not: the fibrations and trivial fibrations are created by $\Delta$, but the weak equivalences between non-fibrant objects are not.}

\subsection{Results}

Our results are summarized by the following theorem.

\begin{thm}
  \label{thm:main}
  There is a constructively definable model of HoTT in cartesian cubical sets with an associated constructively definable Quillen model structure that is classically Quillen equivalent to the Kan--Quillen model structure on simplicial sets.
\end{thm}

By \emph{associated Quillen model structure}, we mean as in \S\ref{ssec:cubical-model-structures} a model structure whose fibrations are the retracts of context extensions of the model of HoTT and whose trivial fibrations are the retracts of context extensions by contractible types.

By a \emph{model of HoTT} we mean a model of Martin-L\"of type theory validating the univalence axiom, and by \emph{model of Martin-L\"of type theory} we mean a natural model \cite{Awodey:2018nm} equipped with $\Pi$-types, $\Sigma$-types, identity types, and universes closed under these.
More precisely, what we construct is a \emph{natural pseudo-model} in the sense of Shulman \cite[{\S}A]{Shulman:2019ai} with weakly stable equivalents of this structure (a weakly stable class of $\Pi$-types, etc.); one can then apply Lumsdaine and Warren's \emph{left adjoint splitting} coherence construction \cites{LumsdaineWarren:2015}{Awodey:2018nm}[{\S}A]{Shulman:2019ai} to obtain a natural model with strictly stable structure.
Concretely, our category of contexts is the category $\cSet$ of cartesian cubical sets, and the natural pseudo-model specifying the types and terms is the \emph{notion of fibred structure} encoding the equivariant fibrations (Lemma \ref{lem:equiv-loc-rep-fib}).
The interpretation of type formers is as follows.

\begin{itemize}
\item Weakly stable $\Sigma$-types and identity types arise immediately from the model structure (see, e.g., \cite[\S4.2]{LumsdaineWarren:2015}). $\Sigma$-types are interpreted by composition of fibration algebras, while the identity type on $A \fto \Gamma$ is interpreted by the (trivial cofibration, fibration) factorization of its diagonal $A \to A \times_\Gamma A$ (as in \cite{awodey-warren:2009}).
\item Weakly stable $\Pi$-types come from the \emph{Frobenius condition} \cites[3.3.3]{BergGarner:2012}{GambinoSattler:2017fc}, that is the closure of fibrations under pushforward along fibrations, which is verified in Proposition \ref{prop:equivariant-frobenius}.
\item Universes are interpreted by classifiers for the notions of fibred structure encoding $\kappa$-small equivariant fibrations for sufficiently large inaccessible cardinals $\kappa$ (Proposition \ref{prop:equivariant-has-universes}).
  Importantly, these classifiers have fibrant base objects (Proposition \ref{prop:equivariant-fibrant-universe}) and are univalent (Proposition \ref{prop:equivariant-univalent}).
  The former property is closely connected to the model-categorical \emph{fibration extension property} (Proposition \ref{prop:equivariant-FEP}), the latter to the \emph{equivalence extension property} (Proposition \ref{prop:equivariant-EEP}).
\end{itemize}
The main technical work lies in the construction of univalent universes.

In the course of proving the main theorem, we actually construct \emph{two} models of homotopy type theory and associated Quillen model structures: a model on the category $\cSet^{\SSigma}$ of cubical species, which does not model classical homotopy theory, and a model on $\cSet$, which does.
To avoid repetition and with an eye towards future applications, we prove the core theorems that will establish the necessary properties of these model categories in more general axiomatic settings, proving results that are of independent interest.

\subsubsection{Outline}

Our development proceeds as follows.

\begin{itemize}
\item In \S\ref{sec:fibred}, we recall Shulman's \emph{notions of fibred structure}, which in particular include categories of right maps obtained from an algebraic weak factorization system.
  Again following Shulman, we define a \emph{universe} for a notion of fibred structure to be a representable ``resolution'' via an acyclic fibration.
  We define our first example of a notion of fibred structure, the \emph{uniform trivial fibrations}, following \cite{Awodey:2023}.

\item In \S\ref{sec:cylindrical}, we work in the abstract setting of a \emph{cylindrical premodel structure} as defined in \cites{Sattler:2020cms}[\S3]{CavalloSattler:2022re}.
  We establish, individually, sufficient conditions under which a cylindrical premodel structure
  \begin{itemize}
  \item satisfies the equivalence extension property;
  \item satisfies the Frobenius condition,
  \item supports fibrant and univalent universes of fibrations;
  \item defines a Quillen model structure.
  \end{itemize}
  These constructions form the backbone of existing model-categorical cubical interpretations and could be applied with appropriate inputs from  \cite{ABCFHL} or \cite{Awodey:2023} to recover the known model structures on, e.g., cartesian or De Morgan cubical sets.
  In the following sections, we apply them to two cylindrical premodel structures: first to $\cSet^{\SSigma}$ and then to $\cSet$ itself.
  As a rule of thumb, properties whose proofs rely only on \emph{closure} properties of fibrations (such as the equivalence extension property) are derived directly in $\cSet$, while properties whose proofs rely on the \emph{generation} of fibrations by box filling (such as the Frobenius condition) are first proven in $\cSet^{\SSigma}$ and then transferred to $\cSet$.

\item In \S\ref{sec:symmetric}, we introduce the category $\cSet^{\SSigma}$ of cubical species.
  We define the \emph{symmetric interval} $\II \in \cSet^{\SSigma}$ and use it to define, by essentially the same construction used for the ordinary cartesian cubical set model \cite{ABCFHL,CMS20,Awodey:2023}, a model of HoTT and Quillen model structure on $\cSet^{\SSigma}$.

\item In \S\ref{sec:equivariant}, we transfer the cylindrical premodel structure on $\cSet^{\Sigma}$ to $\cSet$ by means of the constant functor $\Delta \colon \cSet \to \cSet^{\SSigma}$, defining the equivariant (trivial) fibrations to be those sent to (trivial) fibrations in $\cSet^{\SSigma}$ by $\Delta$.
  We show that this premodel structure satisfies 2-of-3, proving the first part of Theorem \ref{thm:main}: the existence of a constructively definable model of HoTT and associated Quillen model structure on $\cSet$ whose fibrations are the equivariant fibrations.
\item In \S\ref{sec:classical}, we prove the second part of Theorem \ref{thm:main}, building a Quillen equivalence between the equivariant model structure on $\cSet$ and the Kan--Quillen model structure on $\sSet$.
  The left adjoint of this equivalence is the triangulation functor $T \colon \cSet \to \sSet$ mentioned above; we rely on a characterization, due to Reid Barton, of $T$ as restriction along a functor $i \colon \DDelta \to \CCube$.
  Key to the proof is that $\DDelta$ and $\CCube$ are \emph{Eilenberg--Zilber categories}, which implies that the monomorphisms in their respective presheaf categories are cell complexes of quotients by automorphism groups of boundary inclusions into representables.
  In this way, the fact that $T$ reflects weak equivalences comes to rest on the contractibility of the quotients $I^n_{/H} \in \cSet$, which we have seen in \S\ref{sssec:equivariant-solution} is ensured by the definition of equivariant fibration.
  Using the Quillen equivalence, we also prove that our model structure on $\cSet$ coincides with the test model structure.
\end{itemize}

Finally, we devote an appendix (\S\ref{app:type-theoretic}) to a second perspective on the construction of the equivariant model of HoTT in $\cSet$.
There we outline a translation of HoTT into the internal extensional type theory of cubical sets augmented by an axiomatisation of the interval and cofibration classifier, which is backed by a complete formalisation in the proof assistant Agda \cite{formalization} following Orton and Pitts \cite{OrtonPitts:2018af}.
This also demonstrates that, as usual for cubical models of HoTT, a coherence construction is not actually needed to obtain a model of Martin-Löf type theory: our types are sufficiently structured to directly interpret strictly stable structure (without the need to, e.g., choose lifts).

\subsubsection{Constructivity}
\label{ssec:constructivity}

Part of the aim of this paper is to describe a \emph{constructive} model of HoTT.
Thus \S\S\ref{sec:fibred}--\ref{sec:equivariant}, which culminate in the construction of the equivariant model of HoTT and Quillen model structure on $\cSet$, can be made completely constructive.
Note, however, that one must replace the use of monomorphisms in presheaf categories everywhere with \emph{levelwise decidable monomorphisms}, that is, $m \in (\Set^{\cC^\op})^\2$ such that $m_c$ is isomorphic to a coproduct coprojection for all $c \in \cC$. Constructively, the Hofmann--Streicher classifiers form universes in the sense of \S\ref{ssec:realignment} only with this modification, as used in the CCHM model \cite[15]{CCHM:2018ctt} and observed explicitly by Orton and Pitts \cite[8.4]{OrtonPitts:2018af}.
Note that this replaces the subobject classifier with a classifier for subobjects whose corresponding sieve is decidable.
This also has the effect of making the development \emph{predicative}.

We justify constructivity of our use of Garner's algebraic small object argument in Constructions \ref{con:generating-cof}, \ref{con:generating-tcof}, \ref{con:equivariant-generating-cof}, and \ref{con:equivariant-generating-tcof}.
Note that the set of morphisms of the cartesian cube category $\CCube$ has decidable equality.
The Eilenberg--Zilber category structure on $\CCube$ is constructively definable and has finite slices of face maps.
By induction, we show that any object $S$ with a levelwise decidable inclusion into $\yo a$ with $a \in \CCube$ is compact: if this map contains the generic element of $\yo a$, then then $S \simeq \yo a$ is representable; otherwise, it is a finite cell complex of boundary inclusions of degree lower than $a$.
Note that any double left adjoint preserves compact objects.
The generating (trivial) cofibrations in cubical species are levelwise decidable subobjects of representables.
By Lemma \ref{lem:symmetric-subobjects}, these arise via left Kan extension from levelwise decidable subobjects of representables in cubical sets.
Then the generating (trivial) cofibrations in cubical sets are created from these via the double left adjoint $L$ in \S\ref{ssec:species-to-sets}.
Therefore, all the generating categories for each of our uses of the algebraic small object argument consist of maps with compact domain.
This makes the pointed endofunctor for the one-step factorization preserve filtered colimits.
Hence, its free monad sequence converges at stage $\omega$.
For more details on the algebraic small object argument in a constructive context, we refer to Henry \cite[{\S}C.2]{Henry:2019}.

The Kan--Quillen model structure on simplicial sets is definable constructively~\cite{Henry:2019,GambinoSattlerSzumilo:2022}.
However, our proof of the Quillen equivalence between the equivariant model structure and the Kan--Quillen model structure in \S\ref{sec:classical} is not constructive.
In its heart, it relies on the non-constructive presentation of (levelwise decidable) monomorphisms in $\cSet$ as Reedy cell complexes.
Constructively, only the \emph{Reedy decidable} monomorphisms---that is, those whose latching maps are decidable---can be presented in this way (compare \cite[\S1.4]{GambinoSattlerSzumilo:2022}).
While it may be possible to define an analogue of the equivariant fibration model structure with Reedy decidable monomorphisms as cofibrations (as in \cite{Henry:2019,GambinoSattlerSzumilo:2022} for simplicial sets), this choice interferes with coherence constructions used to interpret HoTT, as Gambino and Henry find \cite[8.5]{GambinoHenry:2022}.

It is more generally unclear how to judge whether a homotopy theory is ``the homotopy theory of spaces'' constructively.
Indeed, it is likely that there are multiple constructively distinct homotopy theories that are all classically equivalent to the Kan--Quillen model structure.
Shulman does some preliminary analysis of this question through the lens of \emph{derivators} \cite{Shulman:2023ds}.

\subsection{Related and future work}

\subsubsection{Models of HoTT in higher toposes}

Shulman has shown that every Grothendieck $\infty$-topos admits a presentation by a
\emph{type-theoretic model topos}, a Quillen model category with structure sufficient to interpret HoTT \cite{Shulman:2019ai}.
His setup uses classical logic and is inherently simplicial: a type-theoretic model topos is by definition a simplicial model category.
As far as we know, our model structure falls outside of this framework: we are not aware of any appropriate simplicial enrichment on $\cSet$.
The natural candidate definition, taking the mapping space between $X,Y \in \cSet$ to be the triangulated internal hom $T[X,Y] \in \sSet$, does not yield a simplicial model structure, essentially because the left adjoint to $T$ (constructed in \S\ref{sec:classical}) does not preserve products.

Shulman uses the description of Grothendieck $\infty$-toposes as left exact localizations of presheaf $\infty$-toposes, showing that the class of type-theoretic model toposes is closed under model-categorical constructions presenting left exact localizations and categories of presheaves.
The base case is then the $\infty$-topos of spaces: the category of simplicial sets with the Kan--Quillen model structure is a type-theoretic model topos.
Our work suggests a future path towards constructivizing (at least some of) Shulman's results, namely by developing a \emph{cubical} notion of type-theoretic model topos and using the equivariant model structure on $\cSet$ as the base model of spaces.

\subsubsection{Other cubical models}

In 2022, the second and fifth author discovered a second cubical interpretation of type theory whose associated Quillen model structure presents spaces, this one in presheaves on the category of \emph{cartesian cubes with one connection} \cite{CavalloSattler:2022re}.
In this setting, it is not necessary to introduce the notion of equivariant fibration: applying the original cartesian cubical model construction as in \cite{ABCFHL,CMS20,Awodey:2023} yields a Quillen model structure presenting spaces.
This can be explained by the fact that any fibration in presheaves over this cube category is \emph{automatically} equivariant, as sketched in \cite[4.25]{CavalloSattler:2022re} (compare our Proposition \ref{prop:kan-is-equivariant}).
A downside of this model is that the cube category with one connection is less well-behaved: while the cartesian cube category is an Eilenberg--Zilber category \cite[8.12(1)]{Campion:2023ez}, the cube category with one connection is not a Reedy category \cite[{\S}A.1]{CavalloSattler:2022re}.
The main task of \cite{CavalloSattler:2022re} is to develop a generalization of Eilenberg--Zilber category which can be used in this case.

Equivariance is not a catch-all solution: it is not the case that we can take any of the existing cubical interpretations and impose an equivariance condition on fibrations to obtain a model for spaces.
For example, as in the one connection case, fibrations in Dedekind cubical sets (i.e., over the cartesian cube category with two connections) are automatically equivariant, but we still do not know if this model presents spaces, essentially because this cube category is even farther from being an Eilenberg--Zilber category than the one-connection category \cite[{\S}A.2]{CavalloSattler:2022re}.
Over the BCH cube category, which \emph{is} an Eilenberg--Zilber category \cite[7.10]{Campion:2023ez}, adding equivariance would have the effect of making the cube quotients $I^n_{/H}$ contractible, as it does in cartesian cubical sets. This is desirable from the point of view of triangulation. Note however that in the test model structure on BCH cubical sets, the quotient $I^2_{/\Sigma_2}$ is \emph{not} contractible but rather presents the suspension of $\RR P^\infty$.

\subsubsection{Constructive simplicial models}
\label{sssec:constructive-simplicial}

The Kan--Quillen model structure on simplicial sets has been developed constructively by Henry \cite{Henry:2019} and Gambino, Sattler, and Szumi{\l}o \cite{GambinoSattlerSzumilo:2022}.
However, Bezem, Coquand, and Parmann show that Voevodsky's model in simplicial sets relies in an essential way on classical principles \cite{BezemCoquand:2015,BezemCoquandParmann:2015,Parmann:2018}.
Essentially, this is because the Kan--Quillen model structure cannot generally be shown to have cofibrant objects; indeed, the cofibrant objects are the \emph{Reedy decidable objects}, those for which we can decide if a cell is degenerate.
In particular, the interpretation of $\Pi$-types is problematic: constructively, the exponential $Y^X$ need not be a Kan complex even if $X$ and $Y$ are---cofibrancy of $X$ is required.

As mentioned in \S\ref{ssec:constructivity}, Gambino and Henry \cite{GambinoHenry:2022} give a constructive reformulation of Voevodsky's simplicial model of HoTT in \emph{cofibrant simplicial sets}.
However, the restriction to cofibrant objects interferes with the coherence constructions needed to obtain a strict model, meaning that the end result falls short of its classical equivalent.

Van den Berg and Faber \cite{BergFaber:2022ek} present a second approach to constructivizing Voevodsky's model replacing Kan fibrations with a restricted notion of \emph{effective Kan fibration}.
As in our own work, the idea is to impose additional uniformity conditions on lifts.
Although this approach does not require restricting cofibrations to Reedy decidable monomorphisms and thus may avoid the coherence issues of \cite{GambinoHenry:2022}, it is still work in progress: to our knowledge, neither an interpretation of universes nor a Quillen model structure have been established thus far.

\subsubsection{Cubical type theories}

Cohen, Coquand, Huber, and M\"{o}rtberg \cite{CCHM:2018ctt} present not only a model of homotopy type theory but also a \emph{cubical type theory}, an extension of Martin-L\"of type theory with new judgments and type formers that reflect the structure of the De Morgan cubical sets model.
Angiuli et al.~\cite{AngiuliFavoniaHarper:2018,ABCFHL} likewise devise a cubical type theory interpreting in cartesian cubical sets.
Unlike HoTT as formulated in \cite{UF:2013}, these theories enjoy canonicity: any closed natural number computes definitionally to a numeral \cite{AngiuliFavoniaHarper:2018,Huber:2019}.

The cartesian cubical type theory of \cite{ABCFHL} can also be interpreted in the equivariant cartesian model: every equivariant fibration is in particular a fibration in the sense of the original cartesian cubical set model, so interprets the filling operator (sometimes also called the \emph{composition} operator) of cartesian cubical type theory \cite[\S1.2]{ABCFHL}.
Thus, cartesian cubical type theory has a model presenting the homotopy theory of spaces.

One could imagine extending cartesian cubical type theory with an equivariant filling operator.
Such an operator could be introduced by the rule
\begin{mathpar}
  \inferrule
  {k \in \NN \\
    \Gamma, \vec{\imath} : I^k \vdash A\ \mathrm{type} \\
    \Gamma \vdash \phi\ \mathrm{cof} \\
    \Gamma \vdash \vec{r},\vec{s} : I^k \\\\
    \Gamma, \phi, \vec{\imath} : I^k \vdash u : A \\
    \Gamma \vdash u_0 : A[\vec{r}/\vec{\imath}] \\
    \Gamma, \phi \vdash u[\vec{r}/\vec{\imath}] = u_0 : A
    }
  {\Gamma \vdash \mathsf{comp}_{\vec{\imath}.A}^{\vec{r}\rightarrow \vec{s}}~[\phi \mapsto \vec{\imath}.u]~u_0 : A[\vec{r}/\vec{\imath}]}
\end{mathpar}
which straightforwardly generalizes the ordinary filling operator by replacing the interval $I$ with an arbitrary $k$-cube $I^k$, together with the usual equations
\begin{align*}
  \mathsf{comp}_{\vec{\imath}.A}^{\vec{r}\rightarrow \vec{s}}~[\phi \mapsto \vec{\imath}.u]~u_0 &= u[\vec{s}/\vec{\imath}] &&\text{when $\phi$ holds} \\
  \mathsf{comp}_{\vec{\imath}.A}^{\vec{r}\rightarrow \vec{s}}~[\phi \mapsto \vec{\imath}.u]~u_0 &= u_0 &&\text{when $\vec{r} = \vec{s}$}
\end{align*}
which specify that the output of $\mathsf{comp}$ is a filler for the input box.
\emph{Equivariance} states that, for each $\sigma \in \Sigma_k$, we have the equation
\begin{align*}
  \mathsf{comp}_{\vec{\imath}.A}^{\sigma^*\vec{r} \rightarrow \sigma^*\vec{s}}~[\phi \mapsto \vec{\imath}.u]~u_0
  = \mathsf{comp}_{\vec{\jmath}.A[\sigma^*\vec{\jmath}/\vec{\imath}]}^{\vec{r}\rightarrow \vec{s}}~[\phi \mapsto \vec{\jmath}.u[\sigma^*\vec{\jmath}/\vec{\imath}]]~u_0 \rlap{,}
\end{align*}
where $\sigma^*$ is the action of $\sigma$ on $k$-tuples of terms in $I$.

We are, however, not aware of any practical use for the equivariant filling operator in cubical type theory.
Synthetic homotopy theorists working in cubical type theories have yet to encounter any fundamental difference in expressivity between, e.g., cartesian and De Morgan cubical type theories, or even between cubical type theories and HoTT \`a la \cite{UF:2013}, and the situation seems to be the same here.
It would also be expensive and complicated to type-check equivariant filling operators: to compare two $k$-dimensional $\mathsf{comp}$ terms for equality requires testing whether they agree modulo any of the $k!$ permutations.

\subsection{Acknowledgments}

The discovery of the equivariant model occurred at the Centre for Advanced Study (CAS) at the Norwegian Academy of Science and Letters in Oslo, Norway, in the academic year 2018--19 research project on Homotopy Type Theory and Univalent Foundations organized by Marc Bezem and Bj{\o}rn Dundas.  We gratefully acknowledge their support.  The first and third authors are also grateful to the
Institut des Hautes \'Etudes Scientifiques for hosting two weeks of very nice discussions in June 2022.

The perspective of the generating categories of cofibrations and trivial cofibrations as internally indexed by cubical species (see \S\ref{ssec:species-cylindrical}) was informed by discussions with Andrew Swan. Reid Barton's recent insights into the triangulation functor enabled us to considerably simplify the proofs of the results in \S\ref{sec:classical}.

The first and fourth author were supported by the US Air Force Office of Scientific Research under award number FA9550-21-1-0009 as well as, for the first author, award number FA9550-20-1-0305. The second author was supported by the US Air Force Office of Scientific Research under award number FA9550-19-1-0216 and by the Knut and Alice Wallenberg Foundation (KAW) under grant numbers 2020.0266 and 2019.0116.  The third author was supported by the ForCUTT project, ERC advanced grant number 101053291.
The fourth author is also supported by US National Science Foundation via the grants DMS-2204304 and DMS-2507077 and by the President's Frontier Award at Johns Hopkins, which supported visits to the other authors.
The fifth author was supported by the Swedish Research Council under grant number 2019-03765 and the US Air Force Office of Scientific Research under award number FA9550-24-1-0302.

\section{Notions of fibred structure and universes}\label{sec:fibred}

A (model-categorical) model of HoTT comes with two classes of ``right'' maps: the \emph{fibrations}, which model type families, and the \emph{trivial fibrations}, which model contractible type families. A key feature of both classes of maps is their stability under pullbacks along arbitrary maps, which models substitution of terms for variables in type theory.

In this section, we consider such ``notions of fibred structure'' abstractly, proving general results that will apply to both the fibrations and the trivial fibrations in the model categories we construct.  In \S\ref{ssec:fibred-str}, we recall the precise, technical meaning of the phrase ``notion of fibred structure'' and explore what it means when such fibred structure is \emph{locally representable}. In \S\ref{ssec:trivial-fibrations}, we specialize to elementary toposes and show that suitably structured maps that lift against the monomorphisms define a locally representable notion of fibred structure. In \S\ref{ssec:realignment}, introduce our notion of universe and, in the case of presheaf toposes, construct universes for locally representable notions of fibred structure from the Hofmann--Streicher classifiers.

\subsection{Locally representable and relatively acyclic notions of fibred structure}\label{ssec:fibred-str}

The maps in a 1-category $\cE$ with pullbacks assemble into a contravariant groupoid-valued pseudofunctor on $\cE$ sending an object $X$ to the large groupoid of maps with codomain $X$. This pseudofunctor $\sE$ is referred to as the \textbf{core of self-indexing}---the ``self-indexing'' referring to the slice categories $\cE_{/X}$ and the ``core'' referring to their groupoid cores. In \cite[3.1]{Shulman:2019ai}, Shulman defines a \textbf{notion of fibred structure} on a category $\cE$ with pullbacks as a strict discrete fibration with small fibers $\psi \colon \sF \to \sE$ in the 2-category of contravariant groupoid-valued pseudofunctors on $\cE$ and pseudonatural transformations between them. Here, a \emph{strict discrete fibration} is a strictly natural transformation whose components are fibrations of groupoids.

Unpacking this, a notion of fibred structure is given by:
\begin{enumerate}
\item for each map $f \colon Y \to X$ of $\cE$, a set of ``fibration structures'',
\item for each pullback square
\begin{equation}\label{eq:F-morphism} \begin{tikzcd} W\arrow[d, "g^*f"'] \arrow[dr, phantom, "\lrcorner" very near start] \arrow[r, "f^*g"] & Y \arrow[d, "f"] \\ Z \arrow[r, "g"'] & X\rlap{,} \end{tikzcd}
\end{equation} a function from the set of fibration structures on $f$ to the set of fibration structures on $g^*f$ that is pseudofunctorial in pullback squares.
\end{enumerate}
See \cite[\S3]{Shulman:2019ai} for considerably more discussion. Following Shulman, we refer to the ``structured fibrations'' associated to a notion of fibred structure $\sF$ as $\sF$-\textbf{algebras} and then refer to a pullback square \eqref{eq:F-morphism} in which the $\sF$-algebra structure on $g^*f$ is induced from the $\sF$-algebra structure on $f$ as an $\sF$-\textbf{morphism}.

\begin{defn}[{\cite[3.2]{Shulman:2019ai}}]
  A notion of fibred structure $\psi \colon \sF \to \sE$ is \textbf{full} if $\sF(X) \to \sE(X)$ is fully faithful for each object $X$ of $\cE$.\footnote{Shulman's definition asks that $\psi \colon \sF \to \sE$ is a subfunctor inclusion; this is equivalent because $\psi$ is a discrete fibration.}
\end{defn}

That is, $\sF$ is full if every pullback square between $\sF$-algebras uniquely extends to an $\sF$-morphism.

Shulman then axiomatizes various conditions associated to such a notion of fibred structure that can be used to build a classifying universe. The first of these conditions is the following:
\begin{defn}[{\cite[3.10]{Shulman:2019ai}}]\label{defn:locally-representable} A notion of fibred structure $\sF$ is \textbf{locally representable} if each pullback in the category of contravariant groupoid-valued pseudofunctors
  \[ \begin{tikzcd} \bullet \arrow[d] \arrow[r] \arrow[dr, phantom, "\lrcorner" very near start] & \sF \arrow[d, "\psi"] \\ \cE(-,X) \arrow[r, "f"'] & \sE \end{tikzcd}
  \]
is representable.
Explicitly, every map $f \colon Y \to X$ has a \emph{classifier} $\psi_f \colon \sF(f) \to X$ for $\sF$-algebra structures on $f$, meaning that that for all $g \colon Z \to X$, $\sF$-algebra structures on $g^*f$ correspond bijectively to lifts of $g$ through $\psi_f$, naturally in $g$:
\[
\begin{tikzcd}  & \sF(f) \arrow[d, "\psi_f"]  \\ Z \arrow[ur, dashed] \arrow[r, "g"'] & X \rlap{.}
\end{tikzcd}
\]
\end{defn}
In particular, sections of the canonical map $\psi_f\colon\sF(f) \to X$ correspond uniquely to $\sF$-algebra structures on $f \colon Y \to X$.

\begin{lem}\label{lem:locally-representable-stability} Let $\sF$ be a locally representable notion of fibred structure.
\begin{enumerate}
\item\label{itm:loc-rep-alg} The pullback of any map $f \colon Y \to X$ along $\psi_f \colon \sF(f) \to X$ has a canonical $\sF$-algebra structure.
\item\label{itm:loc-rep-pb} If $g^*f$ is a pullback of $f$ along $g$, then $\sF(g^*f)$ is a pullback of $\sF(f)$ along $g$, i.e.\  $\sF(g^*f) \cong g^*\sF(f)$.
\end{enumerate}
\end{lem}
\begin{proof}
The top horizontal map in the pullback square
  \[ \begin{tikzcd} \cE(-,\sF(f)) \arrow[d, "\psi_f"'] \arrow[r, "\gamma_f"] \arrow[dr, phantom, "\lrcorner" very near start] & \sF \arrow[d, "\psi"] \\ \cE(-,X) \arrow[r, "f"'] & \sE \end{tikzcd}
  \]
  specifies an $\sF$-algebra structure $\gamma_f$ on the map $\psi_f^*f$.

  By pullback cancelation and fully faithfulness of the Yoneda embedding, local representability implies that the left-hand square is a pullback in contravariant groupoid-valued pseudofunctors and thus also in $\cE$:
  \[ \begin{tikzcd}[baseline=(current bounding box.south)] \cE(-,\sF(g^*f))\arrow[d, "\psi_{g^*f}"']\arrow[dr, phantom, "\lrcorner" very near start] \arrow[r, "i_g"] & \cE(-,\sF(f)) \arrow[d, "\psi_f"'] \arrow[r, "\gamma_f"] \arrow[dr, phantom, "\lrcorner" very near start] & \sF \arrow[d, "\psi"] \\ \cE(-,Z) \arrow[r, "g"'] & \cE(-,X) \arrow[r, "f"'] & \sE \rlap{.} \end{tikzcd} \qedhere
  \]
\end{proof}

\begin{rmk}\label{rmk:F-morphism} Recall that a pullback of $\sF$-algebras as in \eqref{itm:loc-rep-pb} is an $\sF$-morphism just when the $\sF$-algebra structure on $g^*f$ is created from the $\sF$-algebra structure on $f$. The naturality condition in Definition \ref{defn:locally-representable} tells us that this is the case just when the square defined by the corresponding sections of the representing morphisms commute:
\[ \begin{tikzcd}  \sF(g^*f) \arrow[d, "{\psi_{g^*f}}"'] \arrow[r, "i_g"] & \sF(f) \arrow[d, "\psi_f"] \\ \arrow[u, bend right, dashed, "s_{g^*f}"'] Z \arrow[r, "g"'] & X \arrow[u, bend left, dashed, "s_f"] \rlap{.}
  \end{tikzcd}
  \]
\end{rmk}

A large family of examples of locally representable notions of fibred structure are considered in \cite[\S 3]{Shulman:2019ai}. We mention just one, which will be applied in the following section.

\begin{ex}[{\cite[3.7,3.14]{Shulman:2019ai}}]\label{ex:loc-rep-from-cartesian-factorization} From a functorial factorization on $\cE$ one obtains a notion of fibred structure $\sF$ whose $\sF$-algebras are maps with chosen solutions to the canonical lifting problem against their left factor:
  \[ \begin{tikzcd} Y \arrow[r, equals] \arrow[d, "Lf"'] & Y \arrow[d, "f"] \\ Ef \arrow[ur, "j^f", dashed] \arrow[r, "Rf"'] & X \rlap{.} \end{tikzcd}\]
  If $\cE$ is locally cartesian closed and the functorial factorization is cartesian, in the sense that the functors $L,R \colon \cE^\2 \to \cE^\2$ carry pullback squares to pullback squares, then this notion of fibred structure is locally representable. Explicitly, $j^f$ may be encoded as an element in the internal hom $[Rf,f]_X \coloneq (Rf)_* (Rf)^* f$ from $Rf$ to $f$ in $\cE_{/X}$
  \[ \begin{tikzcd}[sep=small] X \arrow[dr, equals] \arrow[rr, dashed, "{j^f}"] & & [-1em] \Pi_{Ef} (Ef \times_X Y) \arrow[dl, "{[Rf,f]_X}"] \\ & X \end{tikzcd}
  \]
   which restricts along $Lf$ to the identity at $Y$. Thus, we define $\phi_f \colon \sF(f) \to X$ to be the pullback
  \[ \begin{tikzcd} \sF(f) \arrow[d, "\phi_f"'] \arrow[dr, phantom, "\lrcorner" very near start] \arrow[r] & \Pi_{Ef} (Ef \times_X Y) \arrow[d, "-\circ L_f"] \\ X \arrow[r, "\id_Y"] & \Pi_{Y} (Y \times_X Y) \end{tikzcd}\]
  of this restriction map.\footnote{The map $-\cdot L_f$ is the restriction between internal homs in the cartesian closed category $\cE_{/X}$. A construction of this map in $\cE$ may be found in \cite[3.9]{HazratpourRiehl:2022tc}.}
\end{ex}
\begin{defn}[{\cite[5.11]{Shulman:2019ai}}]\label{defn:relatively-acyclic} A notion of fibred structure $\sF$ is \textbf{relatively acyclic} if for any pullback square
 \[ \begin{tikzcd} Y' \arrow[d, "f'"'] \arrow[r, tail, "i'"] \arrow[dr, phantom, "\lrcorner" very near start] & Y \arrow[d, "f"] \\ X' \arrow[r, "i"', tail] & X
 \end{tikzcd}\]
with $\sF$-algebra structures $x$ on $f$ and $x'$ on $f'$, there is an $\sF$-algebra structure $\overline{x}$ on $f$ making the square an $\sF$-morphism from $x'$ to $\overline{x}$.
\end{defn}

Recall from~\cite[2.8]{Shulman:2019ai} the bicategorical notion of lifting property in a 2-category $\cK$: morphisms $i \colon A \to B$ and $f \colon Y \to X$ have the lifting property when the map
$\cK(B,Y) \to \cK(A,Y) \times^h_{\cK(A,X)} \cK(B,X)$ is essentially surjective, where $\times^h$ is a weak bicategorical pullback.

\begin{defn}[{\cite[5.1]{Shulman:2019ai}}]
A morphism in contravariant groupoid-valued pseudofunctors on $\cE$ is an \textbf{acyclic fibration} if it right lifts bicategorically against images of monomorphisms under the Yoneda embedding.
\end{defn}

\begin{rmk}
  For strict discrete fibrations in contravariant groupoid-valued pseudofunctors on $\cE$, the bicategorical right lifting property is equivalent to the categorical right lifting property~\cite[2.10]{Shulman:2019ai}.
  In particular, this applies to notions of fibred structure and their pullbacks.
\end{rmk}

\begin{lem}\label{lem:rel-acyclic} Given a notion of fibred structure $\psi \colon \sF \to \sE$, the following conditions are equivalent:
  \begin{enumerate}
    \item \label{lem:rel-acyclic:rel-acyclic} $\psi \colon \sF \to \sE$ is relatively acyclic,
    \item \label{lem:rel-acyclic:kernel} each kernel pair projection of $\psi$ is an acyclic fibration.
  \end{enumerate}
\end{lem}
\begin{proof}
  For a diagram
  \begin{equation}
    \label{eq:lem:rel-acyclic}
    \begin{tikzcd} Y' \arrow[d, "f'"'] \arrow[r, tail, "i'"] \arrow[dr, phantom, "\lrcorner" very near start] & Y \arrow[d, "f"] \\ X' \arrow[r, "i"', tail] & X \rlap{,}
    \end{tikzcd}
  \end{equation}
  a pair of $\sF$-algebra structures on $f$ and $f'$ consists of a pair of maps $x \colon \cE(-,X) \to \sF$ and $x' \colon \cE(-,X') \to \sF$ such that the outer square
  \[
    \begin{tikzcd}
      \cE(-,X') \arrow[rr, bend left=20, "x'"]\arrow[d, "i"'] \arrow[r, dashed] & \sF \times_{\sE} \sF \arrow[d, dotted] \arrow[r, dotted] \arrow[dr, phantom, "\lrcorner" very near start] & \sF \arrow[d, "\psi"] \\
      \cE(-,X) \arrow[r, "x"'] & \sF \arrow[r, "{\psi}"'] & \sE \rlap{,}
    \end{tikzcd}
  \]
  commutes, i.e., corresponds to a lifting problem against the kernel pair of $\psi$.
  A solution to such a lifting problem is determined by a map $\overline{x} \colon \cE(-,X) \to \sF$ such that $\overline{x} i = x'$ and $\psi \overline{x} = \psi x$, which is to say an $\sF$-algebra structure on $f$ such that \eqref{eq:lem:rel-acyclic} is an $\sF$-morphism from $x'$ to $\overline{x}$.
\end{proof}

\begin{lem}\label{lem:loc-rep-rel-acyclic-kernel-pair} When $\sF$ is a locally representable and relatively acyclic notion of fibred structure on $\cE$ then for any map $f \colon Y \to X$ the maps in the kernel pair of $\psi_f \colon \sF(f) \to X$ lift against monomorphisms in $\cE$.
\end{lem}
\begin{proof} Recall the definition of the maps in question:
  \[ \begin{tikzcd} \cE(-,\sF(f))\arrow[d, "\psi_f"'] \arrow[r] \arrow[dr, phantom, "\lrcorner" very near start] & \sF \arrow[d, "\psi"] \\ \cE(-,X) \arrow[r, "f"'] & \sE \rlap{.} \end{tikzcd}
  \]
As the kernel pair of a pullback is the pullback of the kernel pair, the kernel pair of the representable map $\psi_f \colon \cE(-,\sF(f)) \to \cE(-,X)$ lifts against representable monomorphisms. But since the Yoneda embedding is fully faithful and preserves limits, this means that the kernel pair of the map $\psi_f \colon \sF(f) \to X$ lifts against monomorphisms in $\cE$.
\end{proof}

By Lemma \ref{lem:rel-acyclic}, a full notion of fibred structure, such as the following example, is automatically relatively acyclic.

\begin{ex}[{\cite[6.1.6.4--7]{Lurie:2009}\cite[4.18]{Shulman:2019ai}}]\label{ex:small-fibred-structure}
For any locally presentable and locally cartesian closed category $\cE$, for sufficiently large regular cardinals $\kappa$, the relatively $\kappa$-presentable morphisms form a locally representable and relatively acyclic full notion of fibred structure $\sE^\kappa$.\footnote{In a presheaf topos $\cE=\Set^{\cC^\op}$ where $\cC$ is $\kappa$-small, the relatively $\kappa$-presentable morphisms coincide with the \textbf{$\kappa$-small morphisms}, those maps whose fibers have cardinality less that $\kappa$ \cite[4.10]{Shulman:2019ai}.}
\end{ex}

Locally representable notions of fibred structure may also be transferred from one category to another via various devices. Here we make use of a transfer result involving the \emph{Leibniz construction} of \cite[\S4--5]{RiehlVerity:2014rc}, deployed in the following setting.

\begin{defn}\label{defn:leibniz-application} Consider the application bifunctor
  \[ \begin{tikzcd}[row sep=small] \cE^\cD \times \cD  \arrow[r, "\circ"] & \cE \\ (F, X) \arrow[r, maps to] & FX \end{tikzcd}\]
associated to a pair of categories $\cD$ and $\cE$. Assuming $\cE$ has pushouts and pullbacks, this induces \textbf{Leibniz pushout application} and \textbf{Leibniz pullback application} bifunctors
\[ \begin{tikzcd} \cE^{\cD\times\2} \times \cD^\2 \arrow[r, "\check{\circ}"] & \cE^\2 &  \cE^{\cD\times\2} \times \cD^\2  \arrow[r, "\hato "] & \cE^\2
\end{tikzcd}
\]
which, respectively, send a natural transformation $\alpha \colon F \Rightarrow G$ and an arrow $f \colon Y \to X$  to the induced maps in the naturality squares:
\[ \begin{tikzcd}[sep=small] FY \arrow[ddd, "Ff"'] \arrow[rrr, "\alpha_Y"] \arrow[ddrr, phantom, "\ulcorner" very near end] & & & GY  \arrow[ddl, dotted] \arrow[ddd, "Gf"]&  & &FY \arrow[ddd, "Ff"'] \arrow[dr, dashed, "{\alpha\hato f}"]\arrow[rrr, "\alpha_Y"] & & & GY \arrow[ddd, "Gf"] \\ & & & &&  &  & \bullet  \arrow[ddrr, phantom, "\lrcorner" very near start] \arrow[urr, dotted] \arrow[ddl, dotted] \\ & & \bullet \arrow[dr, dashed, "{\alpha\check{\circ}f}"'] \\ FX \arrow[rrr, "\alpha_X"'] \arrow[urr, dotted] & & &  GX&  && FX \arrow[rrr, "\alpha_X"'] & & & GX \rlap{.}
\end{tikzcd}
\]
\end{defn}

\begin{lem}\label{lem:leibniz-application-transposition} Suppose $\cD$ and $\cE$ have weak factorization systems $(\Left,\Right)$ and $(\Mono,\Epi)$ respectively. Then the Leibniz pushout application of a natural transformation $\alpha \colon F \Rightarrow L$ between left adjoints preserves the left classes if and only if the Leibniz pullback application of the conjugate natural transformation $\alpha \colon R \Rightarrow U$ between the right adjoints preserves right classes.
\end{lem}

\begin{proof}
  Write $\Ladj(\cD,\cE) \subset \cE^\cD$ and $\Radj(\cE,\cD) \subset \cD^\cE$ for the full subcategories spanned by the left and right adjoint functors, respectively. Note we have an equivalence of categories $\Ladj(\cD,\cE)^\op \simeq \Radj(\cE,\cD)$ which exchanges left and right adjoints and conjugate transformations. Moreover, via this equivalence, the restricted application bifunctors
  \[ \begin{tikzcd} \Ladj(\cD,\cE) \times \cD \arrow[r, "\circ"] & \cE & \Radj(\cE,\cD) \times \cE \arrow[r, "\circ"] & \cD
  \end{tikzcd}
  \] are parametrized adjoints. Thus, by \cite[4.10, 4.11]{RiehlVerity:2014rc}, the Leibniz pushout application of left adjoints bifunctor and Leibniz pullback application of right adjoints bifunctor are parametrized adjoints, inducing a bijective correspondence between lifting problems:
  \[ \begin{tikzcd} F B \cup_{FA} LA \arrow[d, "{\alpha\check{\circ}\ell}"'] \arrow[r] & Y \arrow[d, "e"] \arrow[dr, phantom, "\leftrightsquigarrow"] & A \arrow[d, "\ell"'] \arrow[r] & R Y \arrow[d, "{\alpha{\hato}e}"] \\ LB \arrow[r] \arrow[ur, dashed]& X & B \arrow[r] \arrow[ur, dashed] & RX \times_{UX} UY
  \end{tikzcd}
  \] for $\ell \colon A \to B$ in $\Left$ and $e \colon Y \to X$ in $\Epi$. The claim follows.
\end{proof}

\begin{lem}\label{lem:leibniz-pullback-application-loc-rep-fibred-structure} Suppose $\cE$ and $\cE'$ have pullbacks, $\alpha : L\Rightarrow K : \cE' \to \cE$ is a natural transformation between pullback-preserving functors, and  $L$ has an indexed right adjoint:
  \[
  \begin{tikzcd}
  \cE' \arrow[rr, bend left, "L"] \arrow[rr, phantom, "\Downarrow\alpha"] \arrow[rr, bend right, "K"'] && \cE  \end{tikzcd}
  \qquad \qquad
  \begin{tikzcd} \cE'_{/X} \arrow[rr, bend left, "L_{/X}"] \arrow[rr, phantom, "\bot"] && \cE_{/LX} \arrow[ll, bend left, "R_X"] & [-3em] ,\ \forall X \in \cE'
  \end{tikzcd}
   \]
 Then if $\cE$ has a notion of fibred structure $\sF$, then $\cE'$ has a notion of fibred structure $\sF'$ in which $\sF'$-algebras are created from $\sF$-algebras under the Leibniz pullback application of $\alpha$. Moreover,
 \begin{enumerate}
  \item if $\sF$ is relatively acyclic, so is $\sF'$, and
  \item if $\cE$ is locally cartesian closed and $\sF$ is locally representable, so is $\sF'$.
 \end{enumerate}
\end{lem}
\begin{proof}
Since the functor $\alpha\hato{-} \colon (\cE')^\2 \to \cE^\2$ preserves pullbacks, $\sF'$ defines a notion of fibred structure on $\cE'$. Since $L$ and $K$ preserve pullbacks, they preserve monomorphisms, so the functor $\alpha\hato -$ preserves the monomorphisms in Definition \ref{defn:relatively-acyclic}, and thus if $\sF$ is relatively acyclic, so is $\sF'$.

It remains to verify local representability.
To that end, consider a pullback in $\cE'$
\[
\begin{tikzcd}
  W
  \ar[d, "g^*f"']
  \ar[r, "f^*g"]
  \ar[dr, phantom, "\lrcorner" very near start]
&
  Y
  \ar[d, "f"]
\\
  Z \ar[r, "g"']
&
  X
\end{tikzcd}
\]
inducing a pullback in $\cE$ as below-left:
\[
\begin{tikzcd} LW \arrow[d, "\alpha\hato g^*f"'] \arrow[dr, phantom, "\lrcorner" very near start] \arrow[r, "Lf^*g"] & LY \arrow[d, "\alpha\hato f"] & \sF(\alpha\hato g^*f) \arrow[d, "\phi_{\alpha\hato g^*f}"'] \arrow[r] \arrow[dr, phantom, "\lrcorner" very near start] & \sF(\alpha\hato f) \arrow[d, "\phi_{\alpha\hato f}"] \\
KW \times_{KZ} LZ \arrow[r, "{Kf^*g \times_{Kg} Lg}"'] & KY \times_{KX} LX & KW \times_{KZ} LZ \arrow[u, bend right, dotted] \arrow[ur, dashed] \arrow[r, "{Kf^*g \times_{Kg} Lg}"'] & KY \times_{KX} LX \rlap{.}
\end{tikzcd}
\]
By definition $\sF'$-algebra structures on $g^*f$ correspond to $\sF$-algebra structures on $\alpha\hato g^*f$.
Since $\sF$ is locally representable, these correspond to sections and thus lifts in the pullback square above-right constructed in Lemma \ref{lem:locally-representable-stability}.
Transposing across the pullback $\dashv$ pushforward adjunction associated to the projection $\alpha_X^*Kf \colon KY \times_{KX} LX \to LX$, such dashed lifts correspond bijectively to lifts as below-left
\[
\begin{tikzcd} & \Pi \sF(\alpha\hato f) \arrow[d, "(\alpha_X^*Kf)_*\phi_{\alpha\hato f}"] & & \sF'(g^*f) \arrow[d, "\psi_{g^*f}"'] \arrow[r] \arrow[dr, phantom, "\lrcorner" very near start] & R_X \Pi \sF(\alpha \hato f) \arrow[d, "R_X (\alpha^*_XKf)_* \phi_{\alpha\hato f}"] \\ LZ \arrow[r, "Lg"'] \arrow[ur, dashed] & LX & & Z \arrow[ur, dashed] \arrow[u, bend right, dotted]\arrow[r, "g"'] & X \rlap{,}
\end{tikzcd}
\]
and since $L$ has an indexed right adjoint $R_X$ \cite[B1.2.3]{Johnstone:2002}, such dashed lifts correspond bijectively to dashed lifts as above right. By the universal property of the pullback, we can thus define $\psi_{g^*f} \colon \sF'(g^*f) \to Z$ as the pullback displayed above-right.
\end{proof}

  \begin{ex}\label{ex:local-rep-transfer} For instance, $L \colon \cE' \to \cE$ might have an ordinary right adjoint and, supposing $\cE$ has a terminal object, $K \colon \cE' \to \cE$ may be taken to be the terminal functor. In this setting, Leibniz pullback application reduces to application of $L$ and Lemma \ref{lem:leibniz-pullback-application-loc-rep-fibred-structure} specializes to Shulman's observation that locally representable notions of fibred structure may be lifted along pullback-preserving left adjoints \cite[3.5, 3.12]{Shulman:2019ai}, though for that result $\cE$ needs only to have pullbacks and need not be locally cartesian closed.

  For instance, for any object $X \in \cE$, such an adjunction is given by the pullback functor along $X \to 1$:
      \[
        \begin{tikzcd} \cE_{/X} \arrow[rr, bend left, "U"] \arrow[rr, phantom, "\bot"] && \cE \rlap{.} \arrow[ll, bend left, "X^*"]
        \end{tikzcd}
        \]
    Thus a locally representable notion of fibred structure on $\cE$ may be lifted to its slice categories.
    \end{ex}

\subsection{Monomorphisms and uniform trivial fibrations}\label{ssec:trivial-fibrations}

Let $\cE$ be an elementary topos and write  $\top \colon 1 \to \Omega$  for its subobject classifier.
We consider a class of ``trivial fibrations'' characterized by the right lifting property against the monomorphisms and show that it underlies a notion of fibred structure which we call \emph{uniform trivial fibration} structure.
We then show that this notion of fibred structure is locally representable.

First, since elementary toposes are in particular locally cartesian closed, every map $f \colon X \to Y$ in $\cE$ induces an adjoint triple of functors
\[
\begin{tikzcd}[sep=large]
\cE_{/X} \arrow[r, bend left=45, "f_!", "\bot"'{yshift=-3pt}] \arrow[r, bend right=45, "f_*"', "\bot"{yshift=3pt}] & \arrow[l, "f^*" description] \cE_{/Y}
\end{tikzcd}
\]
where $f_!$ is post-composition, $f^*$ is pullback, and $f_*$ is (by definition) pushforward.
Furthermore, the following applies to $\cE$:

\begin{lem}\label{lem:mono-pushforward}
  In a locally cartesian closed category, the pullback-pushforward adjunction $i^* \dashv i_*$ along a monomorphism $i$ forms a reflective embedding.
  \end{lem}
  \begin{proof}
  The counit of $i^* \dashv i_*$ is an isomorphism just when its conjugate, the unit of $i_! \dashv i^*$, is an isomorphism, but the latter is clear, since the pullback of $i$ along itself is an isomorphism.
\end{proof}

We note the following closure property of monomorphisms in a topos, for later use:

\begin{rmk}\label{rmk:adhesive-pushout-products} Since elementary toposes are \emph{adhesive}, the class of monomorphisms is closed under pushout products, and the same is true in slice categories: given a pair of monomorphisms $i \colon A \rightarrowtail B$ and $j \colon C \rightarrowtail D$, the pushout product is the join of the subobjects $i \times D \colon A \times D \rightarrowtail B \times D$ and $B \times j \colon B \times C \rightarrowtail B \times D$ \cite[17]{LackSobocinski:2004ac}.
\end{rmk}

We now use the subobject classifier to define partial map classifiers (called partial-map representers in \cite[\S{A.2.4}]{Johnstone:2002}).
In turn, these will be used to define our trivial fibrations.
The following two propositions are proven in \cite[\S{3}]{Awodey:2023} (see also \cite[A2.4.7]{Johnstone:2002} and \cite[9.8--9]{GambinoSattler:2017fc}):

\begin{prop}\label{prop:partial-map-classifier}
For any $Y \in \cE$, there is a pullback square as below-left with the property that any partial map as below-right
  \[
  \begin{tikzcd} Y \arrow[d, tail, "\eta_Y"'] \arrow[r, "!"] \arrow[dr, phantom, "\lrcorner" very near start] & 1 \arrow[d, tail, "\top"] & &  C \arrow[d, tail, "c"'] \arrow[r, "y"] & Y \\ Y^+ \arrow[r, "\top_*Y"'] & {\Omega} & & Z
  \end{tikzcd}
\]
  is classified by a unique map $\zeta_c^y \colon Z\to Y^+$ defining a pullback square
  \[
  \begin{tikzcd} C \arrow[d, tail, "c"'] \arrow[dr, phantom, "\lrcorner" very near start] \arrow[r, "y"] & Y \arrow[d, tail, "\eta_Y"'] \arrow[r, "!"] \arrow[dr, phantom, "\lrcorner" very near start] & 1 \arrow[d, tail, "\top"] \\ Z \arrow[r, dashed, "\zeta_c^y"] \arrow[rr, bend right, "\chi_c"'] & Y^+ \arrow[r, "\top_*Y"] & \Omega \rlap{.}
  \end{tikzcd}
  \]
  Moreover, for any $X \in \cE$, the same results are true in $\cE_{/X}$, and these classifying squares are stable under pullback. \qed
\end{prop}

We refer to the monomorphism $\eta_Y \colon Y \rightarrowtail Y^+$ as the \emph{partial map classifier for $Y$}, since partial maps from $Z$ to $Y$ are classified by (total) maps $Z \to Y^+$.
We write $f^+ \colon Y^{+_X} \to X$ for the codomain of the partial map classifier for $(Y,f) \in \cE_{/X}$, so that we have $\eta_f \colon Y \to Y^{+_X}$.

\begin{defn}
  \label{defn:relative-+-algebra}
  A \textbf{relative $+$-algebra} structure on $f \colon Y \to X$ is a retraction over $X$ to the map $\eta_f \colon Y \rightarrowtail Y^{+_X}$ over $X$:
  \begin{equation}\label{eq:relative-plus-algebra}
    \begin{tikzcd} Y \arrow[d, tail, "\eta_f"'] \arrow[r, equals] & Y \arrow[d, "f"] \\ Y^{+_X} \arrow[ur, dashed, "\rho_f"] \arrow[r, "f^+"'] & X \rlap{.}
    \end{tikzcd}
  \end{equation}
  The \textbf{category of relative $+$-algebras} has relative $+$-algebras as objects and, as morphisms $f' \to f$, squares as below-left such that the induced diagram below-right commutes:
  \[ \begin{tikzcd} Y' \arrow[d, "f'"'] \arrow[r] & Y \arrow[d,"f"] \\ X' \arrow[r] & X \end{tikzcd}
    \qquad\qquad
    \begin{tikzcd} Y' \arrow[r] & Y \\ {Y'}^{+_{X'}} \arrow[u,"\rho_{f'}"] \arrow[r] & Y^{+_{X}} \arrow[u,"\rho_f" right] \rlap{.} \end{tikzcd}
  \]
\end{defn}

\begin{rmk}\label{rmk:partial-map-stable-functorial-factorization}
  The relative version of the construction of Proposition \ref{prop:partial-map-classifier} defines a pullback-preserving functorial factorization:
  \[
    \begin{tikzcd} W \arrow[d, "g^*f"'] \arrow[r, "f^*g"] \arrow[dr, phantom, "\lrcorner" very near start] & Y \arrow[d,"f"] \\ Z \arrow[r, "g"'] & X \end{tikzcd}\qquad = \qquad \begin{tikzcd} W \arrow[d, tail, "\eta_{g^*f}"'] \arrow[r, "f^*g"] \arrow[dr, phantom, "\lrcorner" very near start] & Y \arrow[d, tail, "\eta_f"] \\ W^{+_Z} \arrow[r] \arrow[dr, phantom, "\lrcorner" very near start] \arrow[d, "g^*f^+"'] & Y^{+_X} \arrow[d, "f^+"] \\ Z \arrow[r, "g"'] & X \end{tikzcd}
  \]
  satisfying the hypotheses of Example \ref{ex:loc-rep-from-cartesian-factorization}.
  This defines a weak factorization system whose left maps are the monomorphisms and whose right maps are those admitting a relative $+$-algebra structure.
\end{rmk}

\begin{rmk}\label{rmk:partial-map-awfs}
  The partial map classifier $\eta_Y \colon Y \rightarrowtail Y^+$ is the component at $Y$ of a unit natural transformation which is part of a monad structure on the (fibred) endofunctor $(-)^+ : \cE\to \cE$. Thus the object $Y^+ = \Omega_! \top_*Y$ is itself a (free) $+$-algebra.
  This can be used to show that the functorial factorization of Remark \ref{rmk:partial-map-stable-functorial-factorization} underlies an algebraic weak factorization system.
  See \cite[9.5]{GambinoSattler:2017fc} or \cite[\S{3}]{Awodey:2023} for details.
\end{rmk}

By the following proposition, we can see a relative $+$-algebra structure as consisting of a uniform choice of lifts against all monomorphisms.

\begin{prop}
  \label{prop:relative-plus-algebras-uniform}
  The category of relative $+$-algebras is isomorphic to the category whose
  \begin{enumerate}
  \item objects are maps $f \colon Y \to X$ paired with a choice of lifts against all monomorphisms uniformly in all pullback squares:
    \[
      \begin{tikzcd}[sep=1.5cm]  D \arrow[d, tail, "d"'] \arrow[r, "{\alpha}"] \arrow[dr, phantom, "\lrcorner" very near start]&  C \arrow[d, "c" pos=.6, tail] \arrow[r] & Y \arrow[d, "f"] \\ B \arrow[urr, dashed] \arrow[r, "{\alpha}"'] & A \arrow[r] \arrow[ur, dashed] & X \rlap{,}
      \end{tikzcd}
    \]
    and
  \item morphisms $f' \to f$ are commutative squares compatible with the choices of lifts.
  \end{enumerate}
\end{prop}
\begin{proof}
  By Proposition \ref{prop:partial-map-classifier} any lifting problem between a monomorphism and a map $f$ factors uniquely as
  \[
  \begin{tikzcd} C \arrow[d, tail, "c"'] \arrow[r, "x"] & Y \arrow[d, "f"] \arrow[dr, phantom, "="] & C \arrow[d, tail, "c"'] \arrow[r, "x"]  \arrow[dr, phantom, "\lrcorner" very near start] & Y \arrow[d, tail, "\eta_f"'] \arrow[r, equals] & Y \arrow[d, "f"] \\ Z \arrow[r, "y"'] & X &  Z \arrow[r, "\zeta_c^{x,y}"''] \arrow[rr, bend right, "y"']  & Y^{+_X} \arrow[ur, dashed, "\rho"] \arrow[r, "f^+"] & X
    \end{tikzcd}
 \]
 Thus a relative $+$-algebra structure uniquely equips $f$ with a uniform choice of lifts against all monomorphisms and conversely such lifts specialize to equip $f$ with a relative $+$-algebra structure.
 Likewise, compatibility of a square $f' \to f$ with chosen lifts against all monomorphisms reduces to compatibility with the retractions $\rho_{f'}$ and $\rho_f$.
 See \cite[3.7]{Awodey:2023} and \cite[9.9(i)]{GambinoSattler:2017fc}.
\end{proof}

\begin{defn}
  \label{defn:uniform-tfib-structure}
  Write $\mathscr{TF}$ for the notion of fibred structure on $\cE$ obtained by applying Example \ref{ex:loc-rep-from-cartesian-factorization} with the partial map factorization of Remark \ref{rmk:partial-map-stable-functorial-factorization}.
  We call $\mathscr{TF}$ the notion of \textbf{uniform trivial fibration} structure.
\end{defn}

The $\mathscr{TF}$-algebras are then exactly the relative $+$-algebras, while the $\mathscr{TF}$-morphisms are those pullback squares which are also relative $+$-algebra morphisms.
By Proposition \ref{prop:relative-plus-algebras-uniform}, the $\mathscr{TF}$-algebras are equivalently maps equipped with a choice of lifts against all monomorphisms uniformly in pullback squares, and a pullback square $f' \to f$ is a $\mathscr{TF}$-morphism when the chosen lifts against $f'$ are determined by restriction of those against $f$.

\begin{lem}\label{lem:loc-rep-triv-fib} The notion of fibred structure $\mathscr{TF}$ in an elementary topos is relatively acyclic and locally representable.
\end{lem}
\begin{proof}
Since, by Remark \ref{rmk:partial-map-stable-functorial-factorization}, the functorial factorization preserves pullbacks and our ambient category is locally cartesian closed, Example \ref{ex:loc-rep-from-cartesian-factorization} tells us that relative $+$-algebras define a locally representable notion of fibred structure.

The proof of relative acyclicity follows by an adaptation of Shulman's \cite[5.18]{Shulman:2019ai}. In this setting, relative acylicity asserts that for any solid-arrow pullback square whose horizontal maps are monomorphisms and vertical maps are relative $+$-algebras as below-left, the relative $+$-algebra structures encoded by the dashed maps below-right can be made to commute by changing the relative $+$-algebra structure for $f$:
\[ \begin{tikzcd} Y' \arrow[d, "f'"'] \arrow[r, tail, "i'"] \arrow[dr, phantom, "\lrcorner" very near start] & Y \arrow[d, "f"] \\ X' \arrow[r, "i"', tail] & X
\end{tikzcd}  \qquad \qquad \begin{tikzcd} Y' \arrow[d, tail, "\eta_{f'}"'] \arrow[r, tail, "i'"] \arrow[dr, phantom, "\lrcorner" very near start] & Y \arrow[d, tail, "\eta_f"] \\ Y'^{+_{X'}} \arrow[d, "f'^+"'] \arrow[r, tail] \arrow[dr, phantom, "\lrcorner" very near start] \arrow[u, dashed, bend right, "\rho_{f'}"' pos=.4] & \arrow[u, dashed, bend left, "\rho_f" pos=.4 ] Y^{+_X} \arrow[d, "f^+"] \\ X' \arrow[r, "i"', tail] & X \rlap{.}
\end{tikzcd}  \]
Since the functorial factorization of Remark \ref{rmk:partial-map-stable-functorial-factorization} is cartesian, the pushout below-left constructs the union of subobjects over $Y^{+_X}$ and thus defines a monomorphism:
\[
\begin{tikzcd}[sep=small] Y' \arrow[ddd, "\eta_{f'}"', tail] \arrow[ddrr, phantom, "\lrcorner" very near start, "\ulcorner" very near end] \arrow[rrr, tail, "i'"] & & & Y \arrow[ddd, tail, "\eta_f"]  \arrow[ddl, dotted] \\ \\ & & P \arrow[dr, dashed, tail, "j"'] \\ Y'^{+_{X'}} \arrow[urr, dotted] \arrow[rrr, tail] & & & Y^{+_X} \end{tikzcd} \qquad\qquad \begin{tikzcd} P \arrow[d, tail, "j"'] \arrow[r, "{(i'\rho_{f'}, \id_Y)}"] & Y \arrow[d, "f"] \\ Y^{+_X} \arrow[r, "f^+"'] \arrow[ur, dashed] & X \rlap{.} \end{tikzcd} \]
Since $f$ is a relative $+$-algebra, the resulting lifting problem admits a solution, defining a new relative $+$-algebra structure for $f$ that defines a $\mathscr{TF}$-morphism with the relative $+$-algebra structure for $f'$.
\end{proof}

When we forget structure and consider class of maps underlying $\mathscr{TF}$, we find another equivalent characterization.

\begin{prop}\label{prop:uniform-trivial-fibrations} The following are equivalent for a map $f \colon Y \to X$ in an elementary topos $\cE$:
  \begin{enumerate}
  \item \label{itm:tf-class-relative-+-algebra} $f$ is a relative $+$-algebra.
  \item \label{itm:tf-class-uniform-lift} $f$ lifts against all monomorphisms, uniformly in all pullback squares between monomorphisms.
  \item \label{itm:tf-class-lift} $f$ lifts against all monomorphisms.
  \end{enumerate}
\end{prop}
\begin{proof}
  The equivalence between \eqref{itm:tf-class-relative-+-algebra} and \eqref{itm:tf-class-uniform-lift} follows from Proposition \ref{prop:relative-plus-algebras-uniform}. Clearly \eqref{itm:tf-class-uniform-lift} implies \eqref{itm:tf-class-lift}, and \eqref{itm:tf-class-lift} implies \eqref{itm:tf-class-relative-+-algebra} because the diagram \eqref{eq:relative-plus-algebra} is a lifting problem against a monomorphism.
\end{proof}

\begin{defn}\label{defn:trivial-fibrations}
We refer to maps $f$ satisfying the equivalent conditions of Proposition \ref{prop:uniform-trivial-fibrations} as \textbf{trivial fibrations}.
\end{defn}

Internally to $\cE$, the relative $+$-algebras can be seen as generated by right lifting against the family $\top \colon 1 \to \Omega$ indexed by the subobject classifier $\Omega$.
In the case where $\cE$ is a presheaf topos, this can be externalized as generation by a small \emph{category} of maps.
Both of these viewpoints are instances of the framework of Swan~\cite{swan18lp} of lifting in a Grothendieck fibration: the codomain fibration for the internal viewpoint and the category-indexed families fibration for the external viewpoint.

\begin{con}\label{con:generating-cof}
Let $\cE = \Set^{\cC^\op}$ be a presheaf topos.
In the slice category over $\Omega$, the morphism $\top \colon 1 \to \Omega$ may be regarded as a subterminal object, determining a family of maps internally indexed by the base object $\Omega$. This family can be externalized to determine a functor $I \colon \int\Omega \to \cE^\2$ on the category of elements of $\Omega$, defined by pulling back this internal family to the representables.

The cartesian functor $I$ thus lifts the Yoneda embedding $\yo$ from the discrete fibration associated to the category of elements of $\Omega$ to the codomain fibration of $\cE$:
 \begin{equation*}\label{eq:lift-elements-of-Omega}
  \begin{tikzcd}  \int\Omega \arrow[d, "\pi"'] \arrow[r, dashed, "I"] & \cE^\2 \arrow[d, "\cod"] \\ \cC \arrow[r, hook, "\yo"] & \cE \rlap{.}
  \end{tikzcd}
\end{equation*}
It sends an element $\chi_c : \yo a \to \Omega$ to the subobject $ C \rightarrowtail \yo a$ that it classifies, while morphisms in $\int\Omega$
\[ \begin{tikzcd} \yo b \arrow[rr, "\alpha"] \arrow[dr, "\chi_{d}"'] & & \yo a \arrow[dl, "\chi_c"] \\ & \Omega\end{tikzcd}
  \]
are carried to pullback squares between subobjects as below:
\[
\begin{tikzcd}
   D \arrow[d, tail, "d"'] \arrow[r, "\alpha"] \arrow[dr, phantom, "\lrcorner" very near start] & C \arrow[d, tail, "c"]\\
  \yo b \arrow[r, "\alpha"'] & \yo a \rlap{.}
\end{tikzcd}
\]
\end{con}

Recall that for any index category $\cI$ and functor $I \colon \cI \to \cE^\2$ into an arrow category, there is a corresponding category $\cI^\boxslash$ whose objects are arrows of $\cE$ equipped with chosen lifts against the images of the objects of $\cI$, in a way that is natural in the morphisms of $\cI$ \cite[15]{BourkeGarner:AWFSI}.

In particular, when $\cE = \Set^{\cC^\op}$ is a presheaf topos, an object of the category $(\int\Omega)^\boxslash$ is a morphism $f \colon Y \to X$ in $\cE$ equipped with chosen lifts against subobjects of representables that are uniform in pullback squares:
\[
  \begin{tikzcd}[sep=1.5cm]  D \arrow[d, tail, "d"'] \arrow[r, "{\alpha}"] \arrow[dr, phantom, "\lrcorner" very near start]&  C \arrow[d, "c" pos=.6, tail] \arrow[r] & Y \arrow[d, "f"] \\ \yo b \arrow[urr, dashed] \arrow[r, "{\alpha}"'] & \yo a \arrow[r] \arrow[ur, dashed] & X \rlap{.}
  \end{tikzcd}
\]

\begin{prop}
  \label{prop:relative-+-algebras-generation}
  For $\cE = \Set^{\cC^\op}$ a presheaf topos, the category of relative $+$-algebras is isomorphic over $\cE^\2$ to $(\int\Omega)^\boxslash$.
\end{prop}
\begin{proof}
  The statement asserts that in a presheaf topos, the lifting properties of Proposition \ref{prop:relative-plus-algebras-uniform} reduce to the case where we only ask for lifts against subobjects of representables.
  See \cite[5.16]{GambinoSattler:2017fc}.
\end{proof}

\begin{rmk}
  In summary, in the setting of a presheaf topos, we have multiple isomorphic characterizations of the category of relative $+$-algebras and the notion of fibred structure $\mathscr{TF}$.
  Note, however, that these perspectives suggest two non-isomorphic algebraic weak factorization systems providing a functorial factorization of a map into a monomorphism followed by a trivial fibration.

  On the one hand, the relative $+$-algebra factorization underlies an awfs as described in Remark \ref{rmk:partial-map-awfs}.
  On the other hand, Garner's algebraic small object argument applied to the generating category $I \colon \int\Omega \to \cE^\2$ yields an awfs whose category of monad algebras is isomorphic to $(\int \Omega)^\boxslash$ \cite[4.4]{Garner:USOA}.
  By Proposition \ref{prop:relative-+-algebras-generation}, the category of monad algebras for the second awfs is thus isomorphic to the category of pointed endofunctor algebras for the first, which is the category of relative $+$-algebras of Definition \ref{defn:relative-+-algebra}.
  In fact, the relative $+$-algebra factorization is the one-step factorization of the algebraic small object argument.
  See also the discussion in \cite[9.5]{GambinoSattler:2017fc}.
\end{rmk}

\subsection{Universes}\label{ssec:realignment}

\begin{defn}\label{defn:universeoffibrations}
Fix a notion of fibred structure $\sF$.
A \textbf{universe} for $\sF$ is an $\sF$-algebra $\pi \colon \dot{U} \to U$ such that $\pi \colon \cE(-,U) \to \sF$ is an acyclic fibration, meaning that we have bicategorical lifts against Yoneda embeddings of monomorphisms $i \colon A \rightarrowtail B$ as below:
\[
\begin{tikzcd}
  \cE(-,A)
  \arrow[d, "i"']
  \arrow[r, "h"]
&
  \cE(-,U)
  \arrow[d, "\pi"]
\\
  \cE(-,B)
  \arrow[r, "p"']
  \arrow[ur, dashed,"k"]
&
  \sF \rlap{.}
\end{tikzcd}
\]
\end{defn}

Unpacked, this requires that given any pair of $\sF$-algebras $p, q$ and $\sF$-morphisms as displayed by the solid-arrow squares below, with $i \colon A \rightarrowtail B$ a monomorphism,
\[
\begin{tikzcd}
  D \arrow[dd, two heads, "q"'] \arrow[rr] \arrow[drr, phantom, "\lrcorner" very near start] \arrow[dddr, phantom, "\lrcorner" very near start]\arrow[dr] & & \dot{U} \arrow[dd, two heads, "\pi"] \\ & E  \arrow[dr, phantom, "\lrcorner" very near start]\arrow[ur, dashed] & ~ \\ A \arrow[dr, tail, "i"'] \arrow[rr, "h" near end] & & U \\ & B \arrow[ur, dashed, "k"'] \arrow[from=uu, two heads, crossing over, "p"' near start]
\end{tikzcd}
\]
there exists an extension $k$ of $h$ along $i$ factoring the back pullback square as a composite of pullbacks and defining an $\sF$-morphism from $p$ to $\pi$.

\begin{prop}\label{prop:relatively-acyclic-implies-generic}
Assume that $\cE$ has initial objects which are preserved by pullback along arbitrary maps.
Given a relatively acyclic notion of fibred structure $\sF$ with universe $\pi \colon \dot{U} \to U$, each $\sF$-algebra is a pullback of $\pi$.
\end{prop}
\begin{proof}
Suppose $p \colon E \fto B$ is an $\sF$-algebra.
The back pullback square in the diagram below gives the identity on the initial object an $\sF$-algebra structure, and by relative acyclicity, the $\sF$-algebra $p$ can be given an $\sF$-algebra structure making the left-hand pullback into an $\sF$-morphism.
Because $\pi \colon \dot{U} \to U$ is a universe, $p$ is then a pullback of $\pi$:
\[
\begin{tikzcd}[baseline=(current bounding box.south)]
  \emptyset \arrow[dd, equals] \arrow[rr] \arrow[drr, phantom, "\lrcorner" very near start] \arrow[dddr, phantom, "\lrcorner" very near start]\arrow[dr] & & \dot{U} \arrow[dd, two heads, "\pi"] \\ & E  \arrow[dr, phantom, "\lrcorner" very near start]\arrow[ur, dashed] & ~ \\ \emptyset \arrow[dr, tail, "!"'] \arrow[rr, "!" near end] & & U \rlap{.} \\ & B \arrow[ur, dashed] \arrow[from=uu, two heads, crossing over, "p"' near start]
\end{tikzcd} \qedhere
\]
\end{proof}

We now specialize to the setting of a presheaf topos $\cE = \Set^{\cC^\op}$ for some small indexing category $\cC$ to give an example of a universe.
For any regular cardinal $\kappa$ for which $\cC$ is $\kappa$-small, the Hofmann--Streicher construction \cite{HofmannStreicher:1997lg,Awodey:2022hs} provides a classifier $\varpi \colon \dot{V}_\kappa \to V_\kappa$ for $\kappa$-small families, i.e., those maps whose components have $\kappa$-small fibres.
As noted in Example \ref{ex:small-fibred-structure}, for sufficiently large $\kappa$ this defines a locally representable and relatively acyclic full notion of fibred structure $\sE^\kappa$.
By \cite[3.9]{Cisinski:2014uu}, \cite[8.4]{OrtonPitts:2018af}, or \cite[6]{Awodey:2022hs}, the classifier $\varpi \colon \dot{V}_\kappa \to V_\kappa$ is a universe for $\sE^\kappa$.

Now consider a notion of fibred structure $\sF$ on the presheaf topos $\cE$.

\begin{con}\label{con:loc-rep-universe}
If $\sF$ is locally representable, then for sufficiently large $\kappa$ we may define a $\kappa$-small $\sF$-algebra classifier $\pi \colon \dot{U}_\kappa \to U_\kappa$ as follows. Firstly, we define a new notion of fibred structure $\sF^\kappa$ for which an $\sF^\kappa$-algebra is an $\sF$-algebra that is $\kappa$-small. If $\sF$ is locally representable or relatively acyclic, then for $\kappa$ sufficiently large so that Example \ref{ex:small-fibred-structure} holds, $\sF^\kappa$ inherits these properties \cite[3.3, 3.11, 4.18, 5.14]{Shulman:2019ai}.

Now set $U_\kappa \coloneq \sF^\kappa(\varpi)$ and form the pullback
  \[
    \begin{tikzcd}
        \dot{U}_\kappa \arrow[d, "\pi"'] \arrow[r] \arrow[dr, phantom, "\lrcorner" very near start] & \dot{V}_\kappa \arrow[d, "\varpi"] \\ U_\kappa \arrow[r, "\psi_\varpi"'] & V_\kappa \rlap{.}
    \end{tikzcd}
  \]
\end{con}

As a special case of Lemma \ref{lem:locally-representable-stability}\eqref{itm:loc-rep-alg}:

\begin{lem}\label{lem:loc-rep-universe-algebra} The map $\pi \colon \dot{U}_\kappa \to U_\kappa$ is canonically an $\sF^\kappa$-algebra. \qed
\end{lem}

\begin{prop}\label{prop:has-universes}
Let $\sF$ be a locally representable notion of fibred structure on a presheaf topos.
For sufficiently large regular cardinals $\kappa$, the $\sF^\kappa$-algebra $\pi \colon \dot{U}_\kappa \to U_\kappa$ is a universe for $\sF^\kappa$.
\end{prop}
\begin{proof}
Construction \ref{con:loc-rep-universe} defines the $\sF^\kappa$-algebra classifier as the pullback
\[ \begin{tikzcd} \cE(-,U_\kappa) \arrow[dr, phantom, "\lrcorner" very near start]\arrow[d, "\pi"'] \arrow[r, "\psi_\varpi"] & \cE(-,V_\kappa) \arrow[d, "\varpi"] \\ \sF^\kappa \arrow[r] & \sE^\kappa \rlap{.}
\end{tikzcd}
\]
Note that this strict pullback is also a bicategorical pullback, as $\sF^\kappa \to \sE^\kappa$ is a strict discrete fibration.
Since the Hofmann--Streicher classifier $\varpi \colon \dot{V}_\kappa \to V_\kappa$ is a universe, the right-hand vertical map is an acyclic fibration, whence its bicategorical pullback is as well.
\end{proof}

For size reasons, multiple universes will be required to classify all maps belonging to a given notion of fibred structure. So that the maps classified by a given universe are closed under various categorical operations, we now assume that the cardinals $\kappa$ are inaccessible so that the corresponding Hofmann--Streicher universes $\varpi \colon \dot{V}_\kappa \to V_\kappa$ can be thought of as internalized Grothendieck universes.

\begin{defn}
\label{defn:has-universes}
A pullback-stable class of maps $\mathcal{P}$ in a presheaf topos \textbf{has universes} if for any cardinal $\lambda$, there exists an inaccessible cardinal $\kappa \ge \lambda$ and a universe $\pi \colon \dot{U}_\kappa \to U_\kappa$ for a relatively acyclic notion of fibred structure whose underlying maps are the $\kappa$-small maps in $\mathcal{P}$.
\end{defn}

In particular, each $\kappa$-small map in $\mathcal{P}$ is a pullback of $\pi \colon \dot{U}_\kappa \to U_\kappa$, by Proposition \ref{prop:relatively-acyclic-implies-generic}.

We now make a standing assumption that there exist arbitrarily large inaccessible cardinals.
Proposition \ref{prop:has-universes} then provides universes for the class of maps underlying any locally representable and relatively acyclic notion of fibred structure on a presheaf topos.
See \cite{Shulman:2019ai} or \cite{GratzerSterlingShulman:2022} for a treatment of universe levels in more general categorical settings.

\begin{nt}\label{nt:has-universes}
  In the setting of Definition \ref{defn:has-universes}, it is often not necessary to disambiguate between the inaccessible cardinals indexing universe levels. Thus, we typically write $\pi \colon \dot{U} \to U$ for a generic member of the classifying family of universes, without explicitly designating the cardinal bound.
\end{nt}

\section{Cylindrical model structures}\label{sec:cylindrical}

In this section, we lay the theoretical groundwork for the construction of our two models of homotopy type theory, proving our results at a level of generality that ensures that they will apply to both cubical sets and cubical species while also enabling their use elsewhere. In \S\ref{ssec:premodel}, we introduce the notion of cylindrical premodel structure \cite{Sattler:2020cms}, also used in \cite{CavalloSattler:2022re}, which provides the familiar structures of abstract homotopy theory  in a setting where the weak equivalences are not yet known to satisfy the 2-of-3 property. In particular, these axioms provide fibred mapping path space factorizations that are stable under slicing, the basic properties of which we establish in \S\ref{ssec:brown}.

In \S\ref{ssec:EEP}, we state and prove the equivalence extension property in a locally cartesian closed cylindrical premodel category in which the cofibrations are the monomorphisms and these are stable under pushout products in all slices. In \S\ref{ssec:frobenius}, we introduce the Frobenius condition and mention a few consequences. In \S\ref{ssec:univalence}, we connect the equivalence extension property to the univalence axiom in the presence of the Frobenius condition on the cylindrical premodel structure. In \S\ref{ssec:Ufib}, we use this to establish the fibrancy of the universe, assuming that the fibrations are defined from the trivial fibrations via one of the standard constructions. In \S\ref{ssec:FEP}, we translate the fibrancy of the universe into the fibration extension property, which implies that the cylindrical premodel structure is in fact a model structure, retroactively justifying the title of this section as well as the nonstandard encodings of the weak equivalences and the univalence axioms we use along the way.

\subsection{Cylindrical premodel structures}\label{ssec:premodel}

Following Barton \cite{Barton:2019pm}, a \textbf{premodel structure} on a category $\cE$ is a pair of weak factorization systems, called the \textbf{(\textbf{trivial cofibration}, \textbf{fibration})} and (\textbf{cofibration}, \textbf{trivial fibration}) weak factorization systems, such that every trivial cofibration is a cofibration (equivalently, any trivial fibration is a fibration).
We also require finite limits and colimits (in practice, often only pullbacks along fibrations and pushouts along cofibrations are needed).
We denote trivial cofibrations with the arrow $\cwto$, fibrations with $\fto$, cofibrations with $\cto$, and trivial fibrations with $\fwto$.

In a premodel structure, define a map to be a \textbf{weak equivalence} $\wto$ if it factors as a composite of a trivial cofibration followed by a trivial fibration. In particular, the trivial cofibrations and trivial fibrations admit such factorizations, so both of these classes are included in the class of weak equivalences. Conversely, by a standard argument:

\begin{lem} Any cofibration and weak equivalence is a trivial cofibration, and any fibration and weak equivalence is a trivial fibration.
\end{lem}
\begin{proof}
The proofs are dual, and standard. If a cofibration factors as a trivial cofibration followed by a trivial fibration, this presents a lifting problem
\[\begin{tikzcd} \bullet \arrow[d, tail] \arrow[r, tcofarrow] & \bullet \arrow[d, tfibarrow] \\ \bullet \arrow[ur, dashed] \arrow[r, equals] & \bullet\end{tikzcd}
\]
a solution to which presents the cofibration as a retract of the trivial cofibration.
\end{proof}

Thus, from the Joyal--Tierney characterization \cite[7.7--7.8]{JT:2007} of a (closed) Quillen model structure:

\begin{prop} A premodel structure defines a model structure if and only if the weak equivalences satisfy the 2-of-3 property. \qed
\end{prop}

\begin{rmk}\label{rmk:premodel-slicing} Premodel structures lift to slice and coslice categories, with all of the classes of maps created by the forgetful functor to the base category.
\end{rmk}

For a general premodel structure, the 2-of-3 property for the weak equivalences may be hard to prove (and is often false). A convenient technical device that can be used when present to analyze the weak equivalences in a premodel structure is an \emph{adjoint functorial cylinder}, introduced below, that satisfies three compatibility conditions making the premodel structure into a \emph{cylindrical premodel structure}.

\begin{defn}\label{defn:fun-htpy} A \textbf{functorial notion of homotopy} on a category $\cE$ is a reflexive binary relation on the hom-bifunctor in the category of profunctors from $\cE$ to $\cE$:
\[
    \begin{tikzcd} & \cE(-,-) \arrow[d, "{\epsilon}"] \arrow[dr, equals] \arrow[dl, equals] \\ \cE(-,-) & \cI(-,-) \arrow[r, "\partial_1"'] \arrow[l, "\partial_0"] & \cE(-,-) \rlap{.}
    \end{tikzcd}
\]
\end{defn}

For any pair of objects $A, B \in \cE$, we refer to elements of the set $\cI(A,B)$ as \textbf{homotopies} between maps from $A$ to $B$. More precisely, the fiber over a parallel pair of morphisms $f,g \colon A \rightrightarrows B$
\[ \begin{tikzcd} \cI(A,B)_{(f,g)} \arrow[d] \arrow[r] \arrow[dr, phantom, "\lrcorner" very near start] & \cI(A,B)\arrow[d, "{(\partial_0,\partial_1)}"] \\ \ast \arrow[r, "{(f,g)}"'] & \cE(A, B)\times \cE(A,B)
\end{tikzcd}
\]
defines the set of \textbf{homotopies} from $f$ to $g$. We write $\alpha \colon f \sim g$ to mean that $\alpha \in \cI(A,B)_{f,g}$. The map $\epsilon \colon \cE(A,B) \to \cI(A,B)$ sends each $f \colon A \to B$ to a \textbf{constant homotopy} $\epsilon_f \colon f \sim f$.

\begin{defn}\label{defn:adj-fun-cyl} A functorial notion of homotopy $\cI$ on $\cE$ is
    \begin{itemize}
        \item \textbf{representable} if the profunctor $\cI$ is covariantly represented by a functor $P \colon \cE \to \cE$, which then defines a \textbf{functorial cocylinder} $\cI(A,B) \cong \cE(A,PB)$, and
        \item \textbf{corepresentable} if the profunctor $\cI$ is contravariantly represented by a functor $C\colon \cE \to \cE$, which then defines a \textbf{functorial cylinder} $\cI(A,B) \cong \cE(CA,B)$.
    \end{itemize}
In the co/represented setting, by the profunctorial Yoneda lemma, the natural transformations $(\epsilon,\partial_0,\partial_1)$ determine natural transformations
\[ \begin{tikzcd} & \arrow[dr, equals] \arrow[dl, equals] \id & & & \arrow[dr, equals]\id \arrow[dl, equals] \arrow[d, Rightarrow, "\epsilon"] \\ \id \arrow[r, Rightarrow, "\partial_0"'] & C \arrow[u, Rightarrow, "\epsilon"] & \id \arrow[l, Rightarrow, "\partial_1"] & \id & P \arrow[l, Rightarrow, "\partial_0"] \arrow[r, Rightarrow, "\partial_1"']  & \id \rlap{.}
\end{tikzcd}
\]
When $\cI$ is \textbf{birepresentable}, that is both representable and corepresentable, these functors are adjoints $C \dashv P$ and the natural transformations are conjugates. As in Lemma \ref{lem:leibniz-application-transposition}, we use the same notation for a conjugate pair of transformations, e.g., $\epsilon \colon C \To \id$ and $\epsilon \colon \id \To P$. We follow \cite[3.9]{CavalloSattler:2022re} and refer to a birepresentable functorial notion of homotopy as an \textbf{adjoint functorial cylinder}.
\end{defn}

We now show that all of these notions are stable under slicing---that is, passage to $\cE_{/X}$---and coslicing---that is, passage to ${}^{X/}\cE$---over and under arbitrary objects $X \in \cE$. In fact it suffices to consider slice categories, since functorial notions of homotopy are self-dual.

\begin{lem}\label{lem:sliced-fun-htpy} If $\cE$ has a functorial notion of homotopy $\cI$ then for any $X \in \cE$ the slice category $\cE_{/X}$ has a functorial notion of homotopy $\cI_X$. Moreover:
    \begin{enumerate}
        \item if $\cI$ is corepresentable, then so is $\cI_X$, and
        \item if $\cI$ is representable and $\cE$ has pullbacks then so is $\cI_X$.
    \end{enumerate}
\end{lem}
\begin{proof} We leave the general case to the reader and construct the functorial cylinder and cocylinder in the birepresentable case.

  Given an object $g \colon Y \to X$ in the slice $\cE_{/X}$ its \textbf{fibred cylinder factorization} is created by the forgetful functor to $\cE$, with the projections to $X$ defined by composing in the diagram
  \[ \begin{tikzcd} Y + Y \arrow[r, "{(\partial_0, \partial_1)}"] \arrow[d, "f + f"'] & CY \arrow[r, "\epsilon"] \arrow[d, "Cf"] & Y \arrow[d, "f"] \\ X +X \arrow[r, "{(\partial_0, \partial_1)}"'] & CX \arrow[r, "\epsilon"'] & X \rlap{.} \end{tikzcd}\]

Meanwhile, the \textbf{fibred cocylinder factorization} is constructed as follows:
  \[
    \begin{tikzcd}[sep=tiny,baseline=(current bounding box.south)]
          Y \arrow[rrr] \arrow[dddd, "f"'] \arrow[dr, dashed] & & & PY \arrow[rrr] \arrow[dddd,  "Pf"] & & & Y \times Y  \arrow[dddd,  "f \times f"] \\ &  P_X Y \arrow[dddl,  dotted] \arrow[dddr, phantom, "\lrcorner" very near start]\arrow[urr, dotted] \arrow[dr, dashed ]
      \\ & & Y \times_X Y \arrow[ddll, dotted] \arrow[uurrrr, dotted]  \arrow[ddr, phantom, "\lrcorner" very near start]\\ ~
      \\ X \arrow[rrr] & & ~ & PX \arrow[rrr] & & & X \times X \rlap{.}
  \end{tikzcd} \qedhere
  \]
\end{proof}

\begin{rmk} Definitions \ref{defn:fun-htpy} and \ref{defn:adj-fun-cyl} are self-dual, so in particular the dual of Lemma \ref{lem:sliced-fun-htpy} applies to coslice categories ${}^{X/}\cE$.
\end{rmk}

Let $\cI$ be a birepresented notion of homotopy on a category $\cE$ with finite limits and colimits. Write $\partial \colon \id + \id \Rightarrow C$ and $\partial \colon P \Rightarrow \id\times\id$ for the conjugate pair of natural transformations with components defined by $\partial_0$ and $\partial_1$. The notion of a cylindrical premodel structure makes use of the Leibniz applications introduced in Definition \ref{defn:leibniz-application}.

\begin{defn}\label{defn:cylindrical}
A premodel structure on $\cE$ is \textbf{cylindrical} if $\cE$ admits an adjoint functorial cylinder so that:
\begin{enumerate}
    \item\label{itm:cylindrical-boundary} Leibniz pullback application of $\partial \colon P \Rightarrow \id \times \id$ preserves fibrations and trivial fibrations.
    \item\label{itm:cylindrical-endpoints} Leibniz pullback application of $\partial_0 \colon P \Rightarrow \id$ and $\partial_1 \colon P \Rightarrow \id$ sends fibrations to trivial fibrations.
\end{enumerate}
\end{defn}

By Lemma \ref{lem:leibniz-application-transposition} these conditions could be phrased dually in terms of Leibniz pushout application of the conjugate natural transformations. As observed in \cite[3.2, 3.11, 3.17]{CavalloSattler:2022re}:

\begin{lem}\label{lem:cylindrical-slicing} A cylindrical premodel structure on $\cE$ induces a cylindrical premodel structure on each of its coslice and slice categories.
\end{lem}
\begin{proof}
  We prove the case of slice categories, the coslices being dual. By Lemma \ref{lem:leibniz-application-transposition}, it suffices to show that Leibniz pushout application of $\partial \colon \id + \id \To C$ preserves cofibrations and trivial cofibrations and Leibniz pushout application of $\partial_0, \partial_1 \colon \id \To C$ send cofibrations to trivial cofibrations. But both these classes and these constructions are created by the forgetful functor to $\cE$ and $\cE$ is cylindrical, so this is immediate.
\end{proof}

The cylindrical premodel structure axioms allow us to deduce various ``2-of-3-like'' properties of ``acyclic'' morphisms without relying on the 2-of-3 property for the weak equivalences. Two such results are the following.

\begin{lem}[{\cite[3.19--20, 3.27]{CavalloSattler:2022re}}]\label{lem:fib-tfib-triangles} In a cylindrical premodel structure, in any diagram of the form below-left,  the fibration is a trivial fibration,
\[ \begin{tikzcd} &  A \arrow[dl, tcofarrow, "j"'] \arrow[dr, utcofarrow, "k"] & & &   A \arrow[dl, tfibarrow, "p"'] \arrow[dr, utfibarrow, "q"] \\ Y \arrow[rr, two heads, "f"'] & & X & Y \arrow[rr, two heads, "f"'] & & X \end{tikzcd}\]
and if the trivial fibrations are detected by lifting against cofibrations between cofibrant objects, the same is true in any diagram of the form above-right.
\end{lem}

The first statement is proven by exhibiting $f$ as a retract of a trivial fibration constructed using axiom \ref{defn:cylindrical}\eqref{itm:cylindrical-endpoints} in a retract diagram whose data is defined by lifting. The second statement holds more generally even when $f$ is not known to be a fibration, by an elementary lifting argument.

\subsection{Brown factorizations}\label{ssec:brown}

The structure of a cylindrical premodel structure is designed to provide fibred mapping cylinder and mapping path space factorizations that are stable under coslicing and slicing, respectively.  In this section, we focus on the mapping path space construction, which we call the ``Brown factorization'' after \cite{Brown:1973ah},
which will be used in the next section to establish the equivalence extension property.
\begin{con}\label{con:path-space}
  Given a map $f \colon Z \to Y$ in a cylindrical premodel category, its \textbf{Brown factorization} $f = p_f \cdot s_f$  is constructed by factoring the graph of $f$ as follows:
  \[
  \begin{tikzcd}
    Z
    \ar[r, "f"]
    \ar[dd, bend right=45, "{(1,f)}"']
    \ar[d, dashed, "s_f"]
  &
    Y
    \ar[d, "\epsilon"]
  \\
    B f
    \ar[r]
    \ar[d, "{(q_f, p_f)}" pos=0.6]
    \ar[dr, phantom, "\lrcorner" very near start]
  &
    P Y
    \ar[d, "\partial"]
  \\
    Z \times Y
    \ar[r, "f \times Y"']
  &
    Y \times Y \rlap{.}
  \end{tikzcd}
  \]
  By construction $f = p_f \cdot s_f$ and $1 = q_f \cdot s_f$.
  \end{con}

  \begin{lem}\label{lem:mapping-path-space-fibrations} For the Brown factorization of a map  $f \colon Z \to Y$ in a cylindrical premodel category,
  \[
  \begin{tikzcd} & B f \arrow[dr, "p_f"] \arrow[dl, "q_f"', bend right] \\ Z \arrow[rr, "f"'] \arrow[ur, "s_f"'] & & Y\,,
  \end{tikzcd}
  \]
  \begin{enumerate}
    \item If $Y$ is fibrant, then $(q_f, p_f) \colon Bf \to Z \times Y$ is a fibration.
    \item If $Y$ is fibrant, then $q_f \colon B f \to Z$ is a trivial fibration.
    \item If $Y$ and $Z$ are both fibrant, then $p_f \colon B_f \to Y$ is a fibration.
  \end{enumerate}
  \end{lem}
  \begin{proof}
  These maps arise as
  \begin{align*}
  \begin{tikzcd}[ampersand replacement=\&]
    B f
    \ar[r]
    \ar[d, "q_f"']
    \ar[dr, phantom, "\lrcorner" very near start]
  \&
    P Y
    \ar[d, "\partial_0"]
  \\
    Z
    \ar[r, "f"']
  \&
    Y \rlap{,}
  \end{tikzcd}
  &&
  \begin{tikzcd}[ampersand replacement=\&]
    B f
    \ar[r]
    \ar[d, "{(q_f, p_f)}" pos=0.6]
    \ar[dd, bend right=50, "p_f"']
    \ar[dr, phantom, "\lrcorner" very near start]
  \&
    P Y
    \ar[d, "\partial"]
  \\
    Z \times Y
    \ar[r, "f \times Y"']
    \ar[d, "\pi"]
  \&
    Y \times Y
  \\
    Y \rlap{.}
  \end{tikzcd}
  \end{align*}
  If $Y$ is fibrant, then by Definition \ref{defn:cylindrical}, $\partial \colon PY \twoheadrightarrow Y \times Y$ is a fibration and $\partial_0 \colon PY \fwto Y$ is a trivial fibration, proving the first two statements.
  If $Z$ is fibrant, then the projection $\pi \colon Z \times Y \twoheadrightarrow Y$ is a fibration as well, proving the third statement.
  \end{proof}

  \begin{rmk}\label{rmk:fibredbrownfactor}
  By Lemma \ref{lem:cylindrical-slicing}, Construction \ref{con:path-space} can be implemented in slice categories.
  Given a map $f \colon Z \to Y$ lying over $X$ via $g \colon Y \to X$, the \textbf{fibred Brown factorization} is defined by implementing the Brown factorization construction in the slice over $X$. This factors the graph of $f$, regarded as a morphism with codomain $Z \times_X Y$, through a pullback of the fibred path object as displayed below-left:
\begin{equation}\label{eq:fibred-Brown}
    \begin{tikzcd}[sep=tiny]
    Z \arrow[dddd, "f"'] \arrow[rrrrrr, "1 \times f"] \arrow[dr, dashed,"s_f"'] &&& &&& Z \times Y \arrow[dddd, "f \times 1"] \\ & B_X f\arrow[dddd, dotted] \arrow[dr, dashed, "{(q_f, p_f)}" ] \arrow[urrrrr,dotted] \arrow[dddr, phantom, "\lrcorner" very near start] \\ & & Z\times_X Y \arrow[dddd, dotted] \arrow[uurrrr, dotted]\arrow[ddr, phantom, "\lrcorner" very near start] \\ \\
          Y \arrow[rrr] \arrow[dddd, "g"'] \arrow[dr, dashed] & &~ & PY \arrow[rrr] \arrow[dddd,  "Pg"] & & & Y \times Y \arrow[dddd,  "g \times g"] \\ &  P_XY  \arrow[dddl, dotted] \arrow[dddr, phantom, "\lrcorner" very near start]\arrow[urr, dotted] \arrow[dr, dashed ]
      \\ & & Y\times_XY \arrow[ddll, dotted] \arrow[uurrrr, dotted]  \arrow[ddr, phantom, "\lrcorner" very near start]\\ ~
      \\ X \arrow[rrr] & & ~ & PX \arrow[rrr] & & & X \times X
  \end{tikzcd}
\begin{tikzcd}[sep=tiny]
  Z \ar[ddddr, "g f"'] \ar[dr, dashed, "s_f" near end] \ar[rrr, "1"]& & & Z \ar[dr, "s_f"] \ar[ddd, "g f"] \\
  & B_X f \arrow[dddr, phantom, "\lrcorner" very near start] \ar[ddd, dotted] \ar[rrr, crossing over, dotted] \arrow[dr, dashed, "{(q_f, p_f)}" ]& & &  Bf\arrow[dr, "{(q_f, p_f)}" ] \ar[ddd] \\
  & & Z \times_X Y \ar[ddl, dotted] \ar[dr, phantom, "\lrcorner" very near start] \ar[rrr, dotted, crossing over]  &   & & Z \times Y \ar[ddd, "{gf \times g}"]  \\
   &~ &~ & X\ar[dr, "\epsilon"] \\
  & X \ar[urr, "1"] \ar[rrr, "\epsilon" description] \ar[rrrrd, "{(1,1)}"']&~ & ~& PX \ar[dr, "\partial"] \\
  & & & & & X \times X \rlap{.}
\end{tikzcd}
  \end{equation}
  By interchange of the pullback constructing the Brown factorization with pullback to the slice over $X$, the fibred Brown factorization is a pullback of the non-fibred Brown factorization as indicated in the right diagram above.
  Here the right-hand rectangle is formed by applying the non-fibred Brown factorization to the commutative square from $f$ to the identity on $X$.

  In this setting, Lemma \ref{lem:mapping-path-space-fibrations} specializes to tell us that
  \begin{enumerate}
  \item when $g$ is a fibration, $(q_f, p_f) \colon B_X f \to Z \times_X Y$ is a fibration,
  \item when $g$ is a fibration, $q_f \colon B_X f \to Z$ is a trivial fibration,
  \item when $g$ and $gf$ are both fibrations, $p_f \colon B_X f \to Y$ is a fibration.
  \end{enumerate}
\end{rmk}

\begin{lem}\label{lem:mapping-path-space-stability}
The fibred Brown factorization is stable under all pullbacks.
\end{lem}
\begin{proof}
This is the combination of the description of the fibred Brown factorization in the right-hand diagram of~\eqref{eq:fibred-Brown} with pullback pasting.
\end{proof}

\begin{defn}\label{defn:contractible-map} A map $f \colon Z \to Y$ between fibrant objects in a cylindrical premodel category is called \textbf{contractible} when the right factor $p_f \colon B_f \to Y$ in its Brown factorization is a trivial fibration:
  \[ \begin{tikzcd} & Bf \arrow[dr, utfibarrow, "p_f"] \\ Z \arrow[ur, "s_f"] \arrow[rr, "f"'] & & Y \rlap{.}
  \end{tikzcd}
  \]
\end{defn}

In the presence of the 2-of-3 property, the contractible maps agree with the weak equivalences between fibrant objects:

\begin{lem}\label{lem:contractible-vs-weak-equivalence} In a cylindrical model category, where the weak equivalences satisfy the 2-of-3 property, a map between fibrant objects is contractible if and only if it is a weak equivalence.
\end{lem}
\begin{proof}
If the weak equivalences satisfy the 2-of-3 property, then the section $s_f$ of the trivial fibration $q_f$ is also a weak equivalence. Thus, again by 2-of-3, $f$ is a weak equivalence if and only if the fibration $p_f$ is a trivial fibration.
\end{proof}

For emphasis, we shall refer to a contractible map in a slice $\cE_{/X}$ as a \textbf{contractible map over} $X$. Explicitly, a fibred map $f \colon Z\to Y$ over $X$ is contractible just when its domain $Z\to X$ and codomain $Y\to X$ are fibrations, and the fibration $p_f \colon B_X f \twoheadrightarrow Y$ of Remark \ref{rmk:fibredbrownfactor} is a trivial fibration.

\subsection{Equivalence extension property}\label{ssec:EEP}

In this section, we show that under suitable hypotheses, a cylindrical premodel category satisfies the following condition, the significance of which is explained in the sequel.

\begin{defn}\label{defn:EEP} A cylindrical premodel structure has the \textbf{equivalence extension property} when any contractible map $e$ over an object $A$ can be extended along any cofibration $i \colon A \cto B$ to a contractible map $f$ over $B$ with a specified codomain extending that of the original map:
  \begin{equation}\label{eq:EEP}
    \begin{tikzcd}
      X_0 \arrow[drrr,phantom,"\lrcorner" pos=.05] \arrow[ddr, two heads, "p_0"'] \arrow[rr, dotted] \arrow[dr,uwearrow, "e"] & & Y_0 \arrow[dr, uwedashed, "f"] \arrow[ddr, dotted, two heads]\\ &  X_1 \arrow[d, two heads, "p_1" near start] \arrow[rr, crossing over] \arrow[drr, phantom, "\lrcorner" very near start] & & Y_1 \arrow[d, two heads, "q_1" near start]\\  & A \arrow[rr, tail, "i"]  & & B \rlap{.}
    \end{tikzcd}
  \end{equation}
  In a setting such as a presheaf topos where we have universe levels, there is an additional requirement: for sufficiently large inaccessible cardinals $\kappa$, if $p_0$, $p_1$, and $q_1$ are $\kappa$-small, so is the extended fibration in \eqref{eq:EEP}.
\end{defn}

\begin{thm}\label{thm:cylindrical-EEP} Let $\cE$ be a locally cartesian closed category with a cylindrical premodel structure in which the cofibrations are the monomorphisms, and these are stable under pushout-products in all slices.  Then the equivalence extension property holds in $\cE$.
\end{thm}

\begin{ex}
For instance, by Remark \ref{rmk:adhesive-pushout-products}, the hypotheses are satisfied in a cylindrical premodel structure on an elementary topos if the cofibrations are the monomorphisms. Moreover, in a presheaf topos, all of the constructions in the proof of Theorem \ref{thm:cylindrical-EEP} will respect universe levels.
\end{ex}

Our approach to the equivalence extension property phrased using contractible maps follows \cite{Sattler:2017ee}.
In a cylindrical model category, where the weak equivalences satisfy the 2-of-3 condition, this is equivalent by Lemma \ref{lem:contractible-vs-weak-equivalence} to the equivalence extension property phrased instead using weak equivalences as in \cite{KapulkinLumsdaine:2021sm,Shulman:2015ua}.

The proof of Theorem \ref{thm:cylindrical-EEP} occupies the remainder of this section.
To begin, in the diagram \eqref{eq:EEP}, we have $i^*Y_1 \cong X_1$ by hypothesis, and we define an object $Y_0$ with a map $f\colon Y_0 \to Y_1$ as a pullback of the pushforward along $i$ of the given fibred map $e\colon X_0\to X_1$\,:
\begin{equation}\label{eq:EEP-map} \begin{tikzcd}
  Y_0 \arrow[r, "\eta_{Y_0}"] \arrow[d, "f"'] \arrow[dr, phantom, "\lrcorner" very near start] & i_* X_0 \arrow[d, "i_*e"]  \\
   Y_1 \arrow[r, "\eta_{Y_1}"'] & i_* i^* Y_1 \rlap{.}  
   \end{tikzcd}
\end{equation}
By Lemma \ref{lem:mono-pushforward}, $i^*\eta_{Y_1}$ is invertible.
Considering the image of \eqref{eq:EEP-map} under the pullback-preserving functor $i^*$, we conclude that $i^*f$ is isomorphic to $i^*i_*e \cong e$.
In other words, $f \colon Y_0 \to Y_1$ pulls back along $i$ to the original map $e \colon X_0 \to X_1$, giving a diagram of the required form \eqref{eq:EEP}.

It remains to show that $q_1 f \colon Y_0 \to B$ is a fibration and that $f \colon Y_0 \to Y_1$ is a contractible map over $B$. We shall prove both in the slice over $B$.

For contractibility, consider the fibred Brown factorizations for both $e$ and $f$:
\[
  \begin{tikzcd}
     & B_A e \arrow[rr, dotted] & & B_B f \\ [-15pt]
    X_0 \arrow[ur, dashed] \arrow[ddr, two heads, "p_0"'] \arrow[rr, dotted] \arrow[dr, "e",uwearrow] & & Y_0 \arrow[ur, dashed] \arrow[dr, "f"] \arrow[ddr, "q_1f"' near end]\\ &  X_1 \arrow[from=uu, dashed, two heads, crossing over, "p_e" near end] \arrow[d, two heads, "p_1"' near start] \arrow[rr, crossing over] \arrow[drr, phantom, "\lrcorner" very near start] & & Y_1 \arrow[from=uu, dashed, crossing over, "p_f" near end]\arrow[d, two heads, "q_1" near start]\\  & A \arrow[rr, tail, "i"]  & & B \rlap{.}
  \end{tikzcd}
\]
By Lemma \ref{lem:mapping-path-space-stability}, the fibred Brown factorization for $f$ pulls back along $i$ to the factorization for $e$, and similarly the fibred path objects pullback $i^* P_BY_1 \cong P_AX_1$ (not shown in the diagram).  The relationship between the pushforward of the fibred Brown factorization for $e$ and that for $f$ is more complicated, however. To understand it, first consider the naturality cube resulting from the pullback square defining the map $(q_f, p_f)$ and the unit natural transformation $\eta \colon \id \Rightarrow i_*i^*$, which by Lemma \ref{lem:mapping-path-space-stability} determines the following commutative cube:
\[
  \begin{tikzcd}[sep=tiny]
    B_Bf \arrow[dr, phantom, "\lrcorner" very near start] \arrow[rrrr] \arrow[ddd, "{(q_f,p_f)}"'] \arrow[rrrd, "\eta"']& &  & & P_BY_1 \arrow[ddd,  "\partial" pos=.7] \arrow[rrrd, "\eta"]  \\ &~ & & i_*B_Ae \arrow[rrrr, crossing over] \arrow[dr, phantom, "\lrcorner" very near start]  & & & & i_*P_AX_1 \arrow[ddd, "{i_*\partial}"] \\  &  & ~ & & ~&~  \\ Y_0 \times_B Y_1 \arrow[rrrr] \arrow[rrrd, "\eta"']  & & & & Y_1\times_BY_1 \arrow[drrr, "\eta"] \\   & & & i_*(X_0 \times_A X_1) \arrow[from=uuu, crossing over, two heads, "{i_*(q_e,p_e)}"' pos=.3] \arrow[rrrr] & & & & i_*(X_1 \times_A X_1) \rlap{.}
    \end{tikzcd}
  \]
    The back face is the pullback in Construction \ref{con:path-space}, and the front face is its image under the right adjoint $i_*$, and is therefore also a pullback.  Since \eqref{eq:EEP-map} is a pullback, the bottom square is one as well. By pullback composition and cancelation, the top square is therefore also a pullback.

Now consider the naturality cube associated to the commutative square relating $p_f$ and~$\partial_1$:
    \[
    \begin{tikzcd}[sep=tiny]
B_Bf \arrow[rrrr] \arrow[ddd, "p_f"'] \arrow[ddr, dotted] \arrow[rrrd]& &  & & P_BY_1 \arrow[ddd, tfibarrow, "\partial_1"' pos=.6] \arrow[rrrd, "\eta"] \arrow[ddr,  "z" near start, dashed  ] \\ & & & i_*B_Ae \arrow[rrrr, crossing over]   & & & & i_*P_AX_1 \arrow[ddd, "i_*\partial_1", utfibarrow] \\  & \bullet \arrow[ddrr, phantom, "\lrcorner" pos=.05]  \arrow[urr, dotted] \arrow[dl, dotted] \arrow[rrrr, dotted] &  & & & \bullet  \arrow[urr, dotted] \arrow[dl, dotted] \arrow[ddrr, phantom, "\lrcorner" pos=.05] \\ Y_1 \arrow[rrrr, equals] \arrow[rrrd, "\eta"']  & & & & Y_1 \arrow[drrr, "\eta"]  \\   & & & i_*X_1 \arrow[from=uuu, crossing over, tfibarrow, "i_*p_e"' pos=.4] \arrow[rrrr, equals] & & & & i_*X_1 \rlap{.}
\end{tikzcd}
\]
The top square was just shown to be a pullback, and the bottom square is evidently one.  So when we form the pullbacks indicated in the left and right faces, we obtain a factorization of $p_f$ as a pullback of the map $i_*p_e$, after a pullback of the comparison map $z$ indicated as a dashed arrow in the right-hand face. This factorization will display $p_f$ as a trivial fibration, as we now argue.

First, since $e$ is a contractible map over $A$, its second Brown factor $p_e$ is a trivial fibration. Since the cofibrations are the monomorphisms, and therefore stable under pullback, the trivial fibrations are stable under pushforward, and so $i_*p_e$ is a  trivial fibration, as is any pullback of it.

Next, the map $z$ may be described as a Leibniz pullback application of the unit $\eta$ applied to the trivial fibration $\partial_1 \colon P_B Y_1 \fwto Y_1$.
But this is also a trivial fibration, as it is the Leibniz exponential, in the slice over $B$, of the cofibrant object $i \colon A \rightarrowtail B$ and the trivial fibration $\partial_1\colon P_BY_1 \fwto Y_1$, and monomorphisms are closed under pushout-products in slices.

Thus $p_f$ factors as a composite of pullbacks of trivial fibrations and so is itself a trivial fibration.  The map $f$ is therefore contractible over $B$, provided that its domain $q_1 f \colon Y_0 \to B$ is a fibration. But $q_1 f$ is a retract of $q_1p_f$:
\[ \begin{tikzcd} Y_0 \arrow[d, "f"'] \arrow[r, "s_f"] & B_B f \arrow[d, "p_f", utfibarrow] \arrow[r, "q_f"] & Y_0 \arrow[d, "f"] \\ Y_1 \arrow[d, "q_1"', two heads] \arrow[r, equals] & Y_1 \arrow[d, "q_1", two heads]  & Y_1 \arrow[d, "q_1", two heads] \\ B \arrow[r, equals] & B \arrow[r, equals] & B \rlap{.} \end{tikzcd} \]
Here the right rectangle commutes because in the fibred Brown factorization of $f \colon Y_0 \to Y_1$ over $B$, $p_f$ and $q_f$ live over $B$ and $p_f$ and $f$ both have target $q_1 \colon Y_1 \fto B$.
We have just shown that $p_f$ is a (trivial) fibration, while $q_1$ was assumed to be a fibration; thus the retract diagram implies $q_1 f \colon Y_0 \to B$ is also a fibration, as required. This completes the proof of Theorem \ref{thm:cylindrical-EEP}.

\subsection{The Frobenius condition}\label{ssec:frobenius}

In the setting of a locally cartesian closed category, it is natural to ask that a premodel structure satisfies the Frobenius condition.

\begin{defn}\label{defn:frobenius}
  A weak factorization system satisfies the \textbf{Frobenius condition} if the left maps are stable under pullback along the right maps. A premodel structure satisfies the \textbf{Frobenius condition} if this holds for both of its weak factorization systems.
\end{defn}

When the cofibrations are the monomorphisms, since these are stable under all pullbacks, the Frobenius condition only requires proof for the trivial cofibration--fibration weak factorization system. This condition has been studied in the homotopy type theory literature owing to the fact that, in a locally cartesian closed category, it is equivalent to the fibrations being closed under the pushforward operation, corresponding to type theory's $\Pi$-type construction. Various proofs of the Frobenius condition are given \cite{CCHM:2018ctt,GambinoSattler:2017fc,Awodey:2023,HazratpourRiehl:2022tc,Barton:2024}, depending on how exactly the fibrations are defined from the trivial fibrations. For the premodel structure introduced in \S\ref{sec:symmetric}, the result we will need is the following:

  \begin{prop}[{\cite[3.1.8]{ABCFHL},\cite[\S6]{Awodey:2023},\cite[\S 4]{HazratpourRiehl:2022tc}, \cite[8]{Barton:2024}}]\label{prop:generic-frobenius}
    Let $\cE$ be a locally cartesian closed category with a premodel structure in which the cofibrations are the monomorphisms.   Suppose there is an object $I$ such that a map is a fibration just when the Leibniz exponential of its pullback to the slice over $I$ by the diagonal $\delta \colon I \to I \times I$ is a trivial fibration in the slice premodel structure.
    Then the premodel structure satisfies the Frobenius condition.\qed
  \end{prop}

  Now assume we are working with a premodel structure in which there is a locally representable and relatively acyclic notion of fibred structure $\mathscr{TF}$ such that the $\mathscr{TF}$-algebras are the trivial fibrations. By Lemma \ref{lem:loc-rep-triv-fib}, these hypotheses are satisfied by a premodel structure on an elementary topos whose cofibrations are the monomorphisms. If this premodel structure satisfies the Frobenius condition, then the trivial fibration structure classifier has an important property:

  \begin{lem}\label{lem:TFibpreservesfibrations} Consider a locally cartesian closed category with a cylindrical premodel structure satisfying the Frobenius condition in which the trivial fibrations are generated by right lifting against cofibrations between cofibrant objects. Suppose $\mathscr{TF}$ is a locally representable and relatively acyclic notion of fibred structure such that the $\mathscr{TF}$-algebras are the trivial fibrations. Then if $f \colon Y\fto X$ is a fibration, then so is $\phi_{f}\colon \mathscr{TF}(f) \fto X$.
    \end{lem}
    \begin{proof}
To solve a lifting problem of the form
\[
  \begin{tikzcd}
    A \arrow[d, tcofarrow, "t"'] \arrow[r] & \mathscr{TF}(g) \arrow[d, "\phi_{f}"] \\
    B \arrow[r] \arrow[ur,dashed] & X \rlap{,}
  \end{tikzcd}
\]
we can equivalently solve the induced lifting problem against the pullback of $\phi_f$ along $B \to X$.
By pullback stability of the fibrations and Lemma \ref{lem:locally-representable-stability}(ii), it thus suffices to solve lifting problems of the form
 \begin{equation}\label{eq:TF-fib-lifting} \begin{tikzcd} A \arrow[d, tcofarrow, "t"'] \arrow[r, "u"] & \mathscr{TF}(g) \arrow[d, "\phi_{g}"]\\
 B \arrow[r, equals] \arrow[ur,dashed] & B \end{tikzcd}\end{equation}
where $t \colon A\cwto B$ is a trivial cofibration and $g \colon D \fto B$ is a fibration. This amounts to showing that if the fibration $g$ becomes a $\mathscr{TF}$-algebra upon pulling back along $t$, then it has a $\mathscr{TF}$-algebra structure making the pullback square
\begin{equation}\label{eq:TF-fib-realignment} \begin{tikzcd} C \arrow[d,tfibarrow, "t^*g"'] \arrow[dr, phantom, "\lrcorner" very near start] \arrow[r, utcofarrow, "s"] &  D \arrow[d,two heads, "g"] \\
  A \arrow[r, tcofarrow, "t"'] & B \end{tikzcd}\end{equation}
into a $\mathscr{TF}$-morphism. Note that by the Frobenius condition, the map $s$ in this pullback square is also a trivial cofibration, as a pullback of the trivial cofibration $t$ along the fibration $g$.

  Since $t^*g$ is a trivial fibration by assumption, the pushforward $t_*t^*g \colon t_*C\to B$ is also a trivial fibration.
  Since $t$ is monic, Lemma \ref{lem:mono-pushforward} implies that $t_*t^*g$ pulls back along $t$ to $t^*g$:
  \[ \begin{tikzcd} C \arrow[d, equals] \arrow[dr, phantom, "\lrcorner" very near start] \arrow[rr, utcofarrow, "{s'} "] && t_*C \arrow[dd, utfibarrow, "{t_*t^*g}" ] \\
  C \arrow[d,tfibarrow, "t^*g"']  \arrow[dr, phantom, "\lrcorner" very near start] \arrow[r, utcofarrow, "s"] &  D \arrow[d,two heads, "g"] \\
 A \arrow[r, tcofarrow, "t"'] & B \arrow[r, equals] & B \rlap{.} \end{tikzcd}\]
Again since $t_*t^*g$ is a (trivial) fibration, the pullback $s'$ is also a trivial cofibration, by the Frobenius condition. We therefore have a (trivial cofibration, fibration) and a (trivial cofibration, trivial fibration) factorization of a common map $g \cdot s = t_*t^*g \cdot s'$.  In a cylindrical premodel structure, it follows that the fibration $g$ is a trivial fibration, by an argument we now reprise.

In the commutative square defined by the pair of factorizations, form the pullback $P$ and factor the gap map in the square as a trivial cofibration followed by a fibration:
\[
\begin{tikzcd}[sep=small] C \arrow[dr, tcofarrow] \arrow[dddr, tcofarrow, "s"'] \arrow[rrrd, utcofarrow, "s'"] \\ & E \arrow[rr, two heads, dashed] \arrow[dd, two heads, dashed ] \arrow[dr, two heads] & & t_*C \arrow[dd, utfibarrow, "t_*t^*g"]\\ & & P \arrow[ur, two heads] \arrow[dl, tfibarrow] \arrow[dr, phantom, "\lrcorner" very near start] &  \\ &   D \arrow[rr, two heads, "g"'] & & B \rlap{.}
 \end{tikzcd}
\]
By the first part of Lemma \ref{lem:fib-tfib-triangles}, the dashed composite fibrations are both trivial fibrations, and now the fibration $g$ is the base of a commutative triangle of trivial fibrations with summit $E$, so $g$ is a trivial fibration by the second part of that lemma.

This proves that $g$ admits some $\mathscr{TF}$-algebra structure.
By relative acyclicity, this structure may be aligned with that of $t^* g$ to make the square \eqref{eq:TF-fib-realignment} into a $\mathscr{TF}$-morphism.
This specification of a new $\mathscr{TF}$-algebra structure on $g$ finally solves the original lifting problem \eqref{eq:TF-fib-lifting}.
\end{proof}

In the setting of Lemma \ref{lem:TFibpreservesfibrations}, Voevodsky constructs an alternate contractible map classifier, which we briefly digress to describe.

\begin{dig}\label{dig:triv-fib-logical-equivalence}
  In a locally cartesian closed category with a cylindrical premodel structure satisfying the Frobenius condition, for any fibration $f \colon Y \fto X$, there is a fibration $\phi_f \colon \isContr_Xf \fto X$ defined by pushing forward and then summing over its fibred path space fibration:
  \[ \begin{tikzcd} P_X Y \arrow[d, two heads, "\partial"']& \Pi_{Y} P_X Y \arrow[d, two heads, "(\pi_2)_* \partial"] & \Sigma_{Y}\Pi_{Y} P_X Y \arrow[d, two heads, "f \cdot (\pi_2)_*\partial"] & \isContr_X(f) \arrow[l, phantom, "\eqcolon"] \arrow[d, two heads, "\phi_f"]\\ Y \times_X Y \arrow[r, "\pi_2"'] &  Y \arrow[r, "f"', two heads] & X & X \rlap{.} \arrow[l, phantom, "\eqcolon"] \end{tikzcd}\]
By construction, sections to $\phi_f \colon \isContr_X(f) \fto X$ correspond to sections $s \colon X \to Y$ to $f$ together with a fibred homotopy $s \cdot f \sim_X \id_Y$.

As our notation suggests, there is a close relationship between the map $\phi_f \colon \isContr_X(f) \to X$  and the map $\phi_f \colon \mathscr{TF}(f) \to X$ constructed in Lemma \ref{lem:loc-rep-triv-fib} in the setting of a premodel structure on an elementary topos in which the cofibrations are the monomorphisms. For a fibration $f \colon Y \fto X$, these define ``logically equivalent notions'' of fibred structure witnessing that $f$ is a trivial fibration.

Indeed, if $\phi_f \colon \mathscr{TF}(f) \to X$ has a section, then $f$ is a trivial fibration, so admits a section $s \colon X \to Y$, since all objects are cofibrant. This data defines a lifting problem
\[ \begin{tikzcd} \emptyset \arrow[d, tail] \arrow[r, tail] & P_X Y \arrow[d, utfibarrow, "\partial"] \\ Y \arrow[r, "{(sf, \id_Y)}"']\arrow[ur, dashed, "h"] & Y \times_X Y \rlap{,} \end{tikzcd}\]
which admits a solution by the axiom \ref{defn:cylindrical}\eqref{itm:cylindrical-boundary} in the setting of  Lemma \ref{lem:cylindrical-slicing}, constructing a section $(s,h)$ of $\phi_f \colon \isContr_X(f) \to X$.

Conversely, if $\phi_f \colon \isContr_X(f) \to X$ has a section, then this data defines a retract diagram
\[ \begin{tikzcd} Y \arrow[d, "f"'] \arrow[r, "h"] & P_X Y \arrow[d, utfibarrow, "\partial_0"] \arrow[r, "\partial_1", utfibarrow] & Y \arrow[d, "f"] \\ X \arrow[r, "s"']& Y \arrow[r, "f"'] & X \end{tikzcd}\]
exhibiting $f$ as a retract of $\partial_0$, which is a trivial fibration  in the setting of  Lemma \ref{lem:cylindrical-slicing} by the axiom \ref{defn:cylindrical}\eqref{itm:cylindrical-endpoints}. Thus, $\phi_f \colon \mathscr{TF}(f) \to X$ has a section.
\end{dig}

\subsection{Univalence}\label{ssec:univalence}

In a premodel structure that satisfies the Frobenius condition and for which the fibrations have universes in the sense of Definition \ref{defn:has-universes}, the equivalence extension property of Definition \ref{defn:EEP} is related to Voevodsky's univalence axiom. To state this, we require the following construction. Following Notation \ref{nt:has-universes}, we write $\pi \colon \dot{U} \to U$ for a generic classifying universe and refer to this as the ``universe of fibrations,'' without explicitly designating a cardinal bound.

  \begin{lem}\label{lem:universal-equivalence} In a locally cartesian closed category with a cylindrical premodel structure satisfying the Frobenius condition, any fibration $\pi \colon \dot{U} \fto U$ has a factorization
      \[ \begin{tikzcd} & U \arrow[d, "r"] \arrow[dr, equals] \arrow[dl, equals] \\ U & \Eq(\dot{U}) \arrow[l, "s"] \arrow[r, "t"'] & U \end{tikzcd}\] of the diagonal $U\to U\times U$ such that
    \begin{enumerate}
      \item $(s,t) \colon \Eq(\dot{U}) \fto U \times U$ is a fibration  and
      \item the pullback of $\Eq(\dot{U}) \fto U \times U$ along any $e \colon \Gamma \to U \times U$ classifies (structured) contractible maps over $\Gamma$ between pullbacks of $p \colon \dot{U} \fto U$.
    \end{enumerate}
  \end{lem}

Under the stated hypotheses, the construction is the one due to Voevodsky, described, for instance, in \cite[\S 4]{Shulman:2015ua} and involves his classifier for contractible maps. As discussed in Digression \ref{dig:triv-fib-logical-equivalence}, we can prove Lemma \ref{lem:universal-equivalence} using any locally representable and relatively acyclic notion of fibred structure for trivial fibrations.

 \begin{proof}[Proof of Lemma \ref{lem:universal-equivalence}]   We construct $\Eq(\dot{U}) \fto U \times U$ by first forming the pullbacks on the left below, and then the internal hom between them in the slice over $U \times U$, as shown on the right:
    \[ \begin{tikzcd} \dot{U} \times U \arrow[r] \arrow[dr, phantom, "\lrcorner" very near start] \arrow[d, "\pi_1^* \pi"'] & \dot{U} \arrow[d, two heads, "\pi"] & U \times \dot{U} \arrow[l] \arrow[dl, phantom, "\llcorner" very near start]  \arrow[d, two heads, "\pi_2^* \pi"] & & \textup{Map}_{U \times U}(\pi_1^*\dot{U},\pi_2^*\dot{U}) \arrow[d, two heads, "{[\pi_1^* \pi, \pi_2^* \pi]_{U \times U}}"] \\
    U \times U \arrow[r, "\pi_1"'] & U & U \times U \arrow[l, "\pi_2"]  & & U \times U \rlap{.} \end{tikzcd}\]
  By the Frobenius condition, this map is a fibration.
  The counit $\epsilon \colon \textup{Map}_{U \times U}(\pi_1^*\dot{U},\pi_2^*\dot{U}) \times_{U \times U} \pi_1^*\dot{U} \to \pi_2^*\dot{U}$ equivalently defines a map
  \[ \epsilon \colon \textup{Map}_{U \times U}(\pi_1^*\dot{U},\pi_2^*\dot{U}) \times_{U \times U} \dot{U} \times U \to\textup{Map}_{U \times U}(\pi_1^*\dot{U},\pi_2^*\dot{U}) \times_{U \times U} U \times \dot{U}\]
  over $\textup{Map}_{U \times U}(\pi_1^*\dot{U},\pi_2^*\dot{U})$, which is the universal map between two pullbacks of $\pi$, i.e.\ small fibrations.

 We define $\Eq(\dot{U})$ by equipping this $\epsilon$ with the data of a contractible map over $\textup{Map}_{U \times U}(\pi_1^*\dot{U},\pi_2^*\dot{U})$, by taking the classifier $\phi_{p_\epsilon}\colon \mathscr{TF}(p_\epsilon) \to \textup{Map}_{U \times U}(\pi_1^*\dot{U},\pi_2^*\dot{U}) \times_{U \times U} \pi_2^*\dot{U}$ for trivial fibration structures on the right Brown factor $p_\epsilon \colon B_{\textup{Map}_U(\pi_1^*\dot{U},\pi_2^*\dot{U})}\epsilon \fto \textup{Map}_{p_\epsilon}(\pi_1^*\dot{U},\pi_2^*\dot{U}) \times_{U \times U} \pi_2^*\dot{U}$,
 pushing it forward to obtain an object over $\textup{Map}_{U \times U}(\pi_1^*\dot{U},\pi_2^*\dot{U})$, and then summing to obtain one over $U \times U$.

The resulting map $\Eq(\dot{U}) \to U\times U$ would thus be written in type theory as:
\[
\Eq(\dot{U}) = \Sigma_{A,B:U}\Sigma_{f : A\to B}\Pi_{b:B} \mathscr{TF}(\textup{fib}_f(b))
\to U \times U \rlap{.}
\]
It is easily seen to have the stated classifying property (ii).
It is a fibration as required by (i) provided that the map $\phi_{p_\epsilon}\colon \mathscr{TF}(p_\epsilon) \to \textup{Map}_{U \times U}(\pi_1^*\dot{U},\pi_2^*\dot{U}) \times_{U \times U} \pi_2^*\dot{U}$ is one.  But this follows from Lemma \ref{lem:TFibpreservesfibrations}, since $\textup{fib}_f(b)$ is just the right Brown factor $p_\epsilon \colon B_{\textup{Map}_U(\pi_1^*\dot{U},\pi_2^*\dot{U})}\epsilon \fto \textup{Map}_{p_\epsilon}(\pi_1^*\dot{U},\pi_2^*\dot{U}) \times_{U \times U} \pi_2^*\dot{U}$, which is a fibration by Remark \ref{rmk:fibredbrownfactor}.
\end{proof}

By the construction just given, the fibration $(s,t) \colon \Eq(\dot{U}) \fto U\times U$ factors as follows:
\[ \begin{tikzcd}[sep=large] \Eq(\dot{U}) \arrow[r, two heads, "\upsilon"] \arrow[rr, bend right=10, "{(s,t)}"'] &  \textup{Map}_{U \times U}(\pi_1^*\dot{U},\pi_2^*\dot{U}) \arrow[r, two heads, "{[\pi_1^* \pi,\pi_2^* \pi]_{U \times U}}"] & U \times U \rlap{.} \end{tikzcd}\]
The contractible map classifier just constructed satisfies a relative version of the relative acyclicity property of the following form inherited from relative acyclicity for $\mathscr{TF}$.

\begin{lem}\label{lem:universal-equivalence-alignment}
In a locally cartesian closed category with a cylindrical premodel structure satisfying the Frobenius condition, contractible map structures defined using a locally representable and relatively acyclic notion of fibred structure $\mathscr{TF}$ for trivial fibrations can be aligned along monomorphisms, in the sense that the kernel pair projections lift against monomorphisms:
  \[ \begin{tikzcd} A \arrow[d, "i"', tail] \arrow[r] & \bullet \arrow[r, two heads] \arrow[d, two heads] \arrow[dr, phantom, "\lrcorner" very near start] &   \Eq(\dot{U}) \arrow[d, two heads, "\upsilon"] \\ B \arrow[r] \arrow[ur, dashed] & \Eq(\dot{U}) \arrow[r, two heads, "\upsilon"'] & \textup{Map}_{U \times U}(\pi_1^*\dot{U},\pi_2^*\dot{U}) \rlap{.} \end{tikzcd}\]
\end{lem}
\begin{proof}
By construction, the map $\upsilon$ is the pushforward of the classifier \[ \phi_{p_\epsilon}\colon \mathscr{TF}(p_\epsilon) \to \textup{Map}_{U \times U}(\pi_1^*\dot{U},\pi_2^*\dot{U}) \times_{U \times U} \pi_2^*\dot{U}\] for trivial fibration structures. Since the notion of fibred structure $\mathscr{TF}$ is locally representable and relatively acyclic, by Lemma \ref{lem:loc-rep-rel-acyclic-kernel-pair} the maps in the kernel pair of $\phi_{p_\epsilon}$ lift against monomorphisms. Since monomorphisms are stable under pullback, this condition is stable under pushforward.
\end{proof}

The construction of Lemma \ref{lem:universal-equivalence} allows us to codify univalence as follows.

\begin{defn}\label{defn:univalence} A fibration $\pi  \colon \dot{U} \fto U$ is \textbf{univalent} if the map $t \colon \Eq(\dot{U}) \fto U$ is a trivial fibration.
\end{defn}

\begin{rmk}
Definition \ref{defn:univalence} connects to the standard homotopy type theoretic encoding of the univalence axiom as follows. By Lemma \ref{lem:universal-equivalence}, the diagonal on $U$ lifts through a map $\id \colon U \to \Eq(\dot{U})$, classifying the identity map. This factorization of the diagonal can be related to the canonical one of the cocylinder by a map $u$, as indicated below:
 \[ \begin{tikzcd} U \arrow[d, "\epsilon"'] \arrow[r, "\id"] & \Eq(\dot{U}) \arrow[d, two heads, "{(s,t)}"] \\ PU \arrow[ur, dashed, "u" ]\arrow[r, "\partial"'] & U \times U \rlap{.}
\end{tikzcd}
\]
If the base of the universe is fibrant, as will be proven under mild hypotheses in \S\ref{ssec:Ufib} below, the map $\partial_1 \colon PU \fwto U$ will be a trivial fibration, so in the presence of the 2-of-3 axiom, $t$ is a trivial fibration if and only if $u$ is a weak equivalence.
\end{rmk}

\begin{prop}\label{prop:EEP-univalence} Consider a cylindrical premodel structure on a presheaf topos satisfying the Frobenius condition in which the cofibrations are the monomorphisms. If the premodel structure has universes in the sense of Definition \ref{defn:has-universes},
the equivalence extension property holds if and only if each universe $\pi \colon \dot{U} \fto U$ is univalent.
\end{prop}

\begin{proof}
To prove the equivalence extension property assuming univalence, choose a univalent universe sufficiently large to classify the data in \eqref{eq:EEP} by means of a lifting problem
\[
\begin{tikzcd}
  A
  \arrow[d, tail, "i"']
  \arrow[r, "\overline{e}"]
&
  \Eq(\dot{U})
  \arrow[d, "t", utfibarrow]
\\
  B \arrow[r, "\overline{q}_1"']
  \arrow[ur, dashed, "\overline{f}"']
&
  U \rlap{.}
\end{tikzcd}
\]
For this, we first choose classifying maps $\overline{p}_0 \colon A \to U$ for $p_0$ and $\overline{q}_1 \colon B \to U$ and then use Lemma \ref{lem:universal-equivalence} to extend the map $(\overline{p}_0, \overline{q}_1 i) \colon A \to U \times U$ to a map $\overline{e} \colon A \to \Eq({\dot{U}})$ classifying the contractible map $e$.
By univalence, $t$ is a trivial fibration, so this lifting problem has a solution $\overline{f}$, which classifies a contractible map $f$ that pulls back along $i$ to $e$.

For the converse, consider a lifting problem as above, and suppose the equivalence extension property holds.
By Lemma \ref{lem:universal-equivalence}, the map $\overline{e}$ classifies a contractible map $e$ between fibrations into $A$ as in \eqref{eq:EEP}, while $\overline{q}_1$ classifies a fibration $q_1$ into $B$ that pulls back along $i$ to the codomain of $e$.
By the equivalence extension property, the equivalence extends to an equivalence $f$ over $B$ with codomain $q_1$ at the same universe level.
Using the given universe and relative acyclicity of its associated notion of fibred structure, we obtain a classifying map $\overline{q}_1$ for $q_1$ so that the exterior rectangle of classifying maps commutes:
\[
\begin{tikzcd}
  A
  \arrow[d, tail, "i"']
  \arrow[rrr, "\overline{e}"]
&&&
  \Eq(\dot{U}) \arrow[dl, two heads]
  \arrow[d, two heads, "{(s, t)}"]
\\
  B
  \arrow[r, "\overline{f}"]
  \arrow[rrr, bend right=15, "{(\overline{q}_0, \overline{q}_1)}"']
&
  \Eq(\dot{U})
  \arrow[r, two heads]
&
  \textup{Map}_{U \times U}(\pi_1^*\dot{U}, \pi_2^*\dot{U}) \arrow[r, two heads]
&
  U \times U \rlap{.}
\end{tikzcd}
\]
In fact, by the universal property of the fibration $[\pi_1^* \pi, \pi_2^* \pi]_{U \times U} \colon \textup{Map}_{U \times U}(\pi_1^* \dot{U}, \pi_2^* \dot{U}) \fto U \times U$ and commutativity the diagram \eqref{eq:EEP}, the interior of the diagram commutes as well.
Thus, our original lifting problem factors as displayed below:
\[
\begin{tikzcd}
  A
  \arrow[r]
  \arrow[d, tail, "i"']
  \arrow[rr, bend left=15, "\overline{e}"]
&
  \bullet
  \arrow[d, utfibarrow]
  \arrow[r, utfibarrow]
  \arrow[dr, phantom, "\lrcorner" very near start]
&
  \Eq(\dot{U})
  \arrow[d, two heads]
  \arrow[r, equals]
&
  \Eq(\dot{U})
  \arrow[d, two heads, "t"]
\\
  B
  \arrow[r, "\overline{f}"]
  \arrow[ur, dashed]
  \arrow[rrr, bend right=15, "\overline{q}_1"']
&
  \Eq(\dot{U})
  \arrow[r, two heads]
&
  \textup{Map}_{U \times U}(\pi_1^* \dot{U}, \pi_2^* \dot{U})
  \arrow[r, two heads]
&
  U \rlap{,}
\end{tikzcd}
\]
and can be solved by Lemma \ref{lem:universal-equivalence-alignment}, which aligns the equivalence structure on $f$ with that of $e$.
\end{proof}

\subsection{Fibrant universes}\label{ssec:Ufib}

We next introduce an axiomatic setup that allows us to use Proposition \ref{prop:EEP-univalence} to infer that the universes $\pi \colon \dot{U} \to U$ of fibrations have fibrant base objects $U$. Our argument follows that in \cite[2.12]{ABCFHL}.

Suppose that $\cE$ has a (cofibration, trivial fibration) weak factorization system in which every object is cofibrant, and let $P \colon \cE \to \cE$ be a finite-product preserving endofunctor equipped with a natural retraction, i.e.\ $\epsilon \colon \id \Rightarrow P$ and $\delta \colon P \Rightarrow \id$ such that $\delta \cdot \epsilon = \id$.  For instance, $P$ could be the cocylinder part of an adjoint functorial cylinder with $\delta$ taken to be either $\partial_0$ or $\partial_1$.  Alternately:

\begin{ex}\label{ex:generic-point} For any object $I$ in a cartesian closed category $\cE$, we have a diagram in the slice $\cE_{/I}$
  \[ \begin{tikzcd} I \arrow[dr, equals] \arrow[r, "\delta"] & I \times I \arrow[r, "\epsilon"] \arrow[d, "\epsilon"] & I \arrow[dl, equals] \\ & I \end{tikzcd}\]
  expressing the terminal object as a retract of $I$ pulled back to the slice. Here $\delta$ is the diagonal map and $\epsilon$ is the product projection obtained by pulling back $I \to 1$ to the slice. Exponentiating by these objects defines an endofunctor $P \colon \cE_{/I} \to \cE_{/I}$ together with natural transformations $\epsilon \colon \id \Rightarrow P$ and $\delta \colon P \Rightarrow \id$ such that $\delta\cdot\epsilon = \id$.
\end{ex}

By a \textbf{reflexive relation} $R\rightrightarrows X$ on an object $X$ we mean a factorization of the diagonal:
   \[
   \begin{tikzcd} & X \arrow[dr, equals] \arrow[dl, equals] \arrow[d, "r"] \\ X & R \arrow[l, "s"] \arrow[r, "t"'] & X \rlap{.} \end{tikzcd}
 \]
Note that we do not require the canonical pairing $(s,t) \colon R\to X\times X$ to be monic.

\begin{defn}\label{defn:delta-contractor} A \textbf{$\delta$-contractor} for a reflexive relation $R\rightrightarrows X$ is a map $c \colon PX \to PR$ making the following diagrams commute:
  \[ \begin{tikzcd} & PX \arrow[dl, equals] \arrow[r, "\delta_X"] \arrow[d, dashed, "c"] & X \arrow[d, "\epsilon_X"] & & PX \arrow[r, dashed, "c"] \arrow[d, "\delta_X"'] & PR \arrow[d, "\delta_R"] \\ PX & PR \arrow[l, "Ps"] \arrow[r, "Pt"'] & PX & & X \arrow[r, "r"'] & R \rlap{.}
  \end{tikzcd}
  \]
\end{defn}

\begin{rmk}
  To gain some intuition for this definition, suppose we are in a topological setting and $PX=X^I$ is the path
  space functor, $\epsilon$ the constant path operation, and $\delta$ evaluates a path at some fixed point $i\in I$.
  A $\delta$-contractor $c$ takes a path
  $p \colon x_0 \rightsquigarrow x_1$ in $X$ and produces a square as shown below, where the horizontal arrows
  are paths, the vertical arrows are witnesses to the relation $R$, and $x_i$ is the value of $p$ at $i$:
  \[
    \begin{tikzcd}
      x_0 \ar[<->]{d} \ar[rightsquigarrow]{rr}[xshift=-3ex]{p} &[1em] {\phantom{x}} \ar[<->]{d}{r(x_i)} & x_1 \ar[<->]{d} \\
      x_i \ar[rightsquigarrow]{rr}[below,xshift=-3ex]{\epsilon(x_i)} & {\phantom{x}} & x_i \\[-1.5em]
      0 & i & 1 \rlap{.}
    \end{tikzcd}
  \]

  The first diagram in Definition~\ref{defn:delta-contractor} determines the horizontal arrows: it asks that $c$ is
  a path of witnesses relating $p$ to the constant path $\epsilon(x_i)$. The second diagram asks that the
  value of $c$ at $i$, which relates $x_i$ to itself, is the reflexivity for $R$.
\end{rmk}

\begin{lem}\label{lem:delta-contractor} Let $R\rightrightarrows X$ be a reflexive relation. If the Leibniz pullback application of $\delta$ to $(s,t) \colon R \to X \times X$ is a trivial fibration, then $R$ has a $\delta$-contractor.
\end{lem}
\begin{proof} The required diagrams from \ref{defn:delta-contractor} can be repackaged into a single lifting problem as follows:
  \[ \begin{tikzcd}[column sep=7em] & PR \arrow[d, utfibarrow, "{((Ps,Pt),\delta_R)}"] \\
  PX \arrow[ur, dashed, "c"] \arrow[r, "{((\id, \epsilon_X\delta_X), r\delta_X)}"'] & (PX \times PX) \times_{X \times X} R \rlap{.}
  \end{tikzcd}\]
 But the vertical map is the said Leibniz pullback application $\delta\hato (s,t)$, which is assumed to be a trivial fibration, and so there is the indicated lift $c$, since all objects are cofibrant.
 \end{proof}

\begin{lem}\label{lem:delta-contractor-retract} Let $R\rightrightarrows X$ be a reflexive relation with a $\delta$-contractor. Consider the square
  \[
    \begin{tikzcd}
      R \arrow[r, "s"] \arrow[d, "t"'] & X \arrow[d, "!_X"] \\ X \arrow[r, "!_X"'] & 1
    \end{tikzcd}
  \]
  as a morphism $t \to {!_X}$ in $\cE^\2$.  The image of this morphism under the Leibniz pullback application functor $\delta\hato - \colon \cE^\2 \to \cE^\2$ is a split epimorphism.
\end{lem}
\begin{proof}
  The claim is that the canonical square on the right below admits a section
  \[
    \begin{tikzcd}
      PX \arrow[d, "{\delta_X}"']  \arrow[r, dashed, "c"] & PR \arrow[d, "{(Pt, \delta_R)}"] \arrow[r, "Ps"] & PX \arrow[d, "{\delta_X}"] \\ X \arrow[r, dashed, "{(\epsilon_X,r)}"'] & PX \times_X R  \arrow[r, "{ s\pi}"'] & X \rlap{.}
    \end{tikzcd}
  \]
The notion of a $\delta$-contractor is such that the indicated maps constitute just such a section.
  \end{proof}

\begin{lem}\label{lem:FEP-from-EEP} Let $R\rightrightarrows X$ be a reflexive relation such that the Leibniz pullback applications of $\delta$ to $(s,t) \colon R \to X \times X$ and $t\colon R \to X$ are both trivial fibrations. Then $\delta_X \colon PX \to X$ is also a trivial fibration.
\end{lem}
\begin{proof} Note that $\delta_X \colon PX \to X$ is the Leibniz pullback application of $\delta$ to $!_X \colon X \to 1$.
By Lemmas \ref{lem:delta-contractor} and \ref{lem:delta-contractor-retract}, $\delta\hato !_X$ is a retract of $(Pt,\delta_R) = \delta\hato t$ and thus a trivial fibration.
\end{proof}

When the fibrations are created from the trivial fibrations in a particular way, Lemma \ref{lem:FEP-from-EEP} can be used to establish the fibrancy of an object $X$ admitting a suitable reflexive relation. For later use, we introduce the following general definitions.

\begin{defn}\label{defn:interval-fibrations} Let $\cE$ be a (locally) cartesian closed category with a class of trivial fibrations.
  \begin{enumerate}
    \item\label{itm:biased-fibrations} Relative to an interval object  $\delta_0,\delta_1 \colon 1 \to I$  in $\cE$, the \textbf{biased fibrations} are those maps whose Leibniz exponentials by $\delta_0$ and $\delta_1$ are trivial fibrations.
    \item\label{itm:unbiased-fibrations} Relative to an object $I \in \cE$, the \textbf{unbiased fibrations} are those maps for which the Leibniz exponential of their pullback to the slice over $I$ by the diagonal $\delta \colon I \to I \times I$ is a trivial fibration in the slice.
  \end{enumerate}
\end{defn}

\begin{prop}\label{prop:FEP-standard-endpoints}  Let $\cE$ be a cartesian closed category with a premodel structure in which its fibrations are the biased fibrations defined relative to an interval object.
Then  an object $X$ is fibrant if it has a reflexive relation $s,t\colon R\rightrightarrows X$ such that both $(s,t) \colon R \to X \times X$ and $t \colon R \to X$ are fibrations.
\end{prop}

\begin{proof}
  As in Example \ref{ex:generic-point}, exponentiation by the interval defines an endofunctor $(-)^I$ equipped with a natural retraction  $\epsilon \colon \id \Rightarrow (-)^I$ and $\delta_0,\delta_1 \colon (-)^I \To \id$. Applying Lemma \ref{lem:FEP-from-EEP} separately with $\delta_0$ and $\delta_1$, we see that both $(\delta_0)_X = \delta_0\hato !_X$ and $(\delta_1)_X = \delta_1\hato !_X$ are trivial fibrations, proving that $X$ is fibrant.
\end{proof}

\begin{prop}\label{prop:FEP-generic-point} Let $\cE$ be a cartesian closed category with a premodel structure in which the fibrations are the unbiased fibrations defined relative to an object $I$.  Then an object $X$ is fibrant if it has a reflexive relation $s,t\colon R\rightrightarrows X$ such that $(s,t) \colon R \to X \times X$ and $t \colon R \to X$ are both fibrations.
\end{prop}
\begin{proof}
  By Example \ref{ex:generic-point} the standing hypotheses of this section are satisfied in the slice over $I$. The fibrations $(s,t) \colon R \twoheadrightarrow X \times X$ and $t \colon R \twoheadrightarrow X$ pullback to fibrations in the sliced premodel structure over $I$. Lemma \ref{lem:FEP-from-EEP} applies, and $\delta\hato !_{X\times I}$ is therefore a trivial fibration, proving that $X$ is fibrant.
\end{proof}

We now combine these observations with the construction of the previous section to prove that the universe is a fibrant object under the combined hypotheses of these sections.

\begin{prop}\label{prop:fibrant-universe}
Suppose $\cE$ is a presheaf topos with a cylindrical premodel structure satisfying the Frobenius condition in which the cofibrations are the monomorphisms. If the fibrations are characterized as in Proposition \ref{prop:FEP-standard-endpoints} or \ref{prop:FEP-generic-point} and have universes, then the bases of the universal fibrations $\pi \colon \dot{U} \fto U$ are fibrant objects.
\end{prop}
\begin{proof}
By Lemma \ref{lem:universal-equivalence}, the fibration $\pi \colon \dot{U} \fto U$ gives rise to a reflexive relation $\Eq(\dot{U}) \rightrightarrows U$ for which the pairing $\Eq(\dot{U}) \fto U \times U$ is a fibration.
By Theorem \ref{thm:cylindrical-EEP}, the equivalence extension property holds, so by Proposition \ref{prop:EEP-univalence} the map $t \colon \Eq(\dot{U}) \fwto U$ is a trivial fibration, and in particular a fibration.
Now either Proposition \ref{prop:FEP-standard-endpoints} or \ref{prop:FEP-generic-point} applies to conclude that $U$ is fibrant.
\end{proof}

\subsection{Fibration extension property and 2-of-3}\label{ssec:FEP}

Recall Definition \ref{defn:has-universes}, which introduces what it means for a premodel structure on a presheaf topos to have universes. We say that a premodel structure \textbf{has fibrant universes} if in addition the base of each of these universes for each sufficiently large inaccessible cardinal is fibrant.

The aim of this section will be to connect the fibrancy of the universes to a useful property of the premodel structure.

\begin{defn}\label{defn:FEP}
A premodel structure on a presheaf topos satisfies the \textbf{fibration extension property} just when, for each sufficiently large inaccessible cardinal $\kappa$, any $\kappa$-small fibration $p \colon X \fto A$, and trivial cofibration $t \colon A \cwto B$, there exists a $\kappa$-small fibration over $B$ which pulls back to $p$ along $t$:
\[ \begin{tikzcd} X \arrow[d, two heads, "p"'] \arrow[r, dashed] \arrow[dr, phantom, "\lrcorner" very near start] & Y \arrow[d, dashed, two heads, "q"] \\ A \arrow[r, tcofarrow, "t"'] & B \rlap{.}
\end{tikzcd}
\]
\end{defn}

There is a well-known connection between the fibration extension property and fibrancy of the universe \cite{Shulman:2015ua} that we spell out carefully because we are working with a somewhat different axiomatization here.

\begin{lem}\label{lem:fibU_fibext}
Any premodel structure on a presheaf topos with fibrant universes has the fibration extension property.  Conversely, if a premodel structure with the fibration extension property has universes, then those universes have fibrant base objects.
\end{lem}
\begin{proof}
We first show that fibrant universes imply the fibration extension property.
For any fibration $p \colon X \fto A$, we have a classifying universe $\pi \colon \dot{U} \fto U$ with fibrant base $U$.
In particular, this choice defines a classifying map and thus a lifting problem
\[
\begin{tikzcd}
  A
  \arrow[r, "\overline{p}"]
  \arrow[d, "j"', tcofarrow]
&
  U \rlap{,}
\\
  B \arrow[ur, dashed, "\overline{q}"']
\end{tikzcd}
\]
which admits a solution since $U$ is fibrant.
The pullback of $\pi$ along this map, displayed below-right, defines a small fibration over $B$.
The pullback square for $p$ factors through the one for $q$ defining the desired extension square:
\[
\begin{tikzcd}
  X
  \arrow[rr, bend left]
  \arrow[d, two heads, "p"']
  \arrow[r, dashed]
  \arrow[dr, phantom, "\lrcorner" very near start]
&
  Y
  \arrow[d, dashed, two heads, "q"']
  \arrow[r]
  \arrow[dr, phantom, "\lrcorner" very near start]
&
  \tilde{U}
  \arrow[d, two heads, "\pi"]
\\
  A
  \arrow[r, utcofarrow, "j"]
  \arrow[rr, bend right, "\overline{p}"']
&
  B
  \arrow[r, "\overline{q}"]
&
  U \rlap{.}
\end{tikzcd}
\]
Conversely, suppose the fibration extension property holds and consider a lifting problem into the base of one of the universal fibrations:
\[
\begin{tikzcd}
  A
  \arrow[r, "\overline{p}"]
  \arrow[d, "j"', tcofarrow]
&
  U \rlap{.}
\\
  B
  \arrow[ur, dashed, "\overline{q}"']
\end{tikzcd}
\]
Define a fibration $p \colon X \fto A$ by pulling back $\pi$ along $\overline{p}$.
Then use the fibration extension property to extend this to a fibration $q \colon Y \fto B$ that pulls back along $j$ to $p$.
As required by Definition \ref{defn:FEP}, this extended fibration is classified by the same universe.
Using the given universe and relative acyclicity of its associated notion of fibred structure, the classifying map $\overline{p}$ extends along $j$ to a classifying map $\overline{q}$ for $q$ so that $\overline{q} \cdot j = \overline{p}$, solving the lifting problem.
This proves that the fibration extension property implies fibrancy of the universe.
\end{proof}

We can now make use of the following result from \cite{CavalloSattler:2022re}, the proof of which is entirely axiomatic.
\begin{prop}[{\cite[3.31]{CavalloSattler:2022re}}]\label{prop:cylindrical-model-category}  Let $\cE$ be a cylindrical premodel category in which all objects are cofibrant.  If the fibration extension property holds, then the weak equivalences satisfy the 2-of-3 condition.
\end{prop}

Thus, the constructions of a model of homotopy type theory and of a Quillen model structure from a cylindrical premodel category with all objects cofibrant are intertwined.
First one checks the equivalence extension property, which is the heart of the interpretation of univalence.
Then one proves the Frobenius condition, which provides the interpretation of $\Pi$-types and is connected to right properness of the model structure.
The equivalence extension property and Frobenius condition may then also play a role in the construction of fibrant universes.
Besides interpreting the universes of the type theory, the fibrant universes can be used to derive the fibration extension property, which then yields the model structure.
In the sequel, we see two versions of this story, both showing that a cylindrical premodel structure is a model structure, first in cubical species and then in cubical sets.

\section{The interval model structure on cubical species}\label{sec:symmetric}

On a presheaf topos with a suitable interval object there is a now well-known strategy for defining a model structure that models homotopy type theory. The cofibrations are the monomorphisms, making the trivial fibrations those of Definition \ref{defn:trivial-fibrations}. The fibrations are then defined from the trivial fibrations as either the \emph{biased} or \emph{unbiased} fibrations of Definition \ref{defn:interval-fibrations}.\footnote{As noted in \cite[4.22--23]{CavalloSattler:2022re} and Proposition \ref{prop:kan-is-equivariant}, sometimes these classes coincide.} The results in the previous section then apply to establish the equivalence extension property, the Frobenius condition, the fibration extension property, the univalence and fibrancy of the universes, and verify the 2-of-3 condition for the weak equivalences.

Here we apply this outline not in the category of cubical sets but in the category of \emph{cubical species} introduced in \S\ref{ssec:cubical-species}, which has a suitable ``symmetric'' interval object. The category of cubical species is a category of groupoid-indexed functors valued in cubical sets, so in \S\ref{ssec:groupoid-diagrams} we first discuss some general results about subobject classifiers, pushforwards, and tiny objects that apply in that general setting. In \S\ref{ssec:species-cylindrical}, we establish the cylindrical premodel structure on cubical species. Then in \S\ref{ssec:species-model}, we apply the results from \S\ref{sec:cylindrical} to prove that this premodel structure is a model structure modeling homotopy type theory.

\subsection{Groupoid-indexed diagram categories}\label{ssec:groupoid-diagrams}

We collect some statements about diagram categories indexed by a groupoid. In fact, the first few results apply more generally to category-indexed diagrams.

\begin{lem}\label{cartesian-nat-trans:pushforward}
In a diagram category $\cE^\cC$ whose base category $\cE$ has pullbacks, consider a cartesian natural transformation $f \colon Y \to X$.
The family of evaluation functors $c^* \colon \cE^\cC \to \cE$ at objects $c \colon 1 \to C$ creates pushforward along $f$.
\end{lem}

\begin{proof}
The slice of $\cE^\cC$ over $X$ is the lax bilimit of the categories $\cE_{/X(c)}$ indexed over $c \in \cC$, with functorial action given by pullback, and similarly for $Y$.
For each $u \colon c \to d$ in $\cC$, there are canonical isomorphisms $f_c^* X_u^* \cong Y_u^* f_d^*$ satisfying coherence under pasting. Thus, the pullback functor $f^*\colon \cE_{/X} \to \cE_{/Y}$ is given by functoriality of lax bilimits from pullback along the components of $f$.

Since the naturality square of $f$ at $u$ is a pullback, the mate $(Y_u)_! f_c^* \to f_d^* (X_u)_!$ is invertible.
By adjointness, so is the mate $X_u^* (f_d)_* \to (f_c)_* Y_u^*$, assuming we have pushforward along the components of $f$.
Therefore, the pullback-pushforward adjunctions at each level assemble into an indexed adjunction.
By bifunctoriality of lax bilimits, this gives a right adjoint to pullback along $f$.
\end{proof}

\begin{lem}\label{cartesian-nat-trans:tiny}
In category of diagrams $\cE^\cC$ whose base category $\cE$ has binary products, consider a diagam $A$ with invertible functorial actions.
The family of evaluation functors $c^* \colon \cE^\cC \to \cE$ at objects $c \colon 1 \to C$ creates exponential with $A$ and its right adjoint.
\end{lem}

\begin{proof}
We argue similarly to the previous proof.
The product with $A$ is given bifunctorially from product with $A(c)$ at level $c \in \cC$ and invertibility of the map $(-) \times A(c) \to (-) \times A(d)$ for $u \colon c \to d$, using that $A_u$ is invertible.
Assuming levelwise exponentials, the induced map on right adjoints $(-)^{A(d)} \to (-)^{A(c)}$ is invertible.
Assuming further right adjoints $(-)^{A(c)} \dashv (-)_{A(c)}$ for $c \in \cC$, so is the induced map $(-)_{A(c)} \to (-)_{A(d)}$.
Bifunctoriality of lax bilimits gives the desired right adjoints $(-)^A$ and $(-)_A$.
\end{proof}

\begin{lem}\label{cartesian-nat-trans:subobject-classifier}
Consider a category $\cE$ with pullbacks and a subobject classifier $1 \to \Omega$, and the constant diagram functor $\Delta \colon \cE \to \cE^\cC$. Then $\Delta 1 \to \Delta\Omega$ classifies monomorphisms that define cartesian natural transformations in $\cE^\cC$.
\end{lem}

\begin{proof}
Note that cartesian natural transformations are closed under pullback and that the claimed classifier is one.
Given a cartesian natural transformation that is a componentwise monomorphism, its levelwise classifying squares assemble into a (unique) classifying square by pullback pasting and uniqueness of classification. Since $\cE$ has pullbacks, monomorphisms in $\cE^\cC$ are componentwise monomorphisms.
\end{proof}

For a groupoid $\cG$, every functor from $\cG$ to $\cE$ has invertible functorial action and every natural transformation between such functors is cartesian.
Therefore:

\begin{cor}\label{groupoid-functors:pushforward}
Consider a locally cartesian closed category $\cE$.
For each groupoid $\cG$, the functor category $\cE^\cG$ is locally cartesian closed.
For each functor $F \colon \cG \to \cH$ between groupoids, restriction $F^* \colon \cE^\cH \to \cE^\cG$ preserves pushforward.
\qed
\end{cor}

\begin{cor}\label{groupoid-functors:tiny}
Consider a cartesian closed category $\cE$.
For each groupoid $\cG$, an object $A \in \cE^\cC$ is tiny if it is componentwise tiny.
For each functor $F \colon \cG \to \cH$ between groupoids, restriction $F^* \colon \cE^\cH \to \cE^\cG$ preserves exponentiation with componentwise tiny objects.
\qed
\end{cor}

\begin{cor}\label{groupoid-functor:subobject-classifier}
Consider a finitely complete category $\cE$ with a subobject classifier.
For each groupoid $\cG$, the functor category $\cE^\cG$ has a subobject classifier.
For each functor $F \colon \cG \to \cH$ between groupoids, restriction $F^* \colon \cE^\cH \to \cE^\cG$ preserves subobject classifiers.
\qed
\end{cor}

\subsection{Cubical species and the symmetric interval}\label{ssec:cubical-species}

The ``cubical'' in the phrase \emph{cubical species} refers to the cartesian cube category, defined below.
In Buchholtz and Morehouse's taxonomy of cube categories \cite{BuchholtzMorehouse:2017vo}, this is $\mathbb{C}_{(\mathrm{wec},\cdot)}$.

\begin{defn}\label{defn:cartesian-cubes}
    The \textbf{cartesian cube category} $\CCube \coloneq \Fin_{\bot\neq\top}^\op$ is the opposite of the category of finite strictly bipointed sets and bipointed maps. Its objects are bipointed sets of the form $\{\bot,1,\ldots, n, \top\}$ for $n \geq 0$. We write $\cSet\coloneq\widehat{\CCube}$ for the topos of presheaves and call its objects \textbf{(cartesian) cubical sets}.  Under the Yoneda embedding $\yo \colon \CCube \to \cSet$, the object $\{\bot,1,\ldots, n, \top\}$ is identified with the $n$-cube $I^n$. By the Yoneda lemma, morphisms $\alpha \colon I^m \to I^n$ correspond to functions $\alpha\colon \{\bot,1,\ldots,n,\top\} \to \{\bot,1,\ldots,m,\top\}$ preserving the basepoints $\bot$ and $\top$.
\end{defn}

Let $\SSigma \cong \coprod_{k \geq 1} \SSigma_{k}$ be the maximal subgroupoid of the cube category $\CCube$ excluding, for reasons explained in Remark \ref{rmk:no-degree-zero}, the identity automorphism of the 0-cube. Here $\SSigma_k$ is the one-object groupoid associated to the symmetric group $\Sigma_k$, which acts on $\{\bot,1,\ldots,k,\top\}$ by permuting the indices and thus acts on the representable cubical set $I^k$ by permuting the dimensions.

\begin{defn}
  A \textbf{cubical species} is a set-valued functor on $\CCube^{\op} \times \SSigma$.
\end{defn}

It is convenient to represent a cubical species as a symmetric sequence of cubical sets, i.e., as a family $\mathbb{X} = (X^{k})_{k \geq 1}$ of cubical sets, in which each $X^{k}$ has a specified $\Sigma_{k}$-action. Indeed, as a category we have
\[
\Set^{\CCube^{\op} \times \SSigma}\ \cong\  \cSet^\SSigma\
\cong\ \prod_{k \geq 1} \cSet^{\SSigma_{k}}.
\]
A cubical species that is non-empty  in only a single factor $\cSet^{\SSigma_k}$ is said to be \textbf{concentrated in degree} $k$.

Write $\FF_{k} \colon \cSet \to \cSet^{\SSigma}$ for left Kan extension along $*_{k} \colon \1 \to \SSigma$, the left adjoint to the functor $U_k \colon \cSet^{\SSigma} \to \cSet$ which projects to the $k$th component of the cubical species and forgets the action:
\[
  \begin{tikzcd}
    \cSet \arrow[rr, phantom, "\bot"] \arrow[rr, bend left, "\FF_k"] && \cSet^{\SSigma} \arrow[ll, bend left, "U_k"] \rlap{.}
  \end{tikzcd}\]

\begin{defn} For $k \geq 1$, a \textbf{$k$-free} cubical species is a cubical species of the form $\FF_kX$ for $X \in \cSet$. Explicitly, the $k$-free cubical species $\FF_kX$ is concentrated in degree $k$ with free $\Sigma_{k}$-action on the cubical set $X \times \Sigma_k$.
\end{defn}

We highlight two particularly important examples of cubical species.

\begin{ex}\label{ex:representables}
  The representable cubical species
  \[
  \hom_{\CCube \times \SSigma^{\op}}\big(-, ([n], *_{k})\big),
  \]
   represented by the pair of objects $[n] = \{\bot, 1,\ldots, n ,\top\} \in \CCube$ and $*_{k} \in \SSigma$, is the free cubical species $\FF_kI^n$ concentrated in degree $k$ and given there by the cubical set $I^{n} \times \Sigma_{k}$ with the free $\Sigma_{k}$-action.
  \end{ex}

\begin{ex}\label{ex:interval}
  The restriction of the hom bifunctor $\hom \in \Set^{\CCube^{\op} \times \CCube}$ along the inclusion $\SSigma\hookrightarrow\CCube$ in the codomain variable defines a cubical species $\II$ whose $k$th component is the geometric $k$-cube $I^{k}$ with its \textbf{regular action}, permuting the $k$ dimensions.
\end{ex}

\begin{rmk}\label{rmk:interval-endpoints} The symmetric interval $\II$ has $2^{\omega}$ points $\1 \to \II$: for any countable sequence $\vec{v}$ of 0s and 1s there is a corresponding point $\vec{v} \colon \1 \to \II$ that chooses either the initial or final vertex in each component. Since the terminal cubical species $\1$ has a trivial action in each component, all points of the interval are fixed points for the coordinatewise actions of the symmetric groups.
\end{rmk}

\begin{lem}\label{lem:symmetric-interval-tiny} The cubical species $\II$ is tiny.
\end{lem}
\begin{proof}
Recall that $\II(c) = \CCube(-, c) \in \cSet$ is representable.
Since $\CCube$ has binary products, representables in $\cSet$ are tiny.
Now $\II$ is tiny by Corollary \ref{groupoid-functors:tiny}.
\end{proof}

\subsection{The cylindrical premodel structure on cubical species}\label{ssec:species-cylindrical}

We determine a pair of (algebraic) weak factorization systems that constitute a premodel structure on the cubical species and prove that it is cylindrical, with adjoint functorial cylinder represented by the interval object
\[ \begin{tikzcd} \1 \arrow[r, shift left, "\delta_0"] \arrow[r, shift right, "\delta_1"'] & \II \arrow[r, "!"] & \1 \end{tikzcd}
\]
where the points $\delta_0 , \delta_1$ correspond to the constant sequences $\vec{0},\vec{1}$ of Remark \ref{rmk:interval-endpoints}.

As a presheaf topos, the category $\cSet^{\SSigma}$ has a subobject classifier $ \top \colon \1 \rightarrowtail \OOmega$, which we can describe explicitly as follows.

\begin{lem}\label{lem:symmetric-subobjects} For $n,k \in \NN$, $k \geq 1$, elements $\chi_c \colon \FF_kI^n \to \OOmega$ of the subobject classifier correspond bijectively to subobjects $c \colon C \rightarrowtail I^n$ of the $n$-cube.
\end{lem}
\begin{proof}
By definition, an element $\chi_c \colon \FF_k I^n \to  \OOmega$ corresponds to a subobject of the representable cubical species $\FF_kI^n$. Since $\FF_kI^n$ is concentrated in degree $k$ and has a free $\Sigma_k$-action, its subobject must have these properties as well. Thus, we see that the subobject has the form $\FF_k c \colon \FF_kC \rightarrowtail \FF_kI^n$ for a necessarily unique subobject $c \colon C \rightarrowtail I^n$ of the $n$-cube.
\end{proof}

\begin{defn}\label{def:coftrivfib} As the \textbf{cofibrations} we take the monomorphisms, which are classified (up to equivalence) by the subobject classifier $ \top \colon \1 \rightarrowtail \OOmega$. The \textbf{trivial fibrations} are then the maps with the right lifting property against all monomorphisms.
\end{defn}

As we saw in \S\ref{ssec:trivial-fibrations}, the cofibrations and trivial fibrations form a weak factorization system.
By Lemma \ref{lem:loc-rep-triv-fib}, we can recognize the trivial fibrations as the class underlying a locally representable and relatively acyclic notion of fibred structure $\TF$.

We now turn to the (trivial cofibration, fibration) weak factorization system. The fibrations will be the unbiased fibrations of Definition \ref{defn:interval-fibrations}\eqref{itm:unbiased-fibrations}---see Theorem \ref{thm:uniform-fibrations}---which we now describe explicitly. The fibrations will be determined by the trivial fibrations, by Leibniz pullback application of the evaluation natural transformation $\ev \colon (-)^\II \times \II \To (-)$ involving the interval $\II$.
Equivalently, we may describe them as given by right lifting against a category of generating trivial cofibrations constructed using the universal subobject $\top \colon \1 \to \OOmega$ and the ``generic point'' $\delta \colon \II \to \II \times \II$---see Definition \ref{defn:species-trivial-cofibration-fibration}.
With the latter description, we can obtain a functorial factorization (indeed, an awfs) constructively using Garner's algebraic small object argument.

\begin{defn} As a map in the slice category $\cSet^{\SSigma}_{/\II}$, the diagonal $\delta \colon \II \to \II \times \II$ defines an additional point of $\II$, called the \textbf{generic point}.
\end{defn}

The morphisms $\top \colon \1 \to \OOmega$ in $\cSet^{\SSigma}_{/\OOmega}$ and $\delta \colon \II \to \II \times \II$ in $\cSet^{\SSigma}_{/\II}$ can be reindexed to lie in the common slice $\cSet^{\SSigma}_{/\OOmega \times \II}$. Their pushout product there defines a family of maps $\top\hat\times_{\OOmega\times\II}\delta$ internally indexed by the object $\OOmega \times \II$:
\[
  \begin{tikzcd}[column sep=small] & \II \arrow[dd, phantom, "\rotatebox{135}{$\lrcorner$}" very near end]\arrow[dl, "\top \times \II"'] \arrow[dr, "\delta"] \\
   \OOmega \times\II \arrow[dr, dotted] \arrow[ddr, "\OOmega \times \delta"'] & &  \II \times \II  \arrow[dl, dotted] \arrow[ddl, "\top \times \II \times \II"]\\
    [-12pt] &  \OOmega \times \II \cup_\II \II \times \II \arrow[d, dashed, "\top \hat{\times}_{\OOmega \times \II}\delta" description] & \\
     [+24pt] & \OOmega \times \II \times \II \arrow[d, "\pi"] &\\
      & \OOmega \times \II \rlap{.} &
        \end{tikzcd}
\]
Our category of generating trivial cofibrations will be given by externalizing the family $\top\hat\times_{\OOmega\times\II}\delta$ and will therefore be indexed by the category of elements of $\OOmega\times\II$.

\begin{rmk}\label{rmk:product-cat-of-elements} Since in general $\int_{\cX} 1 \cong \cX$, and the category of elements functor $\int$ preserves pullbacks, the category of elements of a product is the pullback of the categories of elements:
  \[
    \begin{tikzcd}
      \int\OOmega \times \II \arrow[r] \arrow[d] \arrow[dr, phantom, "\lrcorner" very near start] & \int \OOmega \arrow[d] \\ \int \II \arrow[r] & \CCube \times \SSigma^\op \rlap{.}
    \end{tikzcd}
   \]
   Now $\II$ is a restriction of the hom bifunctor, so its category of elements is a restriction of the twisted arrow category. Thus, the objects of $\int\OOmega \times \II$ are pairs $(c,\zeta)$ as displayed vertically below while $(\alpha,\sigma) \colon (d,\xi) \to (c,\zeta)$ defines a morphism just when the displayed diagram of cubical sets commutes, and the top square is a pullback:
   \begin{equation}\label{eq:generating-tcof-indexing-map}
     \begin{tikzcd} D \arrow[r, "\alpha"] \arrow[d, "d"', tail] \arrow[dr, phantom, "\lrcorner" very near start] & C \arrow[d, "c", tail] \\ I^m \arrow[r, "\alpha"] \arrow[d, "\xi"'] & I^n \arrow[d, "\zeta"] \\ I^k & I^k \arrow[l, "\sigma"] \rlap{.}
     \end{tikzcd}
    \end{equation}
\end{rmk}

As observed in Remark \ref{rmk:product-cat-of-elements}, the elements of $\OOmega \times \II$ stand in bijection with maps $(\chi_c,\zeta) \colon \FF_kI^n \to \OOmega \times \II$ where $\chi_c \colon \FF_kI^n \to \OOmega$ classifies a subobject $c \colon C \rightarrowtail I^n$ of the cubical set $I^n$ and $\zeta \colon \FF_k I^n \to \II$, by adjunction, corresponds to a map $\zeta \colon I^n \to U_k\II \cong I^k$ in $\CCube$. Thus, we regard the objects in $\int\OOmega \times \II$ as composable pairs of cubical set morphisms
\[
\begin{tikzcd} C \arrow[rr, "c", tail] & & I^n \arrow[dl, "\zeta"]\\  & I^k \rlap{,}
\end{tikzcd}
\]
which we call \textbf{triangles}.

\begin{con}\label{con:generating-tcof}
  The family of maps $\top\hat\times_{\OOmega \times \II}\delta$ internally indexed by the object $\OOmega \times \II$ can be externalized to define a functor $J \colon \int\OOmega\times\II \to (\cSet^{\SSigma})^\2$ externally indexed by the category of elements of $\OOmega \times \II$ and defined by pulling back the given internal family of maps to representables.
The cartesian functor $J$ lifts the Yoneda embedding $\yo$ from the discrete fibration associated to the category of elements of the functor $\OOmega \times \II$ to the codomain fibration:
  \[
    \begin{tikzcd}  \int\OOmega\times \II  \arrow[d, "\pi"'] \arrow[r, dashed, "J"] & (\cSet^{\SSigma})^\2 \arrow[d, "\cod"] \\  \CCube \times \SSigma^\op \arrow[r, hook, "\yo"] & \cSet^{\SSigma} \rlap{.}
    \end{tikzcd}
  \]
  Explicitly, the functor $J$ sends an element $(c,\zeta)$ to the pullback along it of the universal element $\top \hat\times\delta$, as indicated below:
    \[
    \begin{tikzcd}[column sep=small]  & \FF_k C \arrow[dl, "\FF_kc"'] \arrow[dr, "{(\FF_kC,\zeta\cdot\FF_kc)}"]  \arrow[dd, phantom, "\rotatebox{135}{$\lrcorner$}" very near end]&& & \II \arrow[dd, phantom, "\rotatebox{135}{$\lrcorner$}" very near end]\arrow[dl, "\top \times \II"'] \arrow[dr, "\delta"] \\ \FF_k I^n \arrow[dr, dotted] \arrow[ddr, "{(\FF_kI^n,\zeta)}"'] && \FF_k C \times \II \arrow[dl, dotted] \arrow[ddl, "\FF_kc \times \II"] &\OOmega \times\II \arrow[dr, dotted] \arrow[ddr, "\OOmega \times \delta"'] & &  \II \times \II  \arrow[dl, dotted] \arrow[ddl, "\top \times \II \times \II"]\\ [-12pt] & \FF_kI^n \cup_{\FF_kC} \FF_kC \times \II \arrow[d, dashed]  \arrow[rrr, dotted] \arrow[drrr, phantom, "\lrcorner\qquad\qquad" pos=.0001] && &  \OOmega \times \II \cup_\II \II \times \II \arrow[d, dashed, "\top \hat\times\delta" description] & \\ [+2pt] & \FF_k I^n \times \II \arrow[rrr, dotted] \arrow[drrr, phantom, "\lrcorner\qquad\qquad" very near start] \arrow[d, "\pi"] && & \OOmega \times \II \times \II \arrow[d, "\pi"] \\ &\FF_k I^n \arrow[rrr, "{(\chi_c,\zeta)}"'] && & \OOmega \times \II \rlap{.}
          \end{tikzcd}
  \]
The resulting map $J(c,\zeta) = (\chi_c,\zeta)^*(\top \hat\times\delta)$ can also be computed as the pushout product of the subobject $\FF_kc \colon \FF_kC \rightarrowtail \FF_kI^n$ and the generic point $\delta \colon \II \to \II \times \II$ regarded as maps in the slice over $\II$ via $\zeta \colon \FF_kI^n \to \II$ and $\pi \colon \II \times \II \to \II$.

  Note the map $\delta$ pulls back along $(\chi_c,\zeta)$ to define the \textbf{graph} $(\FF_kC,\zeta \cdot \FF_kc) \colon \FF_kC \to \FF_kC \times \II$ of $\zeta \cdot \FF_kc \colon \FF_kC \to \II$ and similarly $\OOmega\times\delta$ pulls back to define the graph of $\zeta \colon \FF_kI^n \to \II$.  Henceforth, for any map $\gamma \colon \AA \to \BB$, we shall write $[\gamma] \colon \AA \to \AA\times \BB$ for its graph $(\AA, \gamma)$.

  Morphisms in $\int\OOmega\times\II$
  \[
    \begin{tikzcd}
      \FF_kI^m \arrow[rr, "\alpha \times \sigma"] \arrow[dr, "{(\chi_d,\xi)}"'] & & \FF_k I^n \arrow[dl, "{(\chi_c,\zeta)}"] \\ & \OOmega \times \II
   \end{tikzcd}
  \]
  correspond to pairs $\alpha \colon I^m \to I^n$ and $\sigma \in \Sigma_k$ as in \eqref{eq:generating-tcof-indexing-map}. The functor $J$ carries such a morphism to the following pullback square of cubical species:
  \begin{equation}\label{eq:generating-tcof-morphism}
    \begin{tikzcd}
      \FF_k I^m \cup_{\FF_kD} \FF_kD \times \II \arrow[d, tail, "{\langle [\xi], \FF_kd \times 1\rangle}"'] \arrow[r, "\alpha \times \sigma \times 1"] \arrow[dr, phantom, "\lrcorner" very near start] & \FF_k I^n \cup_{\FF_kC} \FF_kC \times \II \arrow[d, tail, "{\langle [\zeta], \FF_kc \times 1\rangle}"] \\ \FF_k I^m \times \II \arrow[r, "\alpha \times \sigma \times 1"'] & \FF_k I^n \times \II \rlap{.}
    \end{tikzcd}
  \end{equation}
\end{con}

We refer to the subobjects in the image of the functor $J$ as \textbf{open boxes}, though the nature of the gluing of the ``lid'' $\FF_k I^n$ onto the ``box'' $\FF_k C \times \II$ is somewhat subtle because it involves the map $\zeta \colon \FF_kI^n \to \II$. The open boxes are themselves pushout products on account of the following general lemma.

\begin{lem} If $i$ is a morphism in the slice over $\XX$ and $j$ is a morphism in the slice over $\YY$ and $(x,y) \colon \ZZ \to \XX \times \YY$, then the pushout product of $i$ and $j$ in the slice over $\XX \times \YY$ pulls back along $(x,y)$ to the map over $\ZZ$ obtained as the pushout product over $\ZZ$ of the evident pullbacks of $i$ and $j$.
\end{lem}
\begin{proof}
Pushout products in slices are stable under pullback.
\end{proof}

\begin{cor} The open box
  \[\begin{tikzcd}[column sep=large]
    \FF_kI^n \cup_{\FF_kC} \FF_kC \times \II \arrow[r, "{\langle [\zeta], \FF_kc \times 1\rangle}"] & \FF_kI^n \times \II
  \end{tikzcd}
  \] is the pushout product over $\FF_kI^n$ of the maps obtained by pullback
  \[
  \begin{tikzcd}[column sep=small,baseline=(current bounding box.south)]
    \FF_kC \arrow[dr, dashed, "\FF_kc"'] \arrow[dd, "\FF_kc"', dotted] \arrow[rr, dotted] & & 1 \arrow[dr, "\top"] \arrow[dd, "\top"'] & & \FF_kI^n \arrow[dr, dashed, "{[\zeta]}"'] \arrow[dd, dotted, equals] \arrow[rr, dotted] & [-10pt] ~ &  \II \arrow[dr, "\delta"] \arrow[dd, equals] \\ & \FF_kI^n \arrow[dl, dotted, equals] \arrow[rr, dotted] & & \OOmega \arrow[dl, equals] & & \FF_kI^n \times \II \arrow[rr, dotted] \arrow[dl, "\pi", dotted] & & \II \times \II \arrow[dl, "\pi"] \\ \FF_kI^n \arrow[rr, "\chi_c"'] & & \OOmega & & \FF_kI^n \arrow[rr, "\zeta"'] & & \II \rlap{.}
  \end{tikzcd} \qed
  \]
\end{cor}

\begin{rmk}
Since the representables are concentrated in a single degree, each open box is as well. The ``triangle'' of cubical sets as below-left---where the first map is a morphism and the second map is between representables---gives rise to the ``open-box'' of cubical species as below-center, concentrated in degree $k$:
\[
\begin{tikzcd} C \arrow[rr, tail, "c"]  & & I^n \arrow[dl, "\zeta"] & \arrow[d, phantom, "\rightsquigarrow"] & \FF_k I^n \cup_{\FF_kC} \FF_kC \times \II \arrow[d, tail, "{\langle [\zeta], \FF_kc \times 1\rangle}"] & ~ & \Sigma_k \times I^n  \cup_{\Sigma_k \times C} \Sigma_k \times C \times  I^k \arrow[d, tail, "{\langle [\zeta^{\Sigma_k}],  1 \times c \times 1\rangle}"]  \\ & I^k & & ~ &  \FF_k I^n \times \II & \arrow[u, phantom, "\leftrightsquigarrow"]  &  \Sigma_k \times I^n \times I^k \rlap{.}
\end{tikzcd}
\]
The non-empty component of this map is the map of $\Sigma_k$-cubical sets above-right, defined by the pushout below:
  \[
    \begin{tikzcd}[column sep=small]  &  C \times \Sigma_k  \arrow[dl,  tail, "c \times 1"'] \arrow[dr, "{[(\zeta \cdot c)^{\Sigma_k}]}"]  \arrow[dd, phantom, "\rotatebox{135}{$\lrcorner$}" very near end] \\ I^n \times \Sigma_k \arrow[dr, dotted] \arrow[ddr, "{[\zeta^{\Sigma_k}]}"'] &&   C\times \Sigma_k \times I^k \arrow[dl, dotted, tail] \arrow[ddl, tail, "c \times 1"] \\ [-12pt] & \bullet \arrow[d, dashed]   \\ [+2pt] &  I^n \times \Sigma_k \times I^k \rlap{.}
          \end{tikzcd}
  \]
  Here the action of $\Sigma_k$ is trivial on $C$ and $I^n$; by left multiplication on $\Sigma_k$; and by permuting the dimensions on $I^k$---the ``regular'' action.  The map $[\zeta^{\Sigma_k}] \colon I^n \times \Sigma_k  \to I^n \times\Sigma_k \times  I^k$ is the graph of a twisted version of $\zeta$: the map $\zeta^{\Sigma_k} \colon I^n \times \Sigma_k\to I^k$ acts on the component of the domain coproduct indexed by $\sigma \in \Sigma_k$ by $\sigma \cdot \zeta \colon I^n \to I^k$. The top-right map is defined similarly. Note the maps in the pushout diagram are all $\Sigma_k$-equivariant, as required.

Similarly, the pullback square \eqref{eq:generating-tcof-morphism} is concentrated in degree $k$ and has the form
  \[
    \begin{tikzcd}
      I^m \times \Sigma_k \cup_{D \times \Sigma_k} D \times \Sigma_k \times I^k
     \arrow[d, tail, "{\langle [\xi^{\Sigma_k}], d \times 1\rangle}"'] \arrow[r, "\alpha \times \sigma \times 1"] \arrow[dr, phantom, "\lrcorner" very near start] & I^n \times \Sigma_k \cup_{C \times \Sigma_k} C \times \Sigma_k \times I^k \arrow[d, tail, "{\langle [\zeta^{\Sigma_k}], c \times 1\rangle}"] \\ I^m \times \Sigma_k \times I^k \arrow[r, "\alpha \times \sigma \times 1"'] &  I^n \times \Sigma_k \times I^k
    \end{tikzcd}
  \]
where $\sigma \colon \Sigma_k \to \Sigma_k$ is defined by right multiplication.  Note these definitions make the map $\alpha \times \sigma \times 1 : I^m \times \Sigma_k \times I^k \to I^n \times \Sigma_k \times I^k$ into a $\Sigma_k$-equivariant map.
\end{rmk}

\begin{defn}
  \label{defn:species-trivial-cofibration-fibration}
  Garner's algebraic small object argument \cite{Garner:USOA} yields an algebraic weak factorization system on $\cSet^\SSigma$ which is algebraically free on $J \colon \int\OOmega\times\II \to (\cSet^\SSigma)^\2$, i.e., whose category of monad algebras is given by $(\int\OOmega \times \II)^\boxslash$.
  In particular, a right map is a morphism $f \colon \YY \to \XX$ of cubical species equipped with chosen lifts against open boxes that are uniform in pullback squares:
  \[
    \begin{tikzcd}[column sep=large, row sep=1.5cm]
      \FF_k I^m \cup_{\FF_kD} \FF_kD \times \II \arrow[d, tail, "{\langle [\sigma\zeta\alpha], \FF_kd \times 1\rangle}"'] \arrow[r, "\alpha \times \sigma \times 1"] \arrow[dr, phantom, "\lrcorner" very near start] & \FF_k I^n \cup_{\FF_kC} \FF_kC \times \II \arrow[d, tail, "{\langle [\zeta], \FF_kc \times 1\rangle}" description, pos=.6] \arrow[r] & \YY \arrow[d, "f"] \\ \FF_k I^m \times \II \arrow[urr, dashed] \arrow[r, "\alpha \times \sigma \times 1"'] & \FF_k I^n \times \II \arrow[r] \arrow[ur, dashed] & \XX \rlap{.}
    \end{tikzcd}
  \]

  We call the left and right classes of the underlying weak factorization system the \textbf{trivial cofibrations} and \textbf{fibrations} respectively.
\end{defn}

We now show that these fibrations are the unbiased fibrations.

\begin{defn}\label{defn:parametrized-path-space}
  Given a map $f \colon \YY \to \XX$ define the \textbf{parametrized path space} by forming the Leibniz exponential of $f$ with $\delta$ in the slice over $\II$, as displayed below-left:
  \begin{equation}\label{eq:parametrized-path-space}
    \begin{tikzcd}
      & \YY^\II \times \II \arrow[ddl, "f^\II \times \II"'] \arrow[ddr, "{(\ev,\pi)}"] \arrow[d, dashed, "\ev\hato  f" description]& & &       & \YY^\II \times \II \arrow[ddl, "f^\II \times \II"'] \arrow[ddr, "{\ev}"] \arrow[d, dashed, "\ev\hato  f" description]\\ [+2pt] & \PP^\II \YY \arrow[dr, dotted] \arrow[dl, dotted] \arrow[dd, phantom, "\rotatebox{135}{$\ulcorner$}" very near start] & & &  & \PP^\II \YY \arrow[dr, dotted] \arrow[dl, dotted] \arrow[dd, phantom, "\rotatebox{135}{$\ulcorner$}" very near start] \\[-12pt] \XX^\II \times \II \arrow[dr, "{(\ev,\pi)}"'] & & \YY \times \II \arrow[dl, "f \times \II"] & & [-12pt] \XX^\II \times \II \arrow[dr, "{\ev}"'] & &  [+5pt] \YY \arrow[dl, "f"] \\ & \XX \times \II & & &   & \XX
    \end{tikzcd}
  \end{equation}
  where $\ev \colon \YY^\II \times \II \to \YY$ is evaluation. Equivalently, the map $\ev\hato  f$ may be defined by the pullback above-right, which is not formed in the slice over $\II$.
\end{defn}

From the second of these characterizations, $\ev\hato f$ is the Leibniz pullback application of the evaluation natural transformation to the map $f$, explaining our notation.
This functor is not right adjoint, failing to preserve the terminal object.
However, from the decomposition
  \[
    \begin{tikzcd}
     \arrow[rrr, bend right=15, "f \mapsto \ev\hato  f"'] (\cSet^{\SSigma})^\2 \arrow[r, "-\times \II"] & (\cSet^{\SSigma}_{/\II})^\2 \arrow[r, "\widehat{\{\delta,-\}}_\II"] & (\cSet^{\SSigma}_{/\II})^\2 \arrow[r, "\Sigma"] & (\cSet^{\SSigma})^\2 \rlap{,}
    \end{tikzcd}
    \]
  it is the composition of a right adjoint with the forgetful functor $\Sigma$. In particular, it preserves pullbacks.

\begin{thm}\label{thm:uniform-fibrations}
The category of uniform fibrations $(\int\OOmega\times \II)^\boxslash$ is the pullback of the category of uniform trivial fibrations $(\int\OOmega)^\boxslash$ along the parametrized path space functor:
\[ \begin{tikzcd} (\int\OOmega\times \II)^\boxslash \arrow[d] \arrow[r] \arrow[dr, phantom, "\lrcorner" very near start] & (\int\OOmega)^\boxslash \arrow[d] \\  (\cSet^{\SSigma})^\2 \arrow[r, "{\ev\hato  -}"'] & (\cSet^{\SSigma})^\2 \rlap{.} \end{tikzcd}\]
In particular, a map $f \colon \YY \to \XX$ of cubical species is a fibration if and only if it is an unbiased fibration, i.e., the parametrized path space map
  \[
    \begin{tikzcd}
      \YY^\II \times \II \arrow[r, "\ev\hato  f"] & \PP^\II \YY
    \end{tikzcd}
      \]
is a trivial fibration.
\end{thm}

\begin{proof}
The category of uniform fibrations is defined by right lifting against the category of arrows $J \colon \int\OOmega\times\II \to (\cSet^{\SSigma})^\2$ defined in Construction \ref{con:generating-tcof}. In terms of the functor $I \colon \int\OOmega  \to (\cSet^{\SSigma})^\2$ of Construction \ref{con:generating-cof}, the functor $J$ is the top horizontal composite:
\[ \begin{tikzcd} \int\OOmega\times\II \arrow[r, "\Sigma^*I"] \arrow[d] \arrow[dr, phantom, "\lrcorner" very near start] & (\cSet^\SSigma_{/\II})^\2 \arrow[r, "-\hat\times\delta"] \arrow[d, "\Sigma"] & (\cSet^\SSigma_{/\II})^\2 \arrow[r, "\Sigma"] \arrow[l, bend left, dashed, "\widehat{\{\delta,-\}}_\II", "\bot"'] & (\cSet^{\SSigma})^\2  \arrow[l, bend left, dashed, "-\times \II", "\bot"']\\
  \int\OOmega \arrow[r, "I"'] & (\cSet^{\SSigma})^\2
\end{tikzcd}\]
Thus, by adjunction, $f \in (\cSet^{\SSigma})^\2$ is a uniform fibration if and only if $\widehat{\{\delta,f\times \II\}}_\II \in  (\cSet^\SSigma_{/\II})^\2$ lifts on the right against the category $\Sigma^*I \colon \int\OOmega \times \II \to  (\cSet^\SSigma_{/\II})^\2$. As solutions to lifting problems in slice categories are created by the forgetful functor, this is the case if and only if $\ev\hato f \cong \Sigma \widehat{\{\delta,f\times \II\}}_\II \in  (\cSet^{\SSigma})^\2$ is a uniform trivial fibration as claimed.
\end{proof}

The left maps of an algebraic weak factorization system satisfy additional closure properties, arising from the fact that comonadic functors create colimits \cite{BourkeGarner:AWFSI}. In particular, colimits in the arrow category, of diagrams that factor through the generating category, are trivial cofibrations. The following lemma provides an example of this paradigm.

\begin{lem}\label{lem:interval-endpoints} For any of the $2^\omega$ points $\vec{\epsilon}$ of $\II$, the map $\vec{\epsilon} \colon \1 \to \II$, is a trivial cofibration.
\end{lem}
\begin{proof}
  For any vertex $\vec{v} \in I^k$ we have a triangle
  \begin{equation}\label{eq:point-generating-tcofs}\begin{tikzcd} \emptyset \arrow[rr, tail, "!"] & & 1 \arrow[dl, "\vec{v}"] & \arrow[d, phantom, "\rightsquigarrow"] & \FF_k 1 \arrow[d, tail, "{[\vec{v}]}"] & ~ & \Sigma_k  \arrow[d, tail, "{\vec{v}^{\Sigma_k}}"]  \\ & I^k & & ~ &   \FF_k 1 \times \II & \arrow[u, phantom, "\leftrightsquigarrow"]  &  \Sigma_k \times I^k
  \end{tikzcd}\end{equation}
  The map of $\Sigma_k$-cubical sets on the right sends $\sigma \in \Sigma_k$ to the pair $(\sigma,\sigma \cdot{\vec{v}})$. However, recall from Remark \ref{rmk:interval-endpoints} that a point of $\II$ is specified by choosing either point $\vec{0} \colon 1 \to I^k$ or $\vec{1} \colon 1 \to I^k$ for each component. Note these are the only two points in the $\Sigma_k$-cubical set $I^k$, since the other points in the underlying cubical set are permuted by the regular action. By contrast, since these points are fixed we have automorphisms
   \[
     \begin{tikzcd} \emptyset \arrow[r, equals] \arrow[d, "!"', tail] \arrow[dr, phantom, "\lrcorner" very near start] & \emptyset \arrow[d, "!", tail]  & & \emptyset \arrow[r, equals] \arrow[d, "!"', tail] \arrow[dr, phantom, "\lrcorner" very near start] & \emptyset \arrow[d, "!", tail] \\ 1 \arrow[r, equals] \arrow[d, "\vec{0}"'] & 1 \arrow[d, "\vec{0}"] & &  1 \arrow[r, equals] \arrow[d, "\vec{1}"'] & 1 \arrow[d, "\vec{1}"]\\ I^k & I^k \arrow[l, "\sigma"] & & I^k & I^k \arrow[l, "\sigma"]
     \end{tikzcd}
  \]
  for each $\sigma \in \Sigma_k$. Thus $\Sigma_k^\op$ acts on the open boxes $[\vec{0}] \colon \FF_k 1\rightarrowtail \FF_k 1 \times \II$ and $[\vec{1}] \colon \FF_k 1\rightarrowtail \FF_k 1 \times \II$ and these automorphisms lie in the generating category. The colimits yield the maps $\vec{0} \colon 1 \to I^k$ and $\vec{1} \colon 1 \to I^k$ in $\Sigma_k$-cubical sets, where the codomains have the regular action. Thus, these maps are trivial cofibrations. Picking the appropriate trivial cofibration in each component and forming their coproduct in cubical species yields the point inclusion $\vec{v} \colon 1 \to \II$ in $\cSet^{\SSigma}$.
\end{proof}

We have defined (cofibration, trivial fibration) and (trivial cofibration, fibration) algebraic weak factorization systems, each with an explicit category of generators. The trivial fibrations lift naturally against the generating category for the (trivial cofibration, fibration) awfs by Proposition \ref{prop:uniform-trivial-fibrations}, so trivial fibrations are fibrations and trivial cofibrations are cofibrations. The underlying weak factorization systems thus equip the category of cubical species with a premodel structure to be called the \textbf{interval premodel structure}. As in \S\ref{ssec:premodel}, we define the \textbf{weak equivalences} of cubical species to be those maps that factor as trivial cofibrations followed by trivial fibrations.

\begin{rmk}\label{rmk:no-degree-zero}
We would have a similar result if we had included the identity automorphism of the 0-cube in our definition of $\SSigma$, adding a $k=0$ component to our cubical species. Had we done so, then note that  in the $k=0$ component, all maps would be fibrations, since the components of the exterior squares of \eqref{eq:parametrized-path-space} are both pullbacks. Consequently, in the $k=0$ component, the only trivial cofibrations would be the isomorphisms, which means that the class of weak equivalences would coincide with the class of trivial fibrations, defined as in the other components to be those maps that lift against monomorphisms. But this class evidently fails to satisfy the 2-of-3 property, failing to be closed under left cancelation, so had we included a $k=0$ component our premodel structure would have no chance of defining a model structure.
However, the premodel structure would still suffice to define the model structure on equivariant cubical sets in Section \ref{ssec:species-to-sets}.
\end{rmk}

We next verify that the interval premodel structure is cartesian monoidal. We expect that this property can be made structural: that the cartesian closed structure on the category of cubical species defines two variable adjunctions of algebraic weak factorization systems \cite{Riehl:2013ma}, but as we have no application for that result, we decline to pursue it here.

\begin{prop}\label{prop:sm7}
  Pushout products of cofibrations are cofibrations, while the pushout product of a cofibration and a trivial cofibration is a trivial cofibration.
\end{prop}
\begin{proof}
  As the cofibrations are the monomorphisms in a presheaf category, the first property holds by  Remark \ref{rmk:adhesive-pushout-products}.

  The remaining statement is equivalent to the assertion that the Leibniz exponential $\widehat{\{c,f\}}$ of a fibration $f \colon \YY \to \XX$ and a monomorphism $c \colon \CC \rightarrowtail \ZZ$ is a fibration. By Theorem \ref{thm:uniform-fibrations}, this is equivalent to the assertion that the Leibniz exponential in the slice over $\II$ of $\delta \colon \II \to \II \times \II$ and $\widehat{\{c,f\}} \times \II$ is a trivial fibration, lifting against all monomorphisms $u \colon \JJ \to \KK$ in the slice over $\II$. Since the pullback of $\widehat{\{c,f\}}$ to the slice over $\II$ is isomorphic to the Leibniz exponential in the slice over $\II$ of the pullbacks $c\times \II$ and $f\times \II$, we are equivalently looking to solve lifting problems in the slice over $\II$ between the Leibniz product of $c \times \II$ and $u$ in the slice over $\II$ and the Leibniz exponential
  \[\ev\hato  f \coloneq \widehat{\{\delta, f \times \II\}}_\II\,.\]
  As we are working under the hypothesis that $f$ is a fibration, $\ev\hato  f$ is a trivial fibration so it suffices to verify that the pushout product of the monomorphisms $c \times \II$ and $u$ over $\II$ is a monomorphism. This again holds by Remark \ref{rmk:adhesive-pushout-products}.
\end{proof}

Finally, we observe that the interval premodel structure is cylindrical, satisfying the axioms of Definition \ref{defn:cylindrical}, using the adjunction $(-)\times \II \dashv (-)^\II$ to define an adjoint functorial cylinder.

\begin{lem}\label{lem:symmetric-cylindrical}
The interval premodel structure on cubical species is cylindrical.
\end{lem}
\begin{proof}
Since the endpoints $\vec{0}$ and $\vec{1}$ of the interval $\II$ are disjoint, the copairing $[\delta_0, \delta_1] \colon \1 + \1 \cto \II$ is a monomorphism and thus a cofibration. By Lemma \ref{lem:interval-endpoints}, the single endpoint inclusions $\delta_0, \delta_1 \colon \1 \cwto \II$ are trivial cofibrations. Now the result follows from Proposition \ref{prop:sm7}.
\end{proof}

\subsection{The cubical species model of homotopy type theory}\label{ssec:species-model}

In this section, we apply the results of \S\ref{sec:cylindrical} to verify the type-theoretic properties of the interval premodel structure on cubical species that allow us to show it is a Quillen model structure with the extra features required of a model of homotopy type theory.

The cofibrations in the interval premodel structure are exactly the monomorphisms, which are closed under pushout products in all slices by Remark \ref{rmk:adhesive-pushout-products}. Together with Lemma \ref{lem:symmetric-cylindrical}, this verifies the hypotheses of Theorem \ref{thm:cylindrical-EEP}, and therefore:

\begin{prop}\label{prop:symmetric-EEP}
  The interval premodel structure on cubical species satisfies the equivalence extension property.\qed
\end{prop}

Similarly, the definition of the fibrations is of the form considered by Proposition \ref{prop:generic-frobenius}, and therefore:

\begin{prop}\label{prop:symmetric-frobenius}
  The interval premodel structure on cubical species has the Frobenius property.\qed
\end{prop}

The remaining properties require universes, which we now construct. By Theorem \ref{thm:uniform-fibrations}, the uniform fibrations are determined as a certain pullback of the trivial fibrations. We use this result to define a notion of fibred structure $\FF$ that is locally representable and relatively acyclic and classifies the uniform fibrations.

\begin{lem}\label{lem:loc-rep-fib} There is a locally representable and relatively acyclic notion of fibred structure $\FF$, the notion of \textbf{uniform fibration structure}, whose underlying class of maps is the class of fibrations.
\end{lem}
\begin{proof}
  We apply Lemma \ref{lem:leibniz-pullback-application-loc-rep-fibred-structure}.
  That is, we define a uniform fibration structure on $f \colon \YY \to \XX$ to be a uniform trivial fibration structure on $\ev\hato  f$, the Leibniz pullback application of the evaluation natural transformation
  \[ \begin{tikzcd} \cSet^{\SSigma} \arrow[r, bend left, "{(-)^\II \times \II}"] \arrow[r, phantom, "\Downarrow\ev"]\arrow[r, bend right, equals] & \cSet^{\SSigma} \rlap{.} \end{tikzcd}
  \] Since the interval $\II$ is tiny, the functor $\XX \mapsto \XX^\II \times \II$ has a right adjoint:
  \[
    \begin{tikzcd} \cSet^{\SSigma} \arrow[r, bend left, "{(-)^\II}"] \arrow[r, phantom, "\bot"] & \cSet^{\SSigma} \arrow[l, bend left, "{(-)_{\II}}"] \arrow[r, bend left, "{-\times \II}"] \arrow[r, phantom, "\bot"] & \cSet^{\SSigma} \arrow[l, bend left, "{(-)^\II}"] \rlap{.}
  \end{tikzcd}
  \]
Since Lemma \ref{lem:loc-rep-triv-fib} tells us that the notion of fibred structure $\TF$ is locally representable and relatively acyclic, Lemma \ref{lem:leibniz-pullback-application-loc-rep-fibred-structure} tells  us that the same is true for the uniform fibrations.
\end{proof}

Instantiating Construction \ref{con:loc-rep-universe}:

\begin{con}\label{con:symmetric-universe}
  For sufficiently large $\kappa$, we define a $\kappa$-small fibration classifier $\pi \colon \dot{\UU}_\kappa \to \UU_\kappa$  by defining $\UU_\kappa \coloneq \FF^\kappa(\varpi)$ and forming the pullback
  \[
    \begin{tikzcd}
        \dot{\UU}_\kappa \arrow[d, "\pi"'] \arrow[r] \arrow[dr, phantom, "\lrcorner" very near start] & \dot{\VV}_\kappa \arrow[d, "\varpi"] \\ \UU_\kappa \arrow[r, "\psi_\varpi"'] & \VV_\kappa
    \end{tikzcd}
  \]
  where $\varpi \colon \dot{\VV}_\kappa \to \VV_\kappa$ is the Hofmann--Streicher universe classifying $\kappa$-small families in the presheaf topos $\cSet^\SSigma$.
\end{con}

By Proposition \ref{prop:has-universes}:

\begin{prop}\label{prop:symmetric-has-universes}
The interval premodel structure on cubical species has universes in the sense of Definition \ref{defn:has-universes} for the fibrations given by the classifiers $\pi \colon \dot{\UU}_\kappa \to {\UU}_\kappa$ for sufficiently large inaccessible cardinals $\kappa$.
\qed
\end{prop}

With Propositions \ref{prop:symmetric-has-universes} and \ref{prop:symmetric-frobenius}, we have satisfied the hypotheses of Proposition \ref{prop:EEP-univalence}, so from Proposition \ref{prop:symmetric-EEP} we may conclude:

\begin{prop}\label{prop:symmetric-univalent} The universes in the interval premodel structure on cubical species are univalent. \qed
\end{prop}

By Definition \ref{defn:parametrized-path-space} and Theorem \ref{thm:uniform-fibrations}, our fibrations are characterized in the way demanded by Proposition \ref{prop:FEP-generic-point}. Thus Proposition \ref{prop:fibrant-universe} applies and we may conclude:

\begin{prop}\label{prop:symmetric-fibrant-universe} The bases of the universal fibrations for the interval premodel structure on cubical species are fibrant objects. \qed
\end{prop}

By applying Lemma \ref{lem:fibU_fibext}, we see that:

\begin{prop}\label{prop:symmetric-FEP}
   The interval premodel structure satisfies the fibration extension property. \qed
\end{prop}

These results assemble into the main theorem of this section.

\begin{thm}\label{thm:species-interval-model-structure} The category of cubical species admits a Quillen model structure in which the cofibrations are the monomorphisms and the fibrations are the unbiased fibrations of \ref{defn:interval-fibrations}\eqref{itm:unbiased-fibrations}. This model is cylindrical and cartesian closed and satisfies the Frobenius condition, equivalence extension property, and fibration extension property. Moreover, it has univalent universes whose bases are fibrant objects.
\end{thm}
\begin{proof}
  The only result of the statement that we have not yet proven is the fact that the interval premodel structure is in fact a model structure, but this follows formally from Proposition \ref{prop:cylindrical-model-category}, by Proposition \ref{prop:symmetric-FEP} and the fact that all objects are cofibrant.
\end{proof}

Thus, the interval model structure on the topos of cubical species is a model of homotopy type theory.

\section{The equivariant model structure on cubical sets}\label{sec:equivariant}

Having established a model structure on the category of cubical species, we now transfer it to a model structure, and a model of homotopy type theory, on the category $\cSet$ of cartesian cubical sets. The results of \S\ref{sec:symmetric} both provide conceptual justification for the constructions in this section and also simplify many of the proofs.

In \S\ref{ssec:species-to-sets}, we introduce an adjoint triple of functors between cubical sets and cubical species and establish the basic properties of these functors. In \S\ref{ssec:cylindrical-premodel-equivariant}, we lift the cylindrical premodel structure from cubical species to cubical sets by using the constant diagram functor $\Delta \colon \cSet \to \cSet^{\SSigma}$ to create the fibrations and trivial fibrations. We give explicit characterizations of these classes that reveal that the trivial fibrations are again the trivial fibrations of \S\ref{ssec:trivial-fibrations}, while the fibrations are novel, defining a class of maps we call \emph{equivariant fibrations}.

As the cofibrations in the resulting premodel structure on cubical sets are again the monomorphisms, these are created by the functor $\Delta$ as well, but the trivial cofibrations and weak equivalences are not, so in particular it will again take work to prove that the right-lifted premodel structure in fact defines a Quillen model structure. This is achieved in \S\ref{ssec:equivariant-model}, which proves the analogue of Theorem \ref{thm:species-interval-model-structure} for cubical sets. For some of the constituent results, the proofs are formal, specializing the results of \S\ref{sec:cylindrical}; for other statements, the results of that section do not apply and we leverage the results of \S\ref{sec:symmetric} instead.

\subsection{From cubical species to equivariant cubical sets}\label{ssec:species-to-sets}

The category of cubical sets embeds faithfully into the category of cubical species via the constant diagram functor \[\Delta \colon \cSet \to \cSet^{\SSigma}\cong\prod_{k \geq 1}\cSet^{\SSigma_k},\] which is fully faithful on each factor $\cSet^{\SSigma_k}$ though only faithful on the whole. Since the groupoid $\SSigma$ is small and $\cSet$ is bicomplete, this functor admits left and right adjoints:
\[
  \begin{tikzcd}
 \cSet^\SSigma   \arrow[dd, bend right=40, "{\Leftadj}"', "\dashv"]  \arrow[dd, bend left=40, "\Gamma", "\dashv"'] \\
 \\
  \cSet \arrow[uu, "\Delta" description] \rlap{.}
  \end{tikzcd}
\]

The left adjoint $\Leftadj$ takes the colimit over the groupoid $\SSigma$, and the right adjoint $\Gamma$ takes the limit.
Explicitly, for a cubical species $\XX = (X^k)_{k \geq 1}$, we have
\begin{align*}
\Leftadj(\XX) &\coloneq \coprod_{k \geq 1}X^k_{/\Sigma_k} \\ \Gamma(\XX) &\coloneq \prod_{k \geq 1} (X^k)^{\Sigma_k}
\end{align*}
where $X^k_{/\Sigma_k}$ is the cubical set of \textbf{orbits}, the quotient of the $\Sigma_k$-cubical set $X^k$ by its action, and $(X^k)^{\Sigma_k}$ is the cubical set of $\Sigma_k$-\textbf{fixed points}.

As a category of actions by a groupoid, the topos $\cSet^{\SSigma}$ is well-known to be atomic over $\cSet$, and $\Delta \colon \cSet \to \cSet^{\SSigma}$ to be a logical functor, preserving (co)limits, the subobject classifier and the locally cartesian closed structure.  We  provide some explicit calculations of these.

\begin{ex}\label{ex:free-colimits} For $n,k \in \NN$ and $k \geq 1$, we have $\Leftadj(\FF_kI^n) \cong I^n$, reflecting the fact that left Kan extensions preserve representables. More generally, for any cubical set $X$, we have $\Leftadj(\FF_k X) \cong X$, as $\Leftadj \cdot \FF_k$ is left adjoint to the identity functor.
\end{ex}

\begin{ex} We calculate
  \[ \Leftadj(\II) \cong \coprod_{k \geq 1} I^k_{/\Sigma_k} \qquad \text{and} \qquad \Gamma(\II) \cong \prod_{k \geq 1} I \cong I^{\omega}
  \] using the fact that $(I^k)^{\Sigma_k} \cong I$ for all $k > 0$.
\end{ex}

The left adjoint $\Leftadj$ is far from being left exact, failing to preserve pullbacks (since 1-categorical quotients by a group action do not commute with pullbacks) and even finite products (since coproducts do not commute with finite products); in particular, $\Leftadj(\1)\cong \NN$.  It does, however, interact well with certain finite limits involving constant cubical species.

\begin{cor}\label{cor:constant-exponentials} The constant diagram functor $\Delta \colon \cSet \to \cSet^{\SSigma}$ preserves pushforwards and exponentials.
\end{cor}
\begin{proof}
This is an instance of Corollary \ref{groupoid-functors:pushforward}.
\end{proof}

\begin{lem}\label{lem:constant-subobjects} The constant diagram functor $\Delta \colon \cSet \to \cSet^{\SSigma}$ preserves the subobject classifier and creates monomorphisms.
\end{lem}
\begin{proof}
Preservation of the subobject classifier is an instance of Corollary \ref{groupoid-functor:subobject-classifier}.
For creation of monomorphisms, recall that monomorphisms in $\cSet^{\SSigma}$ are defined pointwise and that $\Sigma$ is inhabited.
\end{proof}

\begin{cor}\label{cor:Deltapreservesandreflects}
The constant diagram functor $\Delta \colon \cSet \to \cSet^{\SSigma}$ preserves the (relative) partial map classifiers $\eta_X \colon X\to X^+$ of Section \ref{ssec:trivial-fibrations}, and therefore also the (relative) $+$-algebras.  Since it is faithful, $\Delta$ also reflects the latter.
\qed
\end{cor}

\subsection{The cylindrical premodel structure on cubical sets}
\label{ssec:cylindrical-premodel-equivariant}

By a well-known transfer procedure, we may obtain a premodel structure on cubical sets from the premodel structure on cubical species by pulling back the right classes of the weak factorization systems along the right adjoint $\Delta \colon \cSet \to \cSet^{\SSigma}$: we say $f$ is a (trivial) fibration in $\cSet$ if $\Delta f$ is a (trivial) fibration in $\cSet^{\SSigma}$.
The transfer procedure gives us the left and right classes as well as generating categories, namely the images of the generating categories of the original weak factorization systems under the left adjoint $\Leftadj \colon \cSet^{\SSigma} \to \cSet$.
Note, however, that we do not mechanically obtain the \emph{factorizations} in $\cSet$ from those in $\cSet^{\SSigma}$; we must construct these ``by hand'', and we want in particular to do so constructively.

\begin{con}\label{con:equivariant-generating-cof}
  The trivial fibrations in $\cSet$ are generated by the category $\int\OOmega$ of Construction \ref{con:generating-cof} and the top composite functor
\begin{equation}\label{eq:initial-transfered-cofibrations}
  \begin{tikzcd}  \int\OOmega \arrow[d, "\pi"'] \arrow[r, dashed, "I"] & (\cSet^{\SSigma})^\2 \arrow[d, "\cod"] \arrow[r, "\Leftadj"] & \cSet^\2 \arrow[d, "\cod"] \\ \CCube \times \SSigma^\op \arrow[r, hook, "\yo"] & \cSet^{\SSigma} \arrow[r, "\Leftadj"] & \cSet \rlap{.}
  \end{tikzcd}
\end{equation}

Explicitly, by Example \ref{ex:free-colimits}, the composite functor $\Leftadj I \colon \int\OOmega \to \cSet^\2$  sends an element $\chi_c : \FF_kI^n \to \OOmega$ to the corresponding subobject $c \colon C \rightarrowtail I^n$ under Lemma \ref{lem:symmetric-subobjects}, while morphisms in $\int\OOmega$ as below-left are carried to pullback squares between subobjects as below-right:
\[ \begin{tikzcd} \FF_kI^m \arrow[rr, "\alpha \times \sigma"] \arrow[dr, "\chi_{d}"'] & & \FF_kI^n \arrow[dl, "\chi_c"] & \arrow[d, phantom, "\rightsquigarrow"] &  D \arrow[d, tail, "d"'] \arrow[r, "\alpha"] \arrow[dr, phantom, "\lrcorner" very near start] & C \arrow[d, tail, "c"] \\ & \OOmega & & ~ & I^m \arrow[r, "\alpha"'] & I^n \rlap{.} \end{tikzcd} \]

Note the image of the functor $\Leftadj I \colon \int\OOmega \to \cSet^\2$ on both objects and morphisms is independent of the parameter $k \in \SSigma$. The isomorphism of Lemma \ref{lem:constant-subobjects} induces an isomorphism of categories $\int\OOmega \cong (\int\Omega) \times \SSigma^\op$ and, by the observations just made,  the functor $\Leftadj I$ factors through the projection $\pi \colon \int\OOmega \to \int\Omega$. Thus, the composite rectangle of \eqref{eq:initial-transfered-cofibrations} also factors as follows:
\[
  \begin{tikzcd} \int\OOmega \arrow[d, "\pi"'] \arrow[r, "\pi"] & \int\Omega \arrow[d, "\pi"'] \arrow[r, dashed, "I"] &  \cSet^\2 \arrow[d, "\cod"] \\ \CCube \times \SSigma^\op \arrow[r, "\pi"'] & \CCube \arrow[r, hook, "\yo"'] &  \cSet \rlap{.}
  \end{tikzcd}
\]
Since the projection $\pi \colon \int\OOmega \to \int\Omega$ is an epimorphism, a generating category for the trivial fibrations on $\cSet$ can be given more simply as the category $I \colon \int\Omega \to \cSet^\2$ internally indexed by the subobject classifier $\top \colon 1 \rightarrowtail \Omega$ in $\cSet$.
\end{con}

It now follows from Remark \ref{rmk:partial-map-stable-functorial-factorization} and Proposition \ref{prop:relative-+-algebras-generation} that the cofibrations are precisely the monomorphisms and the trivial fibrations are the relative $+$-algebras, i.e.\ the algebras for the pointed polynomial endofunctors $+_X \colon \cSet_{/X}\to\cSet_{/X}$, as can also be seen from Corollary \ref{cor:Deltapreservesandreflects}.
In particular, we have a (\textbf{cofibration}, \textbf{trivial fibration}) weak factorization system with functorial factorization given by the partial map factorization of Remark \ref{rmk:partial-map-stable-functorial-factorization}.

We next transfer the (trivial cofibration, fibration) weak factorization system. This case is more delicate, however, because the left class of trivial cofibrations is not simply reflected by the constant diagram functor $\Delta \colon \cSet \to \cSet^{\SSigma}$. In order to characterize the maps in the image of the generating category, we pause to observe a result that will help us calculate orbits.

\begin{lem}\label{lem:orbits-calculation}
  Let $G$ be a group and let $S$ be a $G$-set. Consider the $G$-set $G \times S$ where $G$ acts freely on $G$ and via its specified action on $S$. Then the map
  \[
    \begin{tikzcd}[row sep=tiny]
      G \times S \arrow[r, "\tau"] & S \\ (g,s) \arrow[r, maps to] & g^{-1}\cdot s
    \end{tikzcd}
  \]
exhibits the set $S$ as the set of $G$-orbits in $G \times S$.
\end{lem}
\begin{proof}
First observe that the map in the statement defines a cone under the $G$-indexed diagram of sets defined by the $G$-set $G \times S$. For any $g,h \in G$ and $s \in S$, the action of $h$ sends the pair $(g,s)$ to $(h \cdot g, h \cdot s)$, and $\tau(h \cdot g, h \cdot s) = (h \cdot g)^{-1} \cdot (h \cdot s) = g^{-1} \cdot s = \tau (g,s)$. Given any other map $\phi \colon G \times S \to X$ to a set $X$ that is constant on $G$-orbits in the domain, we define a factorization through $\tau$ by:
  \[
    \begin{tikzcd}[row sep=tiny]
      G \times S \arrow[r, "\tau"] & S \arrow[r, "\psi"] & X\\ (g,s) \arrow[r, maps to] & g^{-1}\cdot s \\ & s \arrow[r, maps to] & \phi(e,s) \rlap{.}
    \end{tikzcd}
  \]
 Since $(g,s)$ and $(e, g^{-1}\cdot s)$ are in the same orbit, $\phi(g,s) = \phi(e, g^{-1} \cdot s) = \psi \cdot \tau(g,s)$. Uniqueness of this factorization is immediate since $\tau$ is an epimorphism.
\end{proof}

\begin{con}\label{con:equivariant-generating-tcof}
  The fibrations in $\cSet$, which we call \textbf{equivariant fibrations}, are generated by the image of the category $\int\OOmega\times \II$ under the composition:
\[
    \begin{tikzcd}  \int\OOmega\times \II  \arrow[d, "\pi"'] \arrow[r, dashed, "J"] & (\cSet^{\SSigma})^\2 \arrow[d, "\cod"] \arrow[r, "\Leftadj"] & \cSet^\2 \arrow[d, "\cod"]  \\  \CCube \times \SSigma^\op \arrow[r, hook, "\yo"'] & \cSet^{\SSigma}\arrow[r, "\Leftadj"'] & \cSet \rlap{.}
    \end{tikzcd}
  \]
  Recall, from Remark \ref{rmk:product-cat-of-elements}, that objects of $\int\OOmega \times \II$ are pairs $(c,\zeta)$ as displayed vertically below while $(\alpha,\sigma) \colon (d,\xi) \to (c,\zeta)$ defines a morphism just when the diagrams of cubical sets commute and the top square is a pullback:
   \[
     \begin{tikzcd} D \arrow[r, "\alpha"] \arrow[d, "d"', tail] \arrow[dr, phantom, "\lrcorner" very near start] & C \arrow[d, "c", tail] \\ I^m \arrow[r, "\alpha"] \arrow[d, "\xi"'] & I^n \arrow[d, "\zeta"] \\ I^k & I^k \arrow[l, "\sigma"] \rlap{.}
     \end{tikzcd}
  \]
The functor $J$ sends an element $(c,\zeta)$ to the morphism of cubical species defined by pushout of cubical species below-left, which corresponds to the pushout of $\Sigma_k$-cubical sets below-right:
  \[
    \begin{tikzcd}[column sep=small]  & \FF_k C \arrow[dl, "\FF_kc"'] \arrow[dr, "{[\zeta\cdot\FF_kc]}"]  \arrow[dd, phantom, "\rotatebox{135}{$\lrcorner$}" very near end]&& & C \times \Sigma_k  \arrow[dl,  tail, "c \times 1"'] \arrow[dr, "{[(\zeta \cdot c)^{\Sigma_k}]}"]  \arrow[dd, phantom, "\rotatebox{135}{$\lrcorner$}" very near end]\\ \FF_k I^n \arrow[dr, dotted] \arrow[ddr, "{[\zeta]}"'] && \FF_k C \times \II \arrow[dl, dotted] \arrow[ddl, "\FF_kc \times \II"] &I^n \times \Sigma_k \arrow[dr, dotted] \arrow[ddr, "{[\zeta^{\Sigma_k}]}"'] &&   C\times \Sigma_k \times I^k \arrow[dl, dotted, tail] \arrow[ddl, tail, "c \times 1"] \\ [-12pt] & \bullet \arrow[d, dashed]  && &  \bullet \arrow[d, dashed] & \\ [+2pt] & \FF_k I^n \times \II  && &I^n \times\Sigma_k \times  I^k \rlap{.}
    \end{tikzcd}
  \]
The image of the left-hand diagram under $\Leftadj$ is given by passing to orbits in the diagram of $\Sigma_k$-cubical sets above-right, and this can be calculated using Lemma \ref{lem:orbits-calculation}. This results in the pushout diagram of cubical sets
  \[
    \begin{tikzcd}[column sep=small]  &  C  \arrow[dl,  tail, "c"'] \arrow[dr, "{[\zeta \cdot c]}"]  \arrow[dd, phantom, "\rotatebox{135}{$\lrcorner$}" very near end] \\I^n \arrow[dr, dotted] \arrow[ddr, "{[\zeta]}"'] &&   C\times I^k \arrow[dl, dotted, tail] \arrow[ddl, tail, "c \times 1"] \\ [-12pt] &  I^n \cup_{C} C \times I^k \arrow[d, dashed]   \\ [+2pt] &  I^n \times I^k \rlap{.}
    \end{tikzcd}
  \]
We again refer to the subobjects in the image of the functor $J$ as \textbf{open boxes} though the nature of the gluing of the ``lid'' $I^n$ onto the ``box'' $C \times I^k$ is somewhat subtle because it involves the map $\zeta \colon I^n \to I^k$.

The functor $J$ sends morphisms \eqref{eq:generating-tcof-indexing-map} to the pullback square of cubical species below-left, which corresponds to the pullback square of $\Sigma_k$-cubical sets below-right:
\[
    \begin{tikzcd}
      \FF_k I^m \underset{\FF_kD}{\cup} \FF_kD \times \II \arrow[d, tail, "{\langle [\xi], \FF_kd \times 1\rangle}"'] \arrow[r, "\alpha \times \sigma \times 1"] \arrow[dr, phantom, "\lrcorner" very near start] & \FF_k I^n \underset{\FF_kC}{\cup} \FF_kC \times \II \arrow[d, tail, "{\langle [\zeta], \FF_kc \times 1\rangle}"] & [-2em]    I^m \times \Sigma_k \underset{D \times \Sigma_k}{\cup} D \times \Sigma_k \times I^k
     \arrow[d, tail, "{\langle [\xi^{\Sigma_k}], d \times 1\rangle}"'] \arrow[r, "\alpha \times \sigma \times 1"] \arrow[dr, phantom, "\lrcorner" very near start] & I^n \times \Sigma_k \underset{C \times \Sigma_k}{\cup} C \times \Sigma_k \times I^k \arrow[d, tail, "{\langle [\zeta^{\Sigma_k}], c \times 1\rangle}"]  \\ \FF_k I^m \times \II \arrow[r, "\alpha \times \sigma \times 1"'] & \FF_k I^n \times \II
  & I^m \times \Sigma_k \times I^k \arrow[r, "\alpha \times \sigma \times 1"'] &  I^n \times \Sigma_k \times I^k \rlap{.}
    \end{tikzcd}
  \]
  Passing to orbits using Lemma \ref{lem:orbits-calculation} this becomes
  \[
      \begin{tikzcd}
    I^m  \cup_{D } D \times I^k
     \arrow[d, tail, "{\langle [\xi], d \times 1\rangle}"'] \arrow[r, "\alpha \times \sigma^{-1}"] \arrow[dr, phantom, "\lrcorner" very near start] & I^n  \cup_{C } C\times I^k \arrow[d, tail, "{\langle [\zeta], c \times 1\rangle}"]  \\ I^m \times  I^k \arrow[r, "\alpha \times \sigma^{-1}"'] &  I^n  \times I^k \rlap{.}
    \end{tikzcd}
  \]
  Thus, an equivariant fibration is a morphism $f \colon Y \to X$ of cubical sets equipped with chosen lifts against open boxes that are uniform in pullback squares:
  \[
    \begin{tikzcd}[column sep=large, row sep=1.2cm]
    I^m  \cup_{D } D \times I^k
     \arrow[d, tail, "{\langle [\xi], d \times 1\rangle}"'] \arrow[r, "\alpha \times \sigma^{-1}"] \arrow[dr, phantom, "\lrcorner" very near start] & I^n  \cup_{C } C\times I^k \arrow[d, tail, "{\langle [\zeta], c \times 1\rangle}"] \arrow[r] & Y \arrow[d, "f"] \\ I^m \times I^k \arrow[urr, dashed] \arrow[r, "\alpha \times \sigma^{-1}"'] &  I^n \times I^k \arrow[r] \arrow[ur, dashed] & X \rlap{.}
    \end{tikzcd}
  \]
\end{con}

By Garner's algebraic small object argument, the functor $\Leftadj J \colon \int\OOmega\times \II \to \cSet^\2$ generates a (\textbf{trivial cofibration}, \textbf{equivariant fibration}) algebraic weak factorization system.
Thus we have a functorial factorization for both weakly orthogonal classes, completing the definition of a premodel structure which we call the \textbf{equivariant premodel structure}.
By construction:

\begin{lem} The adjunction
\[ \begin{tikzcd} \cSet \arrow[r, bend right, "\Delta"'] \arrow[r, phantom, "\bot"] & \cSet^{\SSigma} \arrow[l, bend right, "\Leftadj"' pos=.45]
\end{tikzcd}
\]
defines a Quillen adjunction of premodel structures between the equivariant premodel structure on $\cSet$ and the interval model structure on $\cSet^{\SSigma}$. \qed
\end{lem}

An argument similar to the proof of Lemma \ref{lem:interval-endpoints} can be used to identify explicit trivial cofibrations.

\begin{lem}\label{lem:weakly-contractible-cube-quotients} For any $k \geq 1$ and subgroup $G \subset \Sigma_k$ the inclusions $\vec{0},\vec{1} \colon 1 \to I^k_{/G}$ of the initial or final vertices into the quotient cubical set define trivial cofibrations.
\end{lem}
\begin{proof}
By Construction \ref{con:equivariant-generating-tcof}, the triangle below-left gives rise to the generating trivial cofibration below-right:
\[
  \begin{tikzcd} \emptyset  \arrow[rr, tail, "!"] & & 1 \arrow[dl, "\vec{v}"] \arrow[dr, phantom, "\rightsquigarrow"] & 1 \arrow[d, "\vec{v}"] \\ & I^k & & I^k \rlap{.}
  \end{tikzcd}
\]
When $\vec{v}$ is the point $\vec{0}$ or $\vec{1}$, then any $\sigma \in \Sigma_k$ defines a morphism of triangles, as below-left, giving rise to the morphism in the generating category of trivial cofibrations displayed below-right:
\[
  \begin{tikzcd} \emptyset \arrow[r, equals] \arrow[d, "!"', tail] \arrow[dr, phantom, "\lrcorner" very near start] & \emptyset \arrow[d, "!", tail]   \\ 1 \arrow[r, equals] \arrow[d, "\vec{v}"'] & 1 \arrow[d, "\vec{v}"] \\ I^k & I^k \arrow[l, "\sigma"]
  \end{tikzcd} \qquad\rightsquigarrow\qquad
  \begin{tikzcd} 1 \arrow[d, tail, "\vec{v}"'] \arrow[r, "\sigma^{-1}"] & 1 \arrow[d, tail, "\vec{v}"] \\ I^k \arrow[r, "\sigma^{-1}"'] & I^k \rlap{.}
  \end{tikzcd}
\]
Thus, the maps $\vec{0}, \vec{1} \colon 1 \to I^k_{/G}$ arise as colimits of diagrams valued in the subcategory of generating trivial cofibrations. Since the equivariant fibrations lift uniformly against the generating category, they lift against colimits of diagrams valued in there, proving that  the inclusions $\vec{0}, \vec{1} \colon 1 \to I^k_{/G}$ are trivial cofibrations.
\end{proof}

In particular:

\begin{cor}\label{cor:equivariant-interval-endpoints} For any $k \ge 1$, the inclusions $\vec{0},\vec{1} \colon 1 \to I^k$ of the initial or final vertices into the $k$-cube each define trivial cofibrations. \qed
\end{cor}

We now verify that the equivariant premodel structure is cartesian monoidal.

\begin{prop}\label{prop:equivariant-sm7}
  Pushout products of cofibrations are cofibrations, while the pushout product of a cofibration and a trivial cofibration is a trivial cofibration.
\end{prop}
\begin{proof}
  As the cofibrations are the monomorphisms in a topos, the first property is again immediate.
  The second statement is equivalent to the assertion that the Leibniz exponential $\widehat{\{c,f\}}$ of a uniform fibration $f \colon Y \to X$ and a monomorphism $c \colon C \rightarrowtail Z$ is a uniform fibration. But uniform fibrations and monomorphisms are created by the functor $\Delta \colon \cSet \to \cSet^{\SSigma}$ from the corresponding classes of cubical species, by definition and Lemma \ref{lem:constant-subobjects}, respectively, and in virtue of Corollary \ref{cor:constant-exponentials} the functor $\Delta$ also preserves Leibniz exponentials. So the result follows from Proposition \ref{prop:sm7}.
\end{proof}

We now observe that our premodel structure is cylindrical. Although the equivariant fibrations are not defined using a particular interval object, we will show that the naive interval object
\[ \begin{tikzcd} 1 \arrow[r, shift left, "0"] \arrow[r, shift right, "1"'] & I \arrow[r, "!"] & 1\end{tikzcd}
\]
satisfies the axioms of Definition \ref{defn:cylindrical}, using the adjunction $(-)\times I \dashv (-)^{I}$ to define our adjoint functorial cylinder.

\begin{lem}\label{lem:equivariant-cylindrical}
The equivariant premodel structure on cubical sets is cylindrical.
\end{lem}
\begin{proof}
Since the endpoints $0$ and $1$ of our interval $I$ are disjoint, the map $\partial \colon 1 + 1 \cto I$ is a monomorphism and thus a cofibration. By Corollary \ref{cor:equivariant-interval-endpoints}, the single endpoint inclusions $\partial_0, \partial_1 \colon 1 \cwto I$ are trivial cofibrations. Now the result follows from Proposition \ref{prop:equivariant-sm7}.
\end{proof}

\subsection{The equivariant cubical sets model of homotopy type theory}\label{ssec:equivariant-model}

In this section, we establish the type-theoretic properties of the cylindrical premodel structure on cubical sets needed to infer that it defines a Quillen model structure with the extra features required of a model of homotopy type theory.

The cofibrations in the equivariant premodel structure are exactly the monomorphisms, which are closed under pushout products in all slices by Remark \ref{rmk:adhesive-pushout-products}. Together with Lemma \ref{lem:equivariant-cylindrical}, this verifies the hypotheses of Theorem \ref{thm:cylindrical-EEP}, and therefore:

\begin{prop}\label{prop:equivariant-EEP}
  The equivariant premodel structure on cubical sets satisfies the equivalence extension property.\qed
\end{prop}

Unlike in the case of the interval premodel structure on cubical species, we cannot use the results of \S\ref{ssec:frobenius} to establish the Frobenius condition, as the equivariant fibrations are not the naive unbiased fibrations.
Instead, it follows for the equivariant premodel structure on cubical sets by comparison with cubical species.

\begin{prop}\label{prop:equivariant-frobenius} The equivariant fibrations satisfy the Frobenius condition.
\end{prop}
\begin{proof}
We must show that the pushforward of an equivariant fibration $g$ along an equivariant fibration $f$ defines an equivariant fibration, which is the case just when its image under the constant diagram functor is a fibration of cubical species. But since Corollary \ref{cor:constant-exponentials} tells us that this functor preserves pushforwards, this map is the pushforward of $\Delta g$ along $\Delta f$. Since the equivariant fibrations are pulled back along $\Delta$ from the fibrations, the result follows from Frobenius for the latter, Proposition \ref{prop:symmetric-frobenius}.
\end{proof}

The remaining properties require universes, which we now construct. Since the equivariant fibrations are created from the fibrations in $\cSet^{\SSigma}$ via the functor $\Delta \colon \cSet \to \cSet^{\SSigma}$, and since $\Delta$ preserves pullbacks and has a right adjoint, Example \ref{ex:local-rep-transfer} applies to tell us that that the equivariant fibrations underlie a locally representable and relatively acyclic notion of fibred structure.

\begin{lem}\label{lem:equiv-loc-rep-fib}
  There is a locally representable and relatively acyclic notion of fibred structure $\mathscr{F}$ on cubical sets whose underlying class of maps is the class of equivariant fibrations.
\end{lem}
\begin{proof}
  By Example \ref{ex:local-rep-transfer} and Lemma \ref{lem:loc-rep-fib}, there is a locally representable and relatively acyclic notion of fibred structure $\mathscr{F}$  where an $\mathscr{F}$-algebra structure on a map $f \colon Y \to X$ of cubical sets is defined to be an $\FF$-algebra structure on the map $\Delta f \colon \Delta Y \to \Delta X$ of cubical species. Then, by the proof of Lemma \ref{lem:leibniz-pullback-application-loc-rep-fibred-structure}, the map  $\psi_f \colon \mathscr{F}(f) \to X$ defined by the pullback
  \[  \begin{tikzcd} \mathscr{F}(f) \arrow[d, "\psi_f"] \arrow[r] \arrow[dr, phantom, "\lrcorner" very near start] & \Gamma \FF(\Delta f) \arrow[d, "\Gamma \psi_{\Delta f}"] \\ X \arrow[r, "\eta"'] \arrow[u, bend left, dotted] \arrow[ur, dashed] & \Gamma \Delta X
\end{tikzcd}
\] has the property that for any $g \colon Z \to X$, there is a natural bijection between equivariant fibration structures on $g^*f$ and lifts of $g$ across $\psi_f$.
\end{proof}

The same line of reasoning tells us how to construct the universal equivariant fibration. By \cite[8]{Awodey:2022hs}, the Hofmann--Streicher universe $\varpi \colon \dot{V}_\kappa \to V_\kappa$ for $\cSet$ and the Hofmann--Streicher universe $\varpi \colon \dot{\VV}_\kappa \to \VV_\kappa$ for $\cSet^\SSigma$, defined with respect to the same regular cardinal $\kappa$, are related by a canonical pullback:
\begin{equation}\label{eq:Hofmann-Streicher-comparison}
\begin{tikzcd} \Delta \dot{V}_\kappa \arrow[d, "\Delta\varpi"'] \arrow[r] \arrow[dr, phantom, "\lrcorner" very near start] & \dot{\VV}_\kappa \arrow[d, "\varpi"] \\ \Delta V_\kappa \arrow[r] & \VV_\kappa \rlap{.}
\end{tikzcd}
\end{equation}

\begin{con}\label{con:equivariant-universe}
Define $\pi \colon \dot{U}_\kappa \to U_\kappa$ to be the map of cubical sets defined by the pullbacks in the top and bottom faces of the cube, whose back face is the transpose of \eqref{eq:Hofmann-Streicher-comparison} and whose right face is the image of the pullback square of Construction \ref{con:symmetric-universe} under $\Gamma \colon \cSet^{\SSigma} \to \cSet$:
\[
\begin{tikzcd}[sep=small] & \dot{V}_\kappa \arrow[rr] \arrow[dd, "\varpi" near start]& & \Gamma \dot{\VV}_\kappa \arrow[dd, "\Gamma\varpi"] \\ \dot{U}_\kappa \arrow[dd, dashed, "\pi"']  \arrow[urrr, phantom, "\urcorner" very near start] \arrow[ur, dotted] \arrow[rr, dotted] \arrow[dr, phantom, "\lrcorner" very near start]& & \Gamma \dot{\UU}_\kappa \arrow[ur]  \arrow[dr, phantom, "\lrcorner" very near start]\\ & V_\kappa \arrow[rr] & & \Gamma \VV_\kappa \\ U_\kappa \arrow[ur, dotted] \arrow[rr, dotted, crossing over] \arrow[urrr, phantom, "\urcorner" very near start] & & \Gamma \UU_\kappa \arrow[ur] \arrow[from=uu, "\Gamma\pi"' near start, crossing over] \rlap{.}
\end{tikzcd}
\]
By pullback composition and cancelation, this makes the left face a pullback.
\end{con}

\begin{rmk}\label{rmk:equivariant-universe}
By Construction \ref{con:loc-rep-universe}, we might have instead defined $U_\kappa \to V_\kappa$  to be the map $\mathscr{F}^\kappa(\varpi) \to V_\kappa$ classifying equivariant fibration structures associated to the Hofmann--Streicher universe $\varpi \colon \dot{V}_\kappa \to V_\kappa$. However, on account of the pullback square \eqref{eq:Hofmann-Streicher-comparison} we have a pullback
\[
\begin{tikzcd} \FF^\kappa(\Delta \varpi) \arrow[d, "{\psi_{\Delta\varpi}}"'] \arrow[r] \arrow[dr, phantom, "\lrcorner" very near start] & \FF^\kappa(\varpi) \eqcolon \UU_\kappa \arrow[d, "{\psi_\varpi}"] \\ \Delta V_\kappa \arrow[r] & \VV_\kappa
\end{tikzcd}
\]
of cubical species and thus a composable pair of pullbacks of cubical sets
\[
\begin{tikzcd} U_\kappa \cong \mathscr{F}^\kappa(\varpi) \arrow[d, "\psi_\varpi"'] \arrow[r] \arrow[dr, phantom, "\lrcorner" very near start] & \Gamma \FF^\kappa(\Delta \varpi) \arrow[d, "\Gamma\psi_{\Delta\varpi}"'] \arrow[dr, phantom, "\lrcorner" very near start] \arrow[r]& \Gamma \FF^\kappa(\varpi) \eqcolon \Gamma \UU_\kappa \arrow[d, "\Gamma\psi_\varpi"] \\ V_\kappa \arrow[r, "\eta"'] & \Gamma \Delta V_\kappa \arrow[r] & \Gamma\VV_\kappa
\end{tikzcd}
\]
showing that both definitions agree (cf.\ \cite[12]{Awodey:2022hs}).
\end{rmk}

By Remark \ref{rmk:equivariant-universe} and Proposition \ref{prop:has-universes}:

\begin{prop}\label{prop:equivariant-has-universes}
The equivariant premodel structure on cubical sets has universes in the sense of Definition \ref{defn:has-universes} for the equivariant fibrations given by the classifiers $\pi \colon \dot{U}_\kappa \to {U}_\kappa$ for sufficiently large inaccessible cardinals $\kappa$.
\qed
\end{prop}

With Propositions \ref{prop:equivariant-has-universes} and \ref{prop:equivariant-frobenius}, we have satisfied the hypotheses of Proposition \ref{prop:EEP-univalence}, so from Proposition \ref{prop:equivariant-EEP} we may conclude:

\begin{prop}\label{prop:equivariant-univalent} The universes in the equivariant premodel structure on cubical sets are univalent. \qed
\end{prop}

We now leverage the results of \S\ref{ssec:Ufib} to prove that the bases of these universe are equivariantly fibrant objects. Note, however, that in contrast to the analogous result for cubical species, this is not a direct consequence of Proposition \ref{prop:fibrant-universe}.

\begin{prop}\label{prop:equivariant-fibrant-universe}
The bases of the universal fibrations for the equivariant premodel structure on cubical sets are fibrant objects.
\end{prop}
\begin{proof}
    As in the proof of Proposition \ref{prop:symmetric-fibrant-universe}, we can use Proposition \ref{prop:FEP-generic-point} to show that $U$ is fibrant, though in a more subtle way. First, we again equip $U$ with the reflexive relation defined by the object of equivalences constructed by Lemma \ref{lem:universal-equivalence}:
    \[
    \begin{tikzcd} & U \arrow[d] \arrow[dl, equals] \arrow[dr, equals] \\ U & \Eq(\dot{U}) \arrow[l, "s"] \arrow[r,"t"'] & U \rlap{.}
    \end{tikzcd}
    \]
    The map $(s,t) \colon \Eq(\dot{U}) \to U \times U$ is again a fibration by its construction. By univalence, Proposition \ref{prop:equivariant-univalent}, the map $t \colon \Eq(\dot{U}) \to U$ is a trivial fibration and in particular a fibration.

    Now the equivariant premodel structure lacks an interval $I$ as required by Proposition \ref{prop:FEP-generic-point}, but by the definition of the equivariant fibrations, the images of the maps $(s,t) \colon \Eq(\dot{U}) \to U \times U$ and $t \colon \Eq(\dot{U}) \to U$ under $\Delta$ are uniform fibrations in $\cSet^\SSigma$, and we are trying to show that $\Delta U$ is uniformly fibrant.  Since the interval (pre)model structure on cubical species does have such an interval, and the remaining hypotheses of Proposition \ref{prop:FEP-generic-point} are also satisfied for the reflexive relation $\Delta \Eq(\dot{U}) \rightrightarrows \Delta U$, we can conclude that $\Delta U$ is indeed uniformly fibrant. Thus, $U$ is equivariantly fibrant.
    \end{proof}

By applying Lemma \ref{lem:fibU_fibext}, we see that:

\begin{prop}\label{prop:equivariant-FEP}
   The equivariant premodel structure satisfies the fibration extension property. \qed
\end{prop}

These results assemble into the main theorem of this section.

\begin{thm}\label{thm:equivariant-model-structure} The category of cubical sets admits a Quillen model structure in which the cofibrations are the monomorphisms and the fibrations are the equivariant fibrations. This model is cylindrical and cartesian closed and satisfies the Frobenius condition, equivalence extension property, and fibration extension property. Moreover, it has univalent universes whose bases are fibrant objects.
\end{thm}
\begin{proof}
 Once more, the only result of the statement that we have not yet proven is the fact that the interval premodel structure is in fact a model structure, but this follows formally from Proposition \ref{prop:cylindrical-model-category}, by Proposition \ref{prop:equivariant-FEP} and the fact that all objects are cofibrant.
\end{proof}

Thus, the equivariant model structure on the topos of cubical sets is a model of homotopy type theory. In contrast to the model of Theorem \ref{thm:species-interval-model-structure}, the equivariant model structure presents classical homotopy theory, as we demonstrate in the next section.

\section{The equivalence with classical homotopy theory}\label{sec:classical}

In this section we prove our final main result, that the equivariant cubical model category of Theorem \ref{thm:equivariant-model-structure} is equivalent to classical homotopy theory. More specifically, we demonstrate that the triangulation functor $T \colon \cSet \to \sSet$ defines a left Quillen equivalence from the equivariant model structure, whose fibrations are the equivariant fibrations, to Quillen's model structure on simplicial sets \cite{Quillen:1967ha}, whose fibrations are the Kan fibrations. This argument makes use of classical reasoning; see \S\ref{ssec:constructivity} above.

Our proof makes central use of the fact that the indexing categories  $\CCube$ and $\DDelta$ are \emph{Eilenberg--Zilber categories}, a special class of (generalized) Reedy categories introduced by Berger and Moerdijk \cite{BergerMoerdijk:2008rc}. We develop some general theory of Eilenberg--Zilber categories in \S\ref{ssec:EZ} for that purpose. In particular, we prove in Corollary \ref{cor:EZ-magic} that to check that a natural transformation between left Quillen functors with either $\cSet$ or $\sSet$ as domain is a natural weak equivalence, it will suffice to check this on those components indexed by quotients of representables by subgroups of their automorphism groups. And in fact, by the two-of-three property, this will follow automatically for terminal object preserving functors, provided these objects are weakly contractible---as is the case in both the  equivariant model structure on $\cSet$ and the classical model structure on $\sSet$.

These results make it easy to prove that an opposing pair of left Quillen functors between $\sSet$ and $\cSet$ define a derived equivalence, and thus we seek a left Quillen functor from $\sSet$ to $\cSet$ to define a candidate inverse to triangulation. Our original proof proceeded along the following lines. In \cite{Sattler:2019ic}, Sattler observes that the idempotent completion of the category of \emph{Dedekind} cubes---the full subcategory of $\Cat$ on the posets $\{0 < 1\}^n$ for $n \geq 0$, which adds connections to the cartesian cubes $\CCube$---is the category $\LLat$ whose objects are the finite bounded lattices and whose morphisms are the monotone maps between them.  Thus the category $\lSet$ of presheaves on $\LLat$ is equivalent to the category of presheaves on the Dedekind cubes, which we can equip with the model structure defined in \cite{Sattler:2017ee}, following \cite{CCHM:2018ctt}. The utility of this result is that the finite ordinals $[n] = \{ 0 < 1 < \cdots < n\}$ are finite complete lattices; indeed, we have a fully faithful embedding $j \colon \DDelta \hookrightarrow \LLat$, in addition to the evident (non-full) inclusion $k \colon \CCube \to \LLat$ of the cartesian cube category. These functors induce adjoint triples of functors
\[
\begin{tikzcd}
\CCube \arrow[r, "k"] & \LLat & \DDelta \arrow[l, hook', "j"'] & & \cSet \arrow[r, bend left=45, "k_!", "\bot"'] \arrow[r, bend right=45, "k_*"', "\bot"] & \lSet \arrow[r, "j^*" description] \arrow[l, "k^*" description] & \sSet \arrow[l, bend left=45, "j_*", "\bot"'] \arrow[l, bend right=45, "j_!"', "\bot"] \rlap{.}
\end{tikzcd}
\]
with the left and right adjoints defined by left and right Kan extension. The composite $j^*k_! \colon \cSet \to \sSet$ is the triangulation functor and one can verify that $k^*j_! \colon \sSet \to \cSet$ is a left Quillen homotopy inverse.

While this article was in preparation, Reid Barton observed that the triangulation functor in fact arises by restriction along a single functor $i \colon \DDelta \to \CCube$, and in particular has a left adjoint, which is also left Quillen \cite{Barton:2024t}. These results are verified in \S\ref{ssec:triangulation}. In \S\ref{ssec:EZ}, we then apply the theory of Eilenberg--Zilber categories sketched above to conclude that all three functors in the adjoint triple
\[
\begin{tikzcd}
\cSet  \arrow[r, "i^*" description] & \sSet \arrow[l, bend left=45, "i_*", "\bot"'] \arrow[l, bend right=45, "i_!"', "\bot"]
\end{tikzcd}
\]
are Quillen equivalences.  Finally, in \S\ref{ssec:test}, we compare the equivariant model structure on cubical sets to the test model structure of Cisinski after Grothendieck and prove that they coincide.

\subsection{Triangulation}\label{ssec:triangulation}

As Barton observed, implicit in Joyal's proof that $\sSet$ is the classifying topos for a strict interval is the definition of a faithful dimension-preserving functor \[ \begin{tikzcd}[row sep=tiny] \DDelta \arrow[r, "i"] & \CCube \\ \{0 < \cdots < n \} \arrow[r, maps to] & \{\bot,1, \ldots, n, \top\} \end{tikzcd}\] from the simplex category to the cartesian cube category. This functor may be defined using Joyal's ``interval representation'' \cite{Joyal:1997dd}, a contravariant isomorphism between $\DDelta$ and the opposite of the category of \emph{strict intervals}, linearly ordered sets $\{ \top > 1 >\cdots > n > \top \}$ for $n \geq 0$ with $\bot \neq \top$, and endpoint-preserving ordered maps.\footnote{Our atypical choice of ordering on the interval coordinates is chosen to match the conventions used in \cite{RiehlShulman17}, which uses the functor $i \colon \DDelta \to \CCube$ to give a syntactic encoding of the simplices as ``shapes'' embedded in cubes.} The category of strict intervals is evidently a subcategory of finite bipointed sets $\Fin_{\bot\neq\top} \cong \CCube^\op$, thus defining $i \colon \DDelta \to \CCube$.

The functor $i$ sends sends outer face maps $\delta^0,\delta^n \colon [n-1] \to [n]$ to the face maps $I^{n-1} \to I^n$ that respectively fix the first cube coordinate to be $\top$ and the last cube coordinate to be $\bot$. The inner face maps $\delta^i \colon [n-1] \to [n]$ are sent to the diagonal maps $I^{n-1} \to I^n$ that identify the $i$th and $(i+1)$th coordinates. The degeneracy maps $\sigma^i \colon [n+1] \to [n]$ are sent to the projections $I^{n+1}\to I^n$ away from the $(i+1)$th coordinate.

Barton then observed:

\begin{lem}[Barton]\label{lem:restriction-is-triangulation} Restriction along $i$ defines the triangulation functor $i^* \colon \cSet \to \sSet$.
\end{lem}
\begin{proof}
    The triangulation functor is the unique cocontinuous functor extending the product-preserving functor $\CCube \to \sSet$ that carries the interval in $\CCube$ to the interval in $\sSet$:
    \[ \begin{tikzcd}[row sep=tiny] \CCube \arrow[r, "\yo", hook] & \cSet \arrow[r, "T"] & \sSet  \\  {\{\bot,\top\}} \arrow[dd, shift left] \arrow[dd, shift right] \arrow[r, maps to] & I^0 \arrow[dd, shift left] \arrow[dd, shift right] \arrow[r, mapsto] &  \Delta^0 \arrow[dd, shift left] \arrow[dd, shift right] \\ ~\arrow[r, phantom, "\mapsto"] & ~\arrow[r, phantom, "\mapsto"] & ~ \\{\{\bot,1, \top\}} \arrow[uu] \arrow[r, maps to] & I^1 \arrow[uu] \arrow[r, maps to] & \Delta^1 \arrow[uu] \rlap{.} \end{tikzcd} \]
     The restriction functor $i^* \colon \cSet \to \sSet$ is cocontinuous and product-preserving, as is the Yoneda embedding $\yo \colon \CCube \hookrightarrow\cSet$, so it suffices to show that $i^*(I^1) = \Delta^1$ and similarly for the interval maps. Since $i[1] \coloneq \{\bot,1,\top\}$,  $i^*(I^1)$ is the functor $\CCube(i[-],i[1]) \colon \DDelta^\op \to \Set$. Now the claim follows because the inclusion $i$ is fully faithful on maps with codomain $[1]$, as in $\Fin_{\bot\neq\top} \cong \CCube^\op$ any map of bipointed sets with domain $\{\bot,1,\top\}$ is order-preserving.

    As a right adjoint, $i^*(I^0) = \Delta^0$ and by inspection, $i^*$ carries the maps $0,1 \colon I^0 \to I^1$ and $! \colon I^1 \to I^0$ in $\CCube$ to the corresponding maps involving $\Delta^1$. Thus $i^*$ coincides with the triangulation functor, as claimed.
\end{proof}

We now verify that both left adjoints in the adjoint triple
\[
\begin{tikzcd}
\cSet  \arrow[r, "i^*" description] & \sSet \arrow[l, bend left=45, "i_*", "\bot"'] \arrow[l, bend right=45, "i_!"', "\bot"]
\end{tikzcd}
\]
are left Quillen. To analyze the left Kan extension $i_!$, it will be useful to establish the relationship between $i$ and its augmented analogue. Let $\DDelta_+$ and $\CCube_+$ denote the augmented simplex and augmented cube categories, obtained by freely adjoining initial objects, and write $i_+ \colon \DDelta_+ \to \CCube_+$ for the functor induced by $i$ that preserves them. Write $\sSet_+ \coloneqq \Set^{\DDelta_+^\op}$ and $\cSet \coloneqq \Set^{\CCube_+^\op}$.

\begin{lem}\label{lem:exact-interval-reps} The commutative square below-left is exact, defining a canonical natural isomorphism in the square of functors below-right:
  \[ \begin{tikzcd} \DDelta \arrow[r, "i"] \arrow[d, hook, "\iota"'] \arrow[dr, phantom, "\Rightarrow" rotate=-135] & \CCube \arrow[d, hook, "\iota"] & \sSet \arrow[r, "i_!"] \arrow[dr, phantom, "{\rotatebox[origin=c]{-35}{$\Rightarrow$}\cong}"] & \cSet \\
    \DDelta_+ \arrow[r, "i_+"'] & \CCube_+ & \sSet_+ \arrow[u, "\iota^*"] \arrow[r, "(i_+)_!"'] & \cSet_+ \arrow[u, "\iota^*"'] \rlap{.} \end{tikzcd}\]
\end{lem}
\begin{proof}
  Here the isomorphism in the square above-right is the Beck--Chevalley transformation associated to the identity natural transformation in the square above-left, and thus is invertible when the square is exact \cite{Guitart:1980ce}. Exactness of this square follows from the general observation that for any functor $k \colon \cC \to \cD$, any commutative square of the form below is exact:
  \[ \begin{tikzcd} \cC \arrow[r, "k"] \arrow[d, hook, "\iota"'] \arrow[dr, phantom, "\Rightarrow" rotate=-135] & \cD \arrow[d, hook, "\iota"] \\ \1 \ast \cC \arrow[r, "\1 \ast k"'] & \1 \ast \cD \rlap{.} \end{tikzcd} \]
This in turn can be detected by pasting with exact squares into $\iota\colon \cC \hookrightarrow \1 \ast \cC$ over any family of jointly surjective functors into $\1\ast\cC$ \cite[2.8 with $\mathcal{W} = \mathcal{W}_0$]{Maltsiniotis:2012}, such as the pair formed by the left and right inclusions $\iota \colon \1 \hookrightarrow\1\ast \cC$ and $\iota \colon \cC \hookrightarrow\1\ast\cC$. To that end we observe that
\[ \begin{tikzcd} \emptyset \arrow[r] \arrow[d] \arrow[dr, phantom, "\Rightarrow" rotate=-135] & \cC \arrow[r, "k"] \arrow[d, hook, "\iota"'] \arrow[dr, phantom, "\Rightarrow" rotate=-135] & \cD \arrow[d, hook, "\iota"] \arrow[dr, phantom, "="] & \emptyset  \arrow[r] \arrow[d] \arrow[dr, phantom, "\Rightarrow" rotate=-135] & \cD \arrow[d, hook, "\iota"] \\ \1 \arrow[r, hook, "\iota"'] & \1 \ast \cC \arrow[r, "\1 \ast k"'] & \1 \ast \cD & \1 \arrow[r, hook, "\iota"'] & \1 \ast \cD \end{tikzcd} \]
where both the left-hand square and the composite rectangle are comma squares, and thus exact.
Similarly, the left-hand and right-hand squares in the pasting equation below are exact since the functors $\iota$ are fully-faithful,
\[ \begin{tikzcd} \cC \arrow[r, equals] \arrow[d, equals] \arrow[dr, phantom, "\Rightarrow" rotate=-135] & \cC \arrow[r, "k"] \arrow[d, hook, "\iota"'] \arrow[dr, phantom, "\Rightarrow" rotate=-135] & \cD \arrow[d, hook, "\iota"] \arrow[dr, phantom, "="] & \cC  \arrow[r, "k"] \arrow[d, equals] \arrow[dr, phantom, "\Rightarrow" rotate=-135] & \cD \arrow[r, equals] \arrow[dr, phantom, "\Rightarrow" rotate=-135] \arrow[d, equals] & \cD \arrow[d, hook, "\iota"] \\ \cC \arrow[r, hook, "\iota"'] & \1 \ast \cC \arrow[r, "\1 \ast k"'] & \1 \ast \cD & \cC \arrow[r, "k"'] & \cD \arrow[r, hook, "\iota"'] & \1 \ast \cD \rlap{,} \end{tikzcd} \] while the trivial square is trivially exact.
\end{proof}

Using this, we now demonstrate:

\begin{lem} The functors
  \[ \begin{tikzcd}
  \cSet  \arrow[r, "i^*"', bend right] \arrow[r, phantom, "\bot"] & \sSet  \arrow[l, bend right, "i_!"']
  \end{tikzcd} \] preserve monomorphisms.
\end{lem}
\begin{proof}
This is immediate for the right adjoint $i^*$. For the left adjoint $i_!$, we observe
\[ i_! \cong i_! \iota^* \iota_* \cong \iota^* (i_+)_! \iota_*,\] by Lemma \ref{lem:exact-interval-reps} and fully faithfulness of $\iota \colon \DDelta \hookrightarrow \DDelta_+$. Thus, to prove that $i_!$ preserves monomorphisms it suffices to prove that $(i_+)_!$ does.

Monomorphisms in $\sSet_+$ decompose canonically as sequential colimits of pushouts of coproducts of maps of the form $\partial\Delta^n \hookrightarrow \Delta^n$. As a left adjoint, $(i_+)_!$ preserves cell complexes, so it suffices to show that this functor carries these generating maps to monomorphisms. Each boundary inclusion is the joint image of the family of monomorphisms $\delta \colon \Delta^m \rightarrowtail \Delta^n$ indexed by monomorphisms $\delta \colon [m] \rightarrowtail [n]$ in $\DDelta_+$ with codomain $[n]$. Thus, it suffices to prove that $(i_+)_!$ preserves joint images of monomorphisms between representables. In a Grothendieck topos, the joint image of monomorphisms $(m_i \colon A_i \rightarrowtail B)_{i \in I}$ is given by the coequalizer of the following parallel pair of maps in the slice over $B$
\[ \begin{tikzcd} \coprod_{i,j \in I} A_i \times_B A_j \arrow[r, shift right] \arrow[r, shift left] & \coprod_{k \in I}A_k \end{tikzcd} \] and thus a cocontinuous functor between Grothendieck toposes will preserve the joint image of a family of monomorphisms provided it preserves the pullbacks of cospans in the family. In the case of the functor $(i_+)_!$ and the family of monomorphisms $(\delta_i \colon \Delta^{m_i} \rightarrowtail \Delta^n)_i$, we'll demonstrate this by showing that $\DDelta_+$ has pullbacks of face maps and $i_+ \colon \DDelta_+ \to \CCube_+$ preserves them.\footnote{This is the advantage of working with $i_+$ rather than $i$; $\DDelta$ does not have pullbacks of all face maps.}

The functor $i_+ \colon \DDelta_+ \to \CCube_+$ is the opposite of the functor $i_+ \colon \cat{FinInt} \to \cat{Fin}_{\bot,\top}$ from the category of finite intervals $\{ \bot > 1 > \cdots > n > \top\}$, now possibly with $\bot=\top$, to the category of finite bipointed sets, now dropping the requirement that the basepoints are distinct. We must show that $\cat{FinInt}$ has and $i_+ \colon \cat{FinInt} \to \cat{Fin}_{\bot,\top}$ preserves pushouts of epimorphisms, or equivalently for any finite interval $A$ that the comma category ${A}\downarrow\cat{FinInt}$ has and the forgetful functor $i_+ \colon {A}\downarrow\cat{FinInt} \to {i_+A}\downarrow\cat{Fin}_{\bot,\top}$ preserves binary coproducts of epimorphisms. On account of the epimorphism--monomorphism orthogonal factorization systems, it suffices to restrict to the subcategories of epimorphisms $\cat{FinInt}^{\text{epi}}$ and $\cat{Fin}_{\bot,\top}^{\text{epi}}$ and show that binary coproducts exist in ${A}\downarrow\cat{FinInt}^{\text{epi}}$ and are preserved by the forgetful functor between comma categories $i_+ \colon {A}\downarrow\cat{FinInt}^{\text{epi}} \to {i_+A}\downarrow\cat{Fin}_{\bot,\top}^{\text{epi}}$.

For a finite interval $A$, the category ${A}\downarrow\cat{FinInt}^{\text{epi}}$ is the poset whose objects are equivalence relations on the underlying set of $A$ whose equivalence classes are subintervals of $A$ (where the inclusion of a subinterval need not preserve endpoints). The category ${i_+A}\downarrow\cat{Fin}_{\bot,\top}^{\text{epi}}$ is the poset whose objects are equivalence relations on the underlying set of $A$. Using these descriptions, we see that the functor $i_+ \colon {A}\downarrow\cat{FinInt}^{\text{epi}} \to {i_+A}\downarrow\cat{Fin}_{\bot,\top}^{\text{epi}}$ is a coreflective embedding, whose right adjoint sends an equivalence relation on the underlying set of $A$ to the equivalence relation that relates elements $x$ and $y$ of $A$ if only if the closed subinterval spanned by these elements belongs to a single equivalence class. In particular, this forgetful functor creates the coproducts that exist in ${i_+A}\downarrow\cat{Fin}_{\bot,\top}^{\text{epi}}$, which demonstrates what we needed to show.
\end{proof}

\begin{rmk} The closely-related criterion of \cite[3.5]{Sattler:2019ic} is not strong enough to demonstrate that $i_!$ or $(i_+)_!$ preserve monomorphisms since the pullback in $\DDelta$
  \[ \begin{tikzcd} {[1]} \arrow[d, equals] \arrow[r] \arrow[dr, phantom, "\lrcorner" very near start] & {[3]} \arrow[d, two heads] \\ {[1]} \arrow[r, tail] & {[2]} \end{tikzcd}\] of the maps specified by preserving initial and terminal elements is not preserved by the inclusion into the cartesian cube category. Note however that only one of the maps in the original cospan is a monomorphism. The proof just given demonstrates that pullbacks of pairs of monomorphisms in $\DDelta_+$ exist and are preserved by $i_+$.
\end{rmk}

\begin{lem}\label{lem:i!-left-quillen} The functor $i_! \colon \sSet \to \cSet$ defines a left Quillen functor from the classical model structure to the equivariant model structure.
\end{lem}
\begin{proof}
    As in \cite[3.6]{Sattler:2019ic}, it suffices to show that $i_!$ carries generalized horn inclusions---inclusions of the union of a proper subset of codimension-one faces into a simplex---to trivial cofibrations. Such generalized horn inclusions either have the form of a face map $\delta \colon \Delta^{n-1} \to \Delta^n$ or are pushouts of a generalized horn inclusion with one less face and in one smaller dimension. Thus, by the 2-of-3 property and induction over the dimension and the number of faces in the generalized horn inclusion, it suffices to show that $i_!\Delta^n \cong I^n$ is weakly contractible for each $n$, which holds because the cubes are weakly contractible in the equivariant model structure by Corollary \ref{cor:equivariant-interval-endpoints}.
\end{proof}

We prove that the other left adjoint $i^* \colon \sSet \to \cSet$ is left Quillen by first demonstrating a result of independent interest: that Kan fibrations of simplicial sets are also equivariant fibrations, which we define as follows.

\begin{defn} Let $\cE$ be a locally cartesian closed category equipped with a product-preserving functor $\CCube \to \cE$ from the cartesian cube category, which restricts along the inclusion $\SSigma\subset \CCube$ to define a symmetric sequence $\II \colon \SSigma \to \cE$, specifying $k$-cubes $I^k$ in $\cE$ for all $k \geq 1$ together with automorphisms for each $\sigma \in \Sigma_k$. Then an \textbf{equivariant fibration} is a map $f \colon Y \to X$ whose image under the constant diagram functor $\Delta \colon \cE \to \cE^{\SSigma}$ is an unbiased uniform fibration, i.e., a map which enjoys the uniform lifting property as below-left defined relative to the diagram in $\cE$ below-right:
\[
    \begin{tikzcd}[row sep=large, column sep=5em]
  B  \cup_{D } D \times I^k
   \arrow[d, tail, "{\langle [\xi], d \times 1\rangle}"'] \arrow[r, "\alpha \times \sigma^{-1}"] \arrow[dr, phantom, "\lrcorner" very near start] & A  \cup_{C } C\times I^k \arrow[d, tail, "{\langle [\zeta], c \times 1\rangle}"' pos=.25] \arrow[r, "{\langle y,z \rangle}"] & Y \arrow[d, "f"] \\ B \times  I^k \arrow[r, "\alpha \times \sigma^{-1}"'] \arrow[urr, dotted, "{j_{d,\zeta}(y\alpha,z (\alpha \times \sigma^{-1}),x (\alpha \times \sigma^{-1}))}" description, pos=.25] &  A  \times I^k \arrow[r, "x"'] \arrow[ur, dotted, "{j_{c,\zeta}(y,z,x)}" description] & X
  \end{tikzcd}
\qquad \begin{tikzcd} D \arrow[d, tail, "d"'] \arrow[r, "\alpha"] & C \arrow[d, tail, "c"] \\ B \arrow[r, "\alpha"] \arrow[d, "\xi"'] & A \arrow[d, "\zeta"] \\ I^k & I^k \arrow[l, "\sigma"] \rlap{.} \end{tikzcd}\]
\end{defn}

When $\cE$ is a presheaf category, it suffices to consider uniform lifting against monomorphisms with representable codomain.

\begin{prop}\label{prop:kan-is-equivariant} Kan fibrations of simplicial sets are equivariant fibrations.
\end{prop}

\begin{proof}
 Since the classical model structure on simplicial sets is cartesian closed,
   any Kan fibration $f \colon Y \fto X$ admits the structure of a biased uniform fibration, as in Definition \ref{defn:interval-fibrations}\ref{itm:biased-fibrations} with respect to the interval $\Delta^1$; see \cite[\S 9]{GambinoSattler:2017fc}. In fact, when $f \colon Y \fto X$ is a Kan fibration, it also admits the structure of an unbiased uniform fibration by \cite[4.22--23]{CavalloSattler:2022re}.

   Unpacking, this means that a Kan fibration
  $f \colon Y \fto X$ can be equipped with a uniform lifting function $i_{c, \zeta}$ as below:
\[
\begin{tikzcd}[row sep=4em, column sep=5em] B \underset{D}{\cup} D \times I  \arrow[r, "{\alpha \cup_\alpha \alpha \times I}"] \arrow[d, tail, "{\langle [\zeta\alpha], d \times 1 \rangle}"'] \arrow[dr, phantom, "\lrcorner" pos=.01] &A   \underset{C }{\cup}C \times I \arrow[d, tail, "{\langle [\zeta], c \times 1 \rangle}"' pos=.25] \arrow[r, "{\langle y,z \rangle}"] & Y \arrow[d, two heads, "f"] \\ B \times I \arrow[r, "\alpha \times I"'] \arrow[urr, dashed, "{i_{d,\zeta\alpha }(y\alpha,z (\alpha\times I),x(\alpha\times I))}" description, pos=.25] & A \times I \arrow[r, "x"'] \arrow[ur, dashed, "{i_{c,\zeta}(y,z,x)}" description] & X \rlap{.}
\end{tikzcd}
\]

Our task is to equip a uniform fibration $(f \colon Y \twoheadrightarrow X, i_{c,e})$ with the structure of an equivariant fibration.  To do so, we make use of a map
\[ \gamma_\wedge \colon I^k \times I \to I^k \qquad \gamma_\wedge(x_1,\ldots, x_k,e) \coloneq (x_1 \wedge e, \ldots, x_k \wedge e),\] that restricts along $\{0\} \rightarrowtail I$ to the constant map at $\vec{0} \in I^k$ and restricts along $\{1 \} \rightarrowtail I$ to the identity. This ``min connection'' exists because we are working with triangulated cubes in the category of simplicial sets, rather than with cartesian cubes.\footnote{We could equally use the ``max connection'' to obtain a map that restricts along $\{0\} \rightarrowtail I$ to the identity  and restricts along $\{1 \} \rightarrowtail I$ to the constant map at $\vec{1} \in I^k$.}  For any $\zeta \colon A \to I^k$, the composite
\[\begin{tikzcd} \gamma_\wedge\zeta \coloneq A \times I \arrow[r, "\zeta \times I"] & I^k \times I \arrow[r, "\gamma_\wedge"] & I^k
\end{tikzcd}
\]
defines a homotopy from the constant map $\vec{0} \colon A \to I^k$ to $\zeta$. We frequently pair this contracting homotopy with the map that records the coordinates from $A$, which we abbreviate as:
\[\begin{tikzcd}[sep=huge] \vec{\gamma_\wedge}\zeta \coloneq A \times I \arrow[r, "{(\pi, \gamma_\wedge\zeta)}"] & A \times I^k \rlap{.} \end{tikzcd}
\]

The uniform fibration structure of $f$ provides a solution to the lifting problem
\[
\begin{tikzcd}[sep=large]
 { A \times \{1\}\underset{C \times \{1\}}{\cup} C \times I} \arrow[r, "{A \cup \vec{\gamma_\wedge}\zeta c}"] \arrow[d, tail, "{c \hat{\times}\partial_1}"'] &  A\underset{C}{\cup} C \times I^k  \arrow[d, tail, "{\langle [\zeta], c \times I^k\rangle}" pos=.6] \arrow[r, "{\langle y,z \rangle}"] & Y \arrow[d, "f", two heads] \\ A \times I \arrow[urr, dashed, "{i_{c,1}(z\vec{\gamma_\wedge}\zeta c, y, x \vec{\gamma_\wedge}\zeta)}" description, pos=.3]    \arrow[r, "{\vec{\gamma_\wedge}\zeta}"'] &  A \times I^k \arrow[r, "x"']& X \rlap{.}
\end{tikzcd}
\]
This gives rise to a new lifting problem
\[
\begin{tikzcd}[row sep=huge, column sep=2em]
A \underset{C}{\cup} C \times I^k \arrow[d, tail, "{\langle [\zeta], c \times I^k\rangle}"'] \arrow[r] \arrow[dr, phantom, "\lrcorner" very near start] & [-1.75em] \left( C \times I^k \times I\underset{C \times I}{\cup} A \times I  \right)
 \bigcup\limits_{{C \times I^k  \times \{0\}\underset{C \times \{0\}}{\cup}  A \times \{0\}}}  A \times I^k \times \{0\}  \arrow[r, "A \times ! \times I"] \arrow[d, hook, "{\langle [\zeta], c \times I^k \rangle \hat{\times} \partial_0}"']  & A \times I \arrow[r, "{i_{c,1}(\cdots)}"] \arrow[d, "{\vec{\gamma_\wedge}\zeta}" pos=.7] & Y \arrow[d, "f", two heads] \\ A \times I^k \times \{1\} \arrow[r, "{A \times I^k \times \partial_1}"'] & A \times I^k \times I \arrow[r, "{A \times \gamma_\wedge}"']  \arrow[urr, dashed, "{i_{{\langle c \times I^k, [\zeta] \rangle},0}(i_{c,1}(\cdots) !, x \gamma_\wedge)}" description, pos=.3] & A \times I^k \arrow[r, "x"'] & X \rlap{,}
\end{tikzcd}
\]
which restricts to the original lifting problem. Thus, we define $j_{c,\zeta}(y,z,x)$ to be the composite
\[ j_{c,\zeta}(y,z,x) \coloneq i_{{\langle c \times I^k, [\zeta] \rangle},0}(i_{c,1}(z\vec{\gamma_\wedge}\zeta c, y, x \vec{\gamma_\wedge}\zeta) !, x \gamma_\wedge)\cdot (A \times I^k \times \partial_1).\]

It remains to verify that
\[ j_{c,\zeta}(y,z,x) \cdot (\alpha \times\sigma^{-1}) = j_{d,\sigma\zeta\alpha}( y\alpha,  z(\alpha \times \sigma^{-1}), x(\alpha \times \sigma^{-1})).\]
On account of the commutative diagrams
\[
\begin{tikzcd} B \times I \arrow[r, "\sigma\zeta\alpha \times I"] \arrow[d, "\alpha \times I"'] & I^k \times I \arrow[r, "\gamma_\wedge"] \arrow[d, "\sigma^{-1} \times I" description] & I^k \arrow[d, "\sigma^{-1}"] & B \times I \arrow[d, "\alpha \times I"'] \arrow[r, "\vec{\gamma_\wedge}\sigma\zeta\alpha"] & B \times I^k \arrow[d, "\alpha \times \sigma^{-1}"] \\ A \times I \arrow[r, "\zeta \times I"'] & I^k \times I \arrow[r, "\gamma_\wedge"'] & I^k & A \times I \arrow[r, "\vec{\gamma_\wedge}\zeta"'] & A \times I^k \rlap{,}
\end{tikzcd}
\]
we see that the outer rectangles in the following lifting problems coincide:
\[
\begin{tikzcd}[row sep=huge, column sep=5.25em]
 {B \times \{1\} \!\!\! \underset{D \times \{1\}}{\cup}\!\!\!D \times I } \arrow[r, "{B \cup \vec{\gamma_\wedge}\sigma\zeta \alpha d }"] \arrow[d, tail, "{d \hat{\times}\partial_1}"'] & B \underset{D}{\cup} D \times I^k \arrow[r, " \alpha \cup_\alpha \alpha \times \sigma^{-1} "] \arrow[d, tail, "{\langle [\sigma\zeta\alpha], d \times I^k\rangle}"' pos=.2] & A\underset{C}{\cup} C \times I^k  \arrow[d, tail, "{\langle [\zeta], c \times I^k\rangle}"] \arrow[r, "{\langle y, z \rangle}"] & Y \arrow[d, "f", two heads] \\ B \times I \arrow[urrr, dashed, "\qquad\quad\qquad{i_{d,1}(y\alpha, z(\alpha \times \sigma^{-1})\vec{\gamma_\wedge}\sigma\zeta \alpha d , x(\alpha \times \sigma^{-1}) \vec{\gamma_\wedge}\sigma\zeta\alpha)}" description, pos=.4]    \arrow[r, "{\vec{\gamma_\wedge}\sigma\zeta\alpha}"'] & B \times I^k \arrow[r, "\alpha \times \sigma^{-1}"']  & A \times I^k \arrow[r, "x"']  & X
\end{tikzcd}
\]
\[
\begin{tikzcd}[row sep=huge, column sep=large]
{B \times \{1\}\underset{D \times \{1\}}{\cup} D \times I } \arrow[dr, phantom, "\lrcorner" pos=.01] \arrow[d, tail, "{d \hat{\times}\partial_1}"'] \arrow[r, "\alpha \cup_a (\alpha \times I)"]  &  {A \times \{1\} \underset{C \times \{1\}}{\cup} C \times I} \arrow[r, "{A\cup\vec{\gamma_\wedge}\zeta c }"] \arrow[d, tail, "{c \hat{\times}\partial_1}"'] &  A \underset{C}{\cup} C \times I^k \arrow[d, tail, "{\langle [\zeta], c \times I^k\rangle}" pos=.6] \arrow[r, "{\langle y, z \rangle}"] & Y \arrow[d, "f", two heads] \\ B \times I \arrow[r, "\alpha \times I"'] &  A \times I \arrow[urr, dashed, "{i_{c,1}(y,z\vec{\gamma_\wedge}\zeta c,  x \vec{\gamma_\wedge}\zeta)}" description, pos=.3]    \arrow[r, "{\vec{\gamma_\wedge}\zeta}"'] &  A \times I^k \arrow[r, "x"']& X \rlap{.}
\end{tikzcd}
\]
By uniformity of $(f,i)$ in the left-hand pullback square of the second of these diagrams,
\begin{align*}
i_{c,1}(y, z\vec{\gamma_\wedge}\zeta c,  x \vec{\gamma_\wedge}\zeta) \cdot (\alpha \times I) &= i_{d,1}(y\alpha, z\vec{\gamma_\wedge}\zeta c(\alpha\times I) , x\vec{\gamma_\wedge}\zeta(\alpha\times I)) \\
&= i_{d,1}(y\alpha, z(\alpha \times \sigma^{-1})\vec{\gamma_\wedge}\sigma\zeta \alpha d,  x(\alpha \times \sigma^{-1}) \vec{\gamma_\wedge}\sigma\zeta\alpha).
\end{align*}
By construction, the chosen lift $j_{d,\sigma\zeta\alpha}(y\alpha,  z(\alpha \times \sigma^{-1}),  x(\alpha \times \sigma^{-1}))$
is the diagonal composite
\[
\begin{tikzcd}[row sep=huge, column sep=2em]
B\underset{D}{\cup} D \times I^k  \arrow[d, tail, "{\langle [\sigma\zeta\alpha], d \times I^k \rangle}" description] \arrow[r] \arrow[dr, phantom, "\lrcorner" very near start] &[-1.5em]
\left( D \times I^k \times I\underset{D \times I}{\cup} B \times I  \right) \bigcup\limits_{{D \times I^k  \times \{0\}\underset{D \times \{0\}}{\cup}  B \times \{0\}}} B \times I^k \times \{0\}\arrow[r, "B \times ! \times I"] \arrow[d, tail, "{\langle [\sigma\zeta\alpha], d \times I^k \rangle \hat{\times} \partial_0}"']  & B \times I \arrow[r, "{i_{d,1}(\cdots)}"] \arrow[d, "{\vec{\gamma_\wedge}\sigma\zeta\alpha}" pos=.7] & Y \arrow[d, "f", two heads] \\ B \times I^k \times \{1\} \arrow[r, "{B \times I^k \times \partial_1}"'] & B \times I^k \times I \arrow[r,  "{B \times \gamma_\wedge}"']  \arrow[urr, dashed, "{i_{{\langle d \times I^k, [\sigma\zeta\alpha] \rangle},0}(i_{d,1}(\cdots) !, x (\alpha \times \sigma^{-1}) \gamma_\wedge)}" description, pos=.3] & B \times I^k \arrow[r, "x(\alpha \times \sigma^{-1})"'] & X \rlap{,}
\end{tikzcd}
\]
while $j_{c,\zeta}(y,z,x) \cdot (\alpha \times\sigma^{-1})$ is the restriction of the diagonal composite
\[
\begin{tikzcd}[row sep=huge, column sep=2em]
A \underset{C}{\cup} C \times I^k \arrow[d, tail, "{\langle [\zeta], c \times I^k\rangle}"'] \arrow[r] \arrow[dr, phantom, "\lrcorner" very near start] &[-1.5em]
\left( C \times I^k \times I\underset{C \times I}{\cup} A \times I  \right)\bigcup\limits_{{C \times I^k  \times \{0\}\underset{C \times \{0\}}{\cup}  A \times \{0\}}} A \times I^k \times \{0\}  \arrow[r, "A \times ! \times I"] \arrow[d, tail, "{\langle [\zeta], c \times I^k \rangle \hat{\times} \partial_0}"']  & A \times I \arrow[r, "{i_{c,1}(\cdots)}"] \arrow[d, "{\vec{\gamma_\wedge}\zeta}" pos=.7] & Y \arrow[d, "f", two heads] \\ A \times I^k \times \{1\} \arrow[r, "{A \times I^k \times \partial_1}"'] & A \times I^k \times I \arrow[r, "{A \times \gamma_\wedge}"']  \arrow[urr, dashed, "{i_{{\langle c \times I^k, [\zeta] \rangle},0}(i_{c,1}(\cdots) !, x \gamma_\wedge)}" description, pos=.3] & A \times I^k \arrow[r, "x"'] &X
\end{tikzcd}
\]
along $\alpha \times \sigma^{-1} \colon B \times I^k \to A \times I^k$. Observe that we can perform these two restrictions needed to compute $j_{c, \zeta}(y,z,x) \cdot (\alpha \times\sigma^{-1})$ in either order, on account of the commutative cube
\[
\begin{tikzcd}
B\underset{D}{\cup} D \times I^k  \arrow[dd, tail, "{\langle [\sigma\zeta\alpha], d \times I^k\rangle}"'] \arrow[dr, hook] \arrow[rr, "\alpha\cup_\alpha (\alpha \times \sigma^{-1}) "] & &  A\underset{C}{\cup} (C\times I^k)  \arrow[dd, tail, "{\langle [\zeta], c \times I^k \rangle}"' pos=.3] \arrow[dr, hook]  \\ & \bullet \arrow[ddrr,phantom, "\lrcorner" pos=.05] \arrow[rr, crossing over]  && \bullet  \arrow[dd, tail, "{\langle [\zeta] , c \times I^k\rangle \hat{\times} \partial_0}"]  \\ B \times I^k \times \{1\} \arrow[rr, "\alpha \times \sigma^{-1}" pos =.7] \arrow[dr, hook, "B \times I^k \times \partial_1"'] & & A \times I^k \times \{1\}\arrow[dr, hook, "A \times I^k \times \partial_1"] \\ & B \times I^k \times I \arrow[rr, "\alpha \times \sigma^{-1} \times I"']  \arrow[from=uu, crossing over, tail, "{\langle [\sigma\zeta\alpha] , d \times I^k\rangle \hat{\times} \partial_0}" description, pos=.7] & & A \times I^k \times I \rlap{.}
\end{tikzcd}
\]
Thus:
\begin{align*}
j&{}_{c,\zeta}(y,z, x) \cdot (\alpha \times\sigma^{-1})
\\ & = i_{{\langle c \times I^k, [\zeta] \rangle},0}(i_{c,1}(y, z\vec{\gamma_\wedge}\zeta c,  x \vec{\gamma_\wedge}\zeta) !, x \gamma_\wedge) \cdot (A \times I^k \times \partial_1) \cdot (\alpha \times \sigma^{-1})
\\& = i_{{\langle c \times I^k, [\zeta] \rangle},0}(i_{c,1}(y, z\vec{\gamma_\wedge}\zeta c, x \vec{\gamma_\wedge}\zeta) !, x \gamma_\wedge) \cdot  (\alpha \times \sigma^{-1} \times I) \cdot (B \times I^k \times \partial_1)
\\ \intertext{
Note further that the front face of this cube is a pullback, since it arises as the pushout product of the pullback in the back face with $\partial_1 \colon \{1\} \rightarrowtail I$. By uniformity of $(f,i)$ in this pullback square:}
& = i_{{\langle d \times I^k, [\sigma\zeta\alpha] \rangle},0}(i_{c,1}(y, z\vec{\gamma_\wedge}\zeta c, x \vec{\gamma_\wedge}\zeta) (\alpha \times I) !, x \gamma_\wedge(\alpha\times \sigma^{-1}))  \cdot (B \times I^k \times \partial_1)
\\ \intertext{By the uniformity calculation above, the domains of these lifting problems coincide. Thus:}
&= {i_{{\langle d \times I^k, [\sigma\zeta\alpha] \rangle},0}(i_{d,1}(y\alpha,  z(\alpha \times \sigma^{-1})\vec{\gamma_\wedge}\sigma\zeta \alpha d, x(\alpha \times \sigma^{-1}) \vec{\gamma_\wedge}\sigma\zeta\alpha) !, x (\alpha \times \sigma^{-1}) \gamma_\wedge)} \cdot (B \times I^k \times \partial_1)
\\&= j_{d,\sigma\zeta\alpha}( y\alpha, z(\alpha \times \sigma^{-1}), x(\alpha \times \sigma^{-1})) \rlap{,}
\end{align*}
which is the required equivariant uniformity condition.
\end{proof}

\begin{lem}\label{lem:i*-left-quillen}
  The functor $i^* \colon \cSet \to \sSet$ defines a left Quillen functor from the equivariant model structure to the classical model structure.
\end{lem}
\begin{proof}
  To prove that triangulation is left Quillen, it suffices to show that the right adjoint $i_*$ carries Kan fibrations to equivariant fibrations of cubical sets, for which it suffices to show that Kan fibrations lift against the image of the generating category of Construction \ref{con:equivariant-generating-tcof} under the functor $i^*$. After triangulation, the objects and morphisms in this generating category have the form
  \[
    \begin{tikzcd}
  (\Delta^1)^m  \cup_{D } D \times (\Delta^1)^k
   \arrow[d, tail, "{\langle [\xi], d \times 1\rangle}"'] \arrow[r, "\alpha \times \sigma^{-1}"] \arrow[dr, phantom, "\lrcorner" very near start] & (\Delta^1)^n  \cup_{C } C\times (\Delta^1)^k \arrow[d, tail, "{\langle [\zeta], c \times 1\rangle}"]  \\ (\Delta^1)^m \times  (\Delta^1)^k \arrow[r, "\alpha \times \sigma^{-1}"'] &  (\Delta^1)^n  \times (\Delta^1)^k
  \end{tikzcd}
\] where $C$ and $D$ are triangulations of cubical subsets of the $n$-cube and $m$-cube respectively. Thus, the equivariance of Kan fibrations established in Proposition \ref{prop:kan-is-equivariant} defines uniform lifts against these squares.
\end{proof}

To prove that the left Quillen functors of Lemmas \ref{lem:i!-left-quillen} and \ref{lem:i*-left-quillen} define Quillen equivalences, we appeal to the general theory of Eilenberg--Zilber categories, which we now review.

\subsection{Eilenberg--Zilber categories}\label{ssec:EZ}

The categories $\DDelta$ and $\CCube$ are both \emph{Reedy categories}---the former in Dan Kan's original ``strict'' sense and the latter in the ``generalized'' sense of \cite{BergerMoerdijk:2008rc}---that are moreover \emph{Eilenberg--Zilber categories}, defined below.  These properties enable inductive arguments concerning the monomorphisms in the presheaf categories $\sSet$ and $\cSet$ respectively.

A Reedy category $\cA$ comes with classes of ``degree-decreasing'' and ``degree-increasing maps,'' defined relative to a degree function $\deg \colon \ob\cA \to \NN$. In the case of Eilenberg--Zilber categories, defined below, the degree-decreasing maps are the split epimorphisms, while the degree-increasing maps are the monomorphisms.

\begin{defn}[{\cite[6.7]{BergerMoerdijk:2008rc}}]\label{defn:eilenberg-zilber-category}
An \textbf{Eilenberg--Zilber} category is a small category $\cA$ equipped with a degree function $\deg \colon \ob\cA \to \NN$  so that
\begin{enumerate}
\item Isomorphisms preserve the degree, whereas non-invertible monomorphisms or split epimorphisms strictly raise and lower the degree, respectively, when moving from their domain to their codomain.
\item Every $f \in \mor{\cA}$ may be factored as a split epimorphism followed by a monomorphism.
\item Any pair of split epimorphisms with common domain has an \textbf{absolute pushout}: a pushout in $\cA$ that is preserved by the Yoneda embedding $\yo \colon \cA \hookrightarrow \Set^{\cA^\op}$.
\end{enumerate}
\end{defn}

Berger and Moerdijk observe that $\DDelta$ is an Eilenberg--Zilber category \cite[6.8]{BergerMoerdijk:2008rc}. By \cite[Theorem 8.12(1)]{Campion:2023ez}, the cartesian cube category is as well (as could also be checked by directly verifying that each pair of epimorphisms in $\CCube$ with common domain has an absolute pushout).

We review a few results from general Reedy category theory \cite{RiehlVerity:2014rc,Riehl:generalized-reedy} and then explain what is special about Eilenberg--Zilber categories. Let $\cA$ be an Eilenberg--Zilber category and write $\cA\in \Set^{\cA^\op \times \cA}$ for the hom bifunctor of arrows in $\cA$. Let
\[ \sk_n\cA \hookrightarrow\cA \in \Set^{\cA^\op \times \cA}\] denote the subfunctor of arrows of degree at most $n$, by which we mean arrows that factor through an object of degree $n$.

\begin{defn}[boundaries of representable functors]
For $a \in \cA$, write $\cA_a \in \Set^{\cA}$ and $\cA^a \in \Set^{\cA^\op}$ for the co- and contravariant representable functors.  If $a\in \cA$ has degree $n$, write
\begin{align*}
\lbound{\cA}_a &\coloneq \sk_{n-1}\cA_a \qquad \in \Set^\cA \qquad \mathrm{and} \\ \rbound{\cA}^a& \coloneq \sk_{n-1}\cA^a \qquad \in\Set^{\cA^\op}.
\end{align*}
\end{defn}

 The external (pointwise) product defines a bifunctor $\Set^\cA \times \Set^{\cA^\op} \xrightarrow{-\utimes-} \Set^{\cA^\op\times\cA}$. For any $a\in\cA$,  the exterior Leibniz product
\begin{equation}\label{eq:leibniz-inclusion} \begin{tikzcd}  \cA_a\utimes\rbound{\cA}^a  \cup \lbound{\cA}_a\utimes\cA^a \arrow[rrrr, hook, " {(\lbound{\cA}_a \hookrightarrow\cA_a) \leib\utimes (\rbound{\cA}^a\hookrightarrow\cA^a)} "] & & & & \cA_a \utimes \cA^a
\end{tikzcd}
\end{equation}
defines the subfunctor of pairs of morphisms $h\cdot g$ with $\dom(h) = \cod(g) = a$ in which at least one of the morphisms $g$ and $h$ has degree less than the degree of $a$. There is a natural ``composition'' map whose domain is the external product of the contravariant and covariant representables \begin{equation}\label{eq:composition-map} \cA_a\utimes \cA^a \xrightarrow{\circ} \cA.\end{equation} Its image is the  subfunctor of arrows in $\cA$ that factor through $a$, but \eqref{eq:composition-map} is not in general a monomorphism: e.g., this fails to be the case whenever $a$ has non-identity automorphisms.

By Definition \ref{defn:eilenberg-zilber-category}(i), the groupoid core $\cG \subset \cA$ of a Reedy category decomposes as a coproduct $\cG = \coprod_{n \in \NN} \cG(n)$, where $\cG(n)$ is the subgroupoid of isomorphisms between objects of degree $n$. Any isomorphism in $\cA$ restricts in the obvious way to a natural isomorphism between the boundaries of the corresponding representable functors, which thus assemble into profunctors
\[ \lbound{\cA}_n\hookrightarrow \cA_n \in \Set^{\cG(n)^\op \times \cA} \qquad \mathrm{and} \qquad \rbound{\cA}^n \hookrightarrow\cA^n \in \Set^{\cA^\op\times \cG(n)} .\]
When we compose these profunctors over $\cG(n)$, we obtain a profunctor from $\cA$ to $\cA$ which is the ``generalized cell'' attached to form $\sk_n\cA$ from $\sk_{n-1}\cA$ \cite[\S4]{Riehl:generalized-reedy}:

\begin{thm}\label{thm:cellular-decomposition} The inclusion $\emptyset\hookrightarrow \cA $ has a canonical presentation as a generalized cell complex:
\[
\begin{tikzcd}[column sep=small]
& &  \lbound{\cA}_n \utimes_{\cG(n)} \cA^n \cup \cA_n \utimes_{\cG(n)} \rbound{\cA}^n \arrow[dr, phantom, "\ulcorner" very near end] \arrow[d, "\circ"'] \arrow[r, hook] & \cA_n \utimes_{\cG(n)} \cA^n \arrow[d, "\circ"] \\
\emptyset \arrow[r, hook] & \sk_0\cA \arrow[r, dotted] & \sk_{n-1}\cA \arrow[r, hook] & \sk_n\cA  \arrow[r, dotted] & \colim_n \sk_n\cA \cong \cA \rlap{,}
\end{tikzcd}
\]
 i.e., a composite of pushouts of cells constructed as coends of exterior Leibniz products \[  (\lbound{\cA}_n \hookrightarrow\cA_n)\leib\utimes_{\cG(n)}(\rbound{\cA}^n\hookrightarrow\cA^n) \coloneq \int^{a \in \cG(n)} (\lbound{\cA}_a \hookrightarrow\cA_a) \leib\utimes (\rbound{\cA}^a\hookrightarrow\cA^a),\] attached at stage $n$. \qed
\end{thm}

As a corollary of Theorem \ref{thm:cellular-decomposition}, any natural transformation $f \colon X \to Y \in \cE^{\cA^\op}$ valued in a cocomplete category $\cE$ admits a canonical presentation as a generalized cell complex, obtained by applying the Leibniz construction to the weighted colimit bifunctor $\ast_\cA \colon \Set^{\cA^\op \times \cA} \times \cE^{\cA^\op} \to \cE^{\cA^\op}$.

\begin{cor}\label{cor:cellular-decomposition} Let $\cA$ be a  Reedy category and let $\cE$ be bicomplete. Any morphism $f\colon X \to Y \in \cE^{\cA^\op}$ is a generalized cell complex
\[
X \to X \cup_{\sk_0X}\sk_0Y \to \cdots \to X \cup_{\sk_{n-1}X} \sk_{n-1}Y \to X \cup_{\sk_nX} \sk_nY \to \cdots \to \colim \cong Y\]
with the generalized cell
\begin{equation}\label{eq:latching-cell} (\rbound{\cA}^n\hookrightarrow \cA^n) \leib\ast_{\cG(n)} \latch{n}f\end{equation}
attached at stage $n$.
\qed
\end{cor}

Here $\latch{n}f \in \cE^{\cG(n)^\op}$ is the diagram formed by the Leibniz weighted colimit of $f$ and $\lbound{\cA}_n\hookrightarrow\cA_n$. Its component at $a \in \cA$ of degree $n$ is the \textbf{relative latching map}, the Leibniz weighted colimit
defined by the pushout of the map $L_af \coloneq \lbound{\cA}_a \ast_\cA f$: \[
\begin{tikzcd} & &
L_aX \arrow[r, "L_af"] \arrow[r] \arrow[d]\arrow[dr, phantom, "\ulcorner" very near end] & L_aY \arrow[d] \arrow[ddr, bend left] & &   \\ \latch{a}f \coloneq (\lbound{\cA}_a \hookrightarrow\cA_a) \leib\ast_\cA f & &  X_a \arrow[r] \arrow[drr, bend right, "f_a"'] & \ell_a f \arrow[dr, dashed, "{\latch{a}f}" description]  \\ & &  & & Y_a \rlap{.}
\end{tikzcd}
\]

We now specialize to the case $\cE=\Set$ and impose the Eilenberg--Zilber hypothesis on $\cA$.  Let $X$ be a presheaf on an Eilenberg--Zilber category $\cA$. An element $x \in X_a$ is \textbf{degenerate} if there exists a non-invertible split epimorphism $\pi \colon a \twoheadrightarrow b$ and a $y \in X_b$ so that $x = y \pi$; and \textbf{non-degenerate} otherwise. For degenerate $x$, we refer to the factorization $x = y \pi$ as an \textbf{Eilenberg--Zilber decomposition} of $x$.
As observed in \cite[6.9--10]{BergerMoerdijk:2008rc}, the axioms of Definition \ref{defn:eilenberg-zilber-category} imply that Eilenberg--Zilber decompositions are essentially unique, which implies that the latching maps $L_aX \rightarrowtail X_a$ are monomorphisms whose images are the degenerate elements. Moreover, the following relative version of this result holds:

\begin{lem}\label{lem:EZ-injective-reedy-monos}
Let $\cA$ be an Eilenberg--Zilber category. Then for all $f \colon X \to Y$ in $\Set^{\cA^{\op}}$ each relative latching map $\latch{a}f$ is a monomorphism if and only if each component $f_a \colon X_a \rightarrowtail Y_a$ is a monomorphism, and either hypothesis implies that for each $a \in \cA$, the latching square below is a pullback:
\[
\begin{tikzcd} L_aX \arrow[d, tail] \arrow[r, tail, "L_af"] \arrow[dr, phantom, "\lrcorner" very near start] & L_aY \arrow[d, tail] \\ X_a \arrow[r, tail, "f_a"'] & Y_a \rlap{.}
\end{tikzcd}
\]
\end{lem}
\begin{proof}
When $f \colon X \to Y$ is a monomorphism, each map in the latching square is a monomorphism, and it is easy to see that the latching square is a pullback. It suffices to show that $L_aX$ surjects onto the pullback $X_a \times_{Y_a} L_aY$. If the image of $x \in X_a$ is degenerate, with $f(x) = y' \cdot \epsilon$, then we may choose a section $\delta$ of $\epsilon$ and observe that $x$ and $x \cdot \delta \cdot \epsilon$ have the same image under $f$, proving that $x$ is degenerate. Thus the latching square is a pullback and then the relative latching map is a monomorphism, the union of the subobjects of $Y_a$.

The converse implication holds for general Reedy categories without the Eilenberg--Zilber hypothesis \cite[\S 8]{Riehl:generalized-reedy}.
\end{proof}

Lemma \ref{lem:EZ-injective-reedy-monos} may be summarized by saying that when $\cA$ is an Eilenberg--Zilber category, the injective Reedy monomorphisms, defined below, are just the pointwise monomorphisms.

\begin{defn}[Berger--Moerdijk]
  A map $f \colon X \to Y$ in $\Set^{\cA^\op}$ is an \textbf{injective Reedy monomorphism} if for all $a \in \cA$, the map $\latch{a}f$ is a monomorphism.
\end{defn}

The injective Reedy monomorphisms form the left class of a weak factorization system that is left-lifted along the left adjoint $\latch{\bullet}-$ displayed below from the (monomorphism, equivariant split epimorphism) weak factorization system on $\Set^{\cG^\op}$, which in turn is the ``injective'' or left lifting of the (monomorphism, split epimorphism) weak factorization system on $\Set^{\ob\cA}$:\footnote{Projective Reedy weak factorization systems may be defined similarly using the ``projective'' or right lifting to $\Set^{\cG}$ \cite[1.6, 1.8]{BergerMoerdijk:2008rc}.}
\[ \begin{tikzcd} {\Mono_\inj[\cA]} \arrow[d] \arrow[r] \arrow[dr, phantom, "\lrcorner" very near start] & \Mono_\inj \arrow[d]  \\ (\Set^{\cA^\op})^\2 \arrow[r, "{\latch{\bullet}-}"'] & (\Set^{\cG^\op})^\2 \rlap{.} \end{tikzcd}\]

When $\cA$ is an Eilenberg--Zilber category, Corollary \ref{cor:cellular-decomposition} tells us that any monomorphism $f$ factors as a transfinite composite of pushouts of maps of the form \eqref{eq:latching-cell} where $\latch{n}f \in \Set^{\cG(n)^\op}$ is a monomorphism. The groupoid $\cG(n)$ of isomorphisms between objects of degree $n$ is equivalent to the disjoint union of the 1-object groupoids associated to automorphism groups $\Aut(a)$, where the disjoint union is over the set of isomorphism classes of objects of degree $n$.\footnote{In both $\CCube$ and $\DDelta$ there is a unique object of degree $n$, but this is not a requirement of the Eilenberg--Zilber axioms.} So $\Set^{\cG(n)^\op}$ is equivalent to the product of categories of the form $\Set^{\Aut(a)^\op}$ where $\deg(a)=n$.

Thus, we study the (injective monomorphism, injective split epimorphism) weak factorization system on the category $\Set^{G^\op}$ of right $G$-sets, for $G$ a group. In this category, the injective monomorphisms are just the monomorphisms, while the injective split epimorphisms are the $G$-split epimorphisms: maps of right $G$-sets that admit a $G$-equivariant section.

\begin{lem}\label{lem:inj-epi-cof-gen} The monomorphisms in $\Set^{G^\op}$ are pushouts of coproducts of the maps  \[ \{ \emptyset \hookrightarrow G_{/H} \}_{H \subset G}\]
where the right $G$-sets are the sets of right cosets $G_{/H}$ of all subgroups $H$ of $G$.
\end{lem}
\begin{proof}
Objects in the category $\Set^{G^\op}$ of right $G$-sets decompose as coproducts of orbits, on which $G$ acts transitively. Each orbit is $G$-equivariantly isomorphic to $G_{/H}$, the right $G$-set of right cosets by a subgroup $H$. The stabilizer groups of the elements in this orbit are then conjugate to $H$. By $G$-equivariance, monomorphisms in this category attach new orbits. Thus, each monomorphism may be expressed as a pushout of maps of the form $\emptyset\hookrightarrow G_{/H}$, for each orbit with stabilizer group $H$ that is not in the image of the domain.
\end{proof}

Putting this together we arrive at the following result:

\begin{prop}\label{prop:EZ-monomorphism-decomposition} Any monomorphism $f\colon X \to Y \in \Set^{\cA^\op}$ between presheaves indexed by an Eilenberg--Zilber category $\cA$ is generalized cell complex
  \[
  X \to X \cup_{\sk_0X}\sk_0Y \to \cdots \to X \cup_{\sk_{n-1}X} \sk_{n-1}Y \to X \cup_{\sk_nX} \sk_nY \to \cdots \to \colim \cong Y\]
  where the cells attached at stage $n$ are coproducts of cells of the form
  \begin{equation}\label{eq:EZ-mono-latching-cell} \rbound{\cA}^a_{/H}\hookrightarrow \cA^a_{/H}\end{equation}
  where $\deg(a) =n$ and $H \subset \Aut(a)$.
\end{prop}
\begin{proof}
  By Lemma \ref{lem:inj-epi-cof-gen}, the relative latching map $\latch{n}f \in \Set^{\cG(n)^\op}$ is a pushout of coproducts of maps of the form $\emptyset\hookrightarrow G(n)^a_{/H}$, where $a \in \cG(n)$, $H \subset \Aut(a)$, and $G(n)^a$ is the contravariant representable. By cocontinuity of the weighted colimit functor and the coYoneda lemma, the cell $(\rbound{\cA}^n \hookrightarrow \cA^n) \leib\ast_{\cG(n)} \latch{n}f$ of \eqref{eq:latching-cell} is then a pushout of coproducts of cells of the form
  \[ (\rbound{\cA}^n\hookrightarrow \cA^n) \leib\ast_{\cG(n)} (\emptyset\hookrightarrow \cG(n)^a_{/H}) \cong  \rbound{\cA}^a_{/H}\hookrightarrow \cA^a_{/H}.\] Thus, by Corollary \ref{cor:cellular-decomposition}, $X \cup_{\sk_{n-1}X}\sk_{n-1}Y \hookrightarrow X \cup_{\sk_nX}\sk_nY$ is a pushout of coproducts of cells of this form.
\end{proof}

\begin{lem}\label{lem:EZ-magic} Let $\cA$ be an Eilenberg--Zilber category. Then the monomorphisms in $\Set^{\cA^\op}$ are generated under coproduct, pushout, sequential composition, and right cancelation among monomorphisms by the maps $\emptyset \to \cA^a_{/H}$ valued in the quotient of a representable presheaf at some $a \in \cA$ by an arbitrary subgroup $H$ of its automorphism group.
\end{lem}

\begin{proof}
  By Proposition \ref{prop:EZ-monomorphism-decomposition}, the monomorphisms are generated under coproduct, pushout, and sequential composition by the maps $\rbound{\cA}^a_{/H} \to \cA^a_{/H}$ for $a \in \cA$ and $H \subset \Aut(a)$. Under right cancelation among monomorphisms
\[ \begin{tikzcd} \emptyset \arrow[r] \arrow[dr] & \rbound{\cA}^a_{/H} \arrow[d, "i^a_H"] \\ & \cA^a_{/H} \rlap{,} \end{tikzcd}\]
these maps are generated by monomorphisms of the form $\emptyset \rightarrowtail \cA^a_{/H}$ and $\emptyset \rightarrowtail \rbound{\cA}^a_{/H}$. We prove that the latter class are generated by the former under coproduct, pushout,  sequential composition, and right cancelation among monomorphisms by induction in the degree of the object $a \in A$.

When $a$ has degree zero, $\rbound{\cA}^a$ is empty, covering the base case of the induction. So we may suppose that $a$ has degree $n$ and our task is to show that $\emptyset \rightarrowtail \rbound{\cA}^a_{/H}$ may be generated under coproduct, pushout,  transfinite composition, and right cancelation among monomorphisms by maps of the form $\emptyset \rightarrowtail \cA^b_{/H}$ with $\deg(b) \leq n$ under the inductive hypothesis that when $\deg(b) < n$, the maps $\emptyset \rightarrowtail \rbound{\cA}^b_{/H}$ are in this class. From right cancelation, this tells us that the maps $\rbound{\cA}^b_{/H} \to \cA^b_{/H}$ are in this class when $\deg(b) < n$. The presheaf $\rbound{\cA}^a_{/H} \in \Set^{\cA^\op}$ is $(n-1)$-skeletal, so from Proposition \ref{prop:EZ-monomorphism-decomposition}, we see that $j^a \colon \emptyset \rightarrowtail \rbound{\cA}^a_{/H}$ factors as a composite of pushouts of coproducts of  the maps $\rbound{\cA}^b_{/K}\hookrightarrow \cA^b_{/K}$ for $K \subset \Aut(b)$, completing the induction.
\end{proof}

We now return to the question of proving that the left Quillen functors $i_!$ and $i^*$ are Quillen equivalences.  As the cofibrations are the monomorphisms, all objects in each of the categories $\sSet$  and $\cSet$ are cofibrant. By Ken Brown's lemma, left Quillen functors preserve weak equivalences between cofibrant objects. Consequently:

\begin{cor}
Each of the functors
\[
\begin{tikzcd}
 \sSet \arrow[r, bend left=30, "i_!"] \arrow[r, phantom, "\bot"] & \cSet  \arrow[l, bend left=30, "i^*"]
\end{tikzcd}
\]
preserves weak equivalences. \qed
\end{cor}

To demonstrate that these functors are inverse left Quillen equivalences, it suffices to show that the total left derived functors define equivalences, for which it suffices to demonstrate that the unit $\eta \colon \id \Rightarrow i^*i_!$ and counit $\epsilon \colon i_! i^* \Rightarrow \id$ are natural weak equivalences. The advantage of working with an inverse pair of left adjoints is that we can use cocontinuity and the fact that both $\DDelta$ and $\CCube$ are Eilenberg--Zilber to reduce to checking that certain components are weak equivalences. In fact, we can treat both cases at once, by an argument we now develop.

\begin{lem}\label{lem:leibniz-magic} Let  $U,V \colon \cK \to \cM$ be a cocontinuous pair of functors valued in a model category and $\alpha \colon U \To V$ a natural transformation between them. Define the cofibrations in $\cK$ to be the maps that are sent to cofibrations under both $U$ and $V$. Define $\NMono$ to be the class of cofibrations between cofibrant objects that are sent by Leibniz pushout application with $\alpha$ to weak equivalences in $\cM$.
  Then $\NMono$ is closed under coproducts, pushouts, (transfinite) composition, and right cancelation among cofibrations.
  \end{lem}

  \begin{proof}
  The claims all follow from the proofs of \cite[\S 5]{RiehlVerity:2014rc}, except for right cancelation, which is not mentioned there. We demonstrate this together with the closure under composition, as these are the most subtle closure properties. Consider a composable pair of monomorphisms and their Leibniz applications:
  \[
  \begin{tikzcd}
  UA \arrow[d, "\alpha_A"'] \arrow[r, tail, "U g"] \arrow[dr, phantom, "\ulcorner" very near end] & UB \arrow[r, tail,  "U h"] \arrow[d] \arrow[dr, phantom, "\ulcorner" very near end] & UC \arrow[d] \arrow[dddrr, bend left, "\alpha_C"] \\
  VA \arrow[r, tail] \arrow[drr, "Vg"', tail] & \bullet \arrow[dr, "{\alpha\check{\circ}g}" description] \arrow[r, tail] \arrow[drr, phantom, "\ulcorner" pos=0.8] & \bullet \arrow[dr] \arrow[ddrr, bend left, "{\alpha\check{\circ}hg}" description] \\  & & VB \arrow[drr,  "V h"', tail] \arrow[r, tail]  & \bullet \arrow[dr, "{\alpha\check{\circ}h}" description] \\ & &  & & VC \rlap{.}
  \end{tikzcd}
  \]
  The diagram reveals that ${\alpha\check{\circ}hg}$ factors as a pushout of ${\alpha\check{\circ}g}$ followed by ${\alpha\check{\circ}h}$. When $g \in \NMono$ and $h$ is a cofibration, our hypotheses imply that the pushout of ${\alpha\check{\circ}g}$ is a pushout of a weak equivalence between cofibrant objects along a cofibration, hence again a weak equivalence. Thus, by the 2-of-3 properties for weak equivalences, $h \in \NMono$ if and only if $hg \in \NMono$.
  \end{proof}

  \begin{cor}\label{cor:EZ-magic}  Let $\cA$ be an Eilenberg--Zilber category and consider a parallel pair of functors $U,V \colon \Set^{\cA^\op} \to \cM$ valued in a model category $\cM$ together with a natural transformation $\alpha \colon U \Rightarrow V$. Suppose that $U$ and $V$ preserve colimits and send monomorphisms in $\cK$ to cofibrations in $\cM$. Then if the components of $\alpha$ at quotients of representables by subgroups of their automorphism groups are weak equivalences, then all components of $\alpha$ are weak equivalences.
  \end{cor}
  \begin{proof}

  Note that the components of $\alpha$ at a presheaf $X$ are obtained by Leibniz application of $\alpha$ at the monomorphism $\emptyset \to X$. The result now follows by combining Lemma \ref{lem:EZ-magic}, which says that the monomorphisms in $\Set^{\cA^\op}$ are generated under coproduct, pushout, transfinite composition, and right cancelation among monomorphisms by the maps $\emptyset \to \cA^a_{/H}$, with Lemma \ref{lem:leibniz-magic}, which says that the class of monomorphisms whose Leibniz applications are weak equivalences has these closure properties.
  \end{proof}

  \begin{cor}\label{cor:derived-equivalence} Let $\cA$ be an Eilenberg--Zilber category for which $\Set^{\cA^\op}$ admits a model structure whose cofibrations are the monomorphisms in which the quotients $\cA^a_{/H}$ of representables by subgroups of their automorphism groups are weakly contractible. Then if $U,V \colon \Set^{\cA^\op} \to \cM$ define a pair of left Quillen functors that preserve the terminal object, then any natural transformation $\alpha \colon U \To V$ is a natural weak equivalence.
  \end{cor}
  \begin{proof}
By Ken Brown's lemma, left Quillen functors from $\Set^{\cA^\op}$ that preserve the terminal object preserve weakly contractible cofibrant objects. Now from the naturality square associated to a weakly contractible cofibrant object $X$
\[ \begin{tikzcd} UX \arrow[d, wearrow, "!"'] \arrow[r, "\alpha_X"] & VX \arrow[d, uwearrow, "!"] \\ U\ast \arrow[r, equals] & V\ast \end{tikzcd}
\] and the 2-of-3 property, we see that the component $\alpha_X$ is a weak equivalence. By Corollary \ref{cor:EZ-magic}, if the components of $\alpha$ at quotients of representables are weak equivalences, then $\alpha$ is a natural weak equivalence. So the result follows.
  \end{proof}

  Note that $i^*$ preserves the terminal object, as a right adjoint, as does $i_!$, since in both domain and codomain it is representable and $i[0] \coloneq [0,1]^0$.

  \begin{prop}\label{prop:triangulation-quillen-equivalences} The functors
      \[ \begin{tikzcd} \cSet \arrow[r, bend right, "i^*"'] \arrow[r, phantom, "\bot"] & \sSet \arrow[l, bend right, "i_!"'] \end{tikzcd}\]
      are left Quillen equivalences.
  \end{prop}
  \begin{proof}
The unit and counit of these adjunctions each define natural transformations between left Quillen adjoints that preserve the terminal object. As the domain and codomain of these functors are categories of presheaves for Eilenberg--Zilber categories equipped with model structures for which all objects are cofibrant and quotients of representables are contractible, Corollary \ref{cor:derived-equivalence} applies to prove that both the unit and counit are natural weak equivalences.
\end{proof}

\subsection{The equivariant model structure is the test model structure}\label{ssec:test}

Finally, we show that the equivariant model structure coincides with the test model structure.

The cartesian cube category is a \emph{strict test category} \cite{BuchholtzMorehouse:2017vo}, so cartesian cubical sets admits a model structure, conjectured to exist by Grothendieck \cite{Grothendieck:pursuing} and established at this level of generality by Cisinski \cite{Cisinski:2006lp}, that presents classical homotopy theory. In Cisinski's model structure on presheaves over a test category---referred to as a \textbf{test model structure} below---the cofibrations are the monomorphisms and the weak equivalences are those maps of presheaves $f \colon X \to Y$ such that the map of simplicial sets defined by applying the functor $N\elf$, which takes the nerve of the category of elements, is a weak homotopy equivalence.

\begin{defn} A category is \textbf{aspherical} if its nerve is weakly contractible in Quillen's model structure.
  A functor $u \colon \cA \to \cB$ between small categories is \textbf{aspherical} if the comma category $u \downarrow b$ is aspherical for every $b \in \cB$.
  A presheaf over a small category is \textbf{aspherical} if its category of elements is aspherical.
\end{defn}

Note that, by definition, a presheaf over a test category is aspherical if and only if it is weakly contractible in the test model structure.

\begin{rmk}[{\cite[7.14]{CavalloSattler:2022re}}] \label{rmk:test-simplicial-is-kan} The test model structure on $\sSet$ is the Kan--Quillen model structure.
  In particular, a simplicial set is aspherical if and only if it is weakly contractible in the Kan--Quillen model structure.
\end{rmk}

Now we can use the machinery of aspherical functors to relate the test model structure on $\cSet$ to the Kan--Quillen model structure.

\begin{prop}[{\cite[4.2.24]{Cisinski:2006lp}}]
  \label{prop:aspherical-functor-quillen-equivalence}
  Let $u \colon \cA \to \cB$ be an aspherical functor between test categories. Then the adjunction
  \[
    \begin{tikzcd}
      \Set^{\cB^\op} \arrow[rr, phantom, "\bot"] \arrow[rr, bend left, "u^*"] && \Set^{\cA^\op} \arrow[ll, bend left, "u_*"]
    \end{tikzcd}
  \]
  defines a Quillen equivalence between test model structures.
\end{prop}

\begin{prop}[{\cite[4.2.23]{Cisinski:2006lp}}]
  \label{prop:aspherical-functor-characterization}
  A functor $u \colon \cA \to \cB$ between small categories is aspherical if and only if $u^*(\yo b)$ is aspherical for every $b \in \cB$.
\end{prop}
\begin{proof}
  The category of elements of $u^*(\yo b)$ is equivalent to the comma category $u \downarrow b$.
\end{proof}

\begin{cor}
  \label{cor:delta-to-cube-aspherical}
  The functor $i \colon \DDelta \to \CCube$ is aspherical.
\end{cor}
\begin{proof}
  By Proposition \ref{prop:aspherical-functor-characterization}, we want to show that $i^*I^n \in \sSet$ is an aspherical presheaf for each $n \in \cN$.
  By Remark \ref{rmk:test-simplicial-is-kan}, this means showing $i^*I^n$ is contractible in the Kan--Quillen model structure.
  We have $i^*I^n \cong (\Delta^1)^n$ by Lemma \ref{lem:restriction-is-triangulation}, so this is indeed the case.
\end{proof}

\begin{thm}
  The equivariant model structure on $\cSet$ coincides with the test model structure.
\end{thm}
\begin{proof}
  These two model structures have the same cofibrations, so it suffices to show they have the same weak equivalences.
  Recall that a left Quillen equivalence preserves and reflects weak equivalences between cofibrant objects.
  Thus, by Proposition \ref{prop:triangulation-quillen-equivalences}, a map $f$ is a weak equivalence in the equivariant model structure if and only if $i^*f$ is a weak equivalence.
  But by Proposition \ref{prop:aspherical-functor-quillen-equivalence} and Corollary \ref{cor:delta-to-cube-aspherical}, $f$ is also a weak equivalence in the test model structure if and only if $i^*f$ is a weak equivalence.
  Thus the weak equivalences of the equivariant and test model structures coincide.
\end{proof}

\newcommand{\PAIR}[1]{\langle #1\rangle}

\newcommand{\textformal}[1]{\texttt{#1}}
\newcommand{\formalurl}{https://ecavallo.github.io/equivariant-cartesian}
\newcommand{\formalmodule}[1]{\href{\formalurl/#1.html}{\textformal{#1}}}
\newcommand{\formaldef}[3]{\href{\formalurl/#1.html\##3}{\textformal{#1.#2}}}

\newcommand{\Elem}{\mathsf{Elem}}
\newcommand{\Univx}{\mathcal{V}}
\newcommand{\Flat}{\flat}

\newcommand{\Ival}{\mathsf{I}}
\newcommand{\UnivCof}{{\Phi}}
\newcommand{\toptt}{\mathsf{tt}}

\newcommand{\Contr}{\mathsf{Contr}}
\newcommand{\LocalFillStr}{\mathsf{LocalFill}}
\newcommand{\Comp}{\mathsf{Fill}}
\newcommand{\Transp}{\mathsf{Transp}}
\newcommand{\Equivariant}{\mathsf{Equivariant}}
\newcommand{\FibStr}{\mathsf{Fib}}

\newcommand{\PathTy}{\mathsf{Path}}
\newcommand{\EquivTy}{\mathsf{Equiv}}
\newcommand{\GlueTy}{\mathsf{Glue}}
\newcommand{\WeakGlueTy}{\mathsf{WeakGlue}}
\newcommand{\isContrTy}{\mathsf{isContr}}
\newcommand{\Univf}{\mathcal{U}}

\newcommand{\draShut}{\mathsf{shut}}
\newcommand{\draOpen}{\mathsf{open}}

\appendix

\section{Type-theoretic development and formalization}\label{app:type-theoretic}

\subsection{Introduction}

This appendix provides a description of the equivariant cartesian cubical set model in the language of dependent type theory.
The category of presheaves on any index category models an {\em extensional} dependent type theory, such as the one introduced by Martin-L\"of \cite{ml79}, as observed by Hofmann \cite[\S4]{hofmann97} and detailed by Awodey, Gambino, and Hazratpour \cite{AGH:2021}.
Briefly, contexts are interpreted as presheaves, and
a type $A$ in context $\Gamma$ is interpreted as a map $A : \Gamma \to \mathcal{V}$, where $\dot{\mathcal{V}} \to \mathcal{V}$ is a classifier for small maps of presheaves as in \S\ref{ssec:realignment} above.\footnote{Alternatively, in the style of Hofmann, a type $A$ in context $\Gamma$ is interpreted as a presheaf on the category of elements $\int{\Gamma}$. Small types (presheaves valued in small sets) then correspond to maps $\Gamma \to \mathcal{V}$ as above.} Starting from type-theoretic axioms that capture the basic structure of cartesian cubical sets (e.g.\ an interval object), we can construct a translation, or \emph{internal model}, of HoTT in extensional type theory, in such a way that the usual functorial, or \emph{external}, notions are recovered under the interpretation into presheaves, again as detailed in \emph{op.cit}.  This was the approach used in the formalization by Orton and Pitts \cite{OrtonPitts:2018af}.

The internal homotopical model interprets contexts as types of the extensional theory, while types are interpreted as type families equipped with \emph{equivariant filling structure}. Most of the required axioms can be formulated within plain extensional type theory, augmented by the cubical axioms; however, in order to interpret (univalent) universes, we follow Licata et al.\ \cite{LOPS18} in using a modal operator to refer to the set of global sections of a presheaf, an external notion that falls outside the type theory of the category of presheaves itself.

This approach has the practical advantage that uniformity conditions on filling structures need not be stated
and checked explicitly, as such conditions are in effect built into the presheaf interpretation (see \cite[\S\S7--8]{AGH:2021}). It has the theoretical benefit that the results can be interpreted in models of extensional type theory other than cubical sets.
For example, Uemura \cite{Uemura18} constructs a model of HoTT in cubical assemblies in this way.  Finally, this approach is amenable to direct formalization in a proof assistant. Beginning from an axiomatization similar to the one in \cite{OrtonPitts:2018af,LOPS18}, all of the material presented in this appendix has been formalized in the proof assistant Agda \cite{agda}.  The formal development can be found at \cite{formalization}, and we include references to relevant definitions from the formalization in our summary below.

Variations on this kind of internal model construction have been presented in detail elsewhere
\cite{OrtonPitts:2018af,LOPS18,Uemura18,Boulier20,CMS20,ABCFHL}, so we limit ourselves to a high level description and
some points that are not stressed in those references. For the sake of concision, we start from simpler but more
restrictive axioms than in the formal development; the additional generality is not principally motivated
by applications, but by ease of formalization. We refer readers interested in a more parsimonious
axiomatization to the documentation at \cite{formalization}.

\subsubsection{Metatheory}

The metatheory can be classical set theory with Grothendieck universes, or a constructive version such as Aczel's constructive set theory with universes \cite{aczel:relate}. For each Grothendieck universe in the metatheory, we have a Hofmann--Streicher universe $\Univx$ in the extensional type theory that reflects all type forming operations (as in \cite{ml79}).
The notions of fibred structure represented by these universes satisfy a relative acyclicity property (as used in \S\ref{sec:fibred}) which can be expressed inside the type theory (\formalmodule{axiom.realignment}); it is called the ``strictness axiom'' by Orton and Pitts \cite[Theorem 8.4]{OrtonPitts:2018af} and ``realignment'' by Gratzer, Shulman, and Sterling \cite[Definition~1.1.4]{GratzerSterlingShulman:2022}.

\subsubsection{Comparison with external proofs}

Since we are working in the internal language of cubical \emph{sets}, rather than cubical \emph{species}, we cannot transfer constructions from the latter to the former as is done in the external development (beginning in \S\ref{sec:symmetric}).  This means that we must check equivariance conditions explicitly:\ e.g.\ compare the proof of the Frobenius condition in Proposition \ref{prop:equivariant-frobenius} to that in Proposition \ref{prop:internal-frobenius} below.  It might be possible to instead work internally to cubical species and then transfer the results to cubical sets by working in a type theory with modalities based on the adjoints $\Leftadj \dashv \Delta \dashv \Gamma$ of \S\ref{ssec:species-to-sets}, but we leave this for
future work.

\subsubsection{Quillen model structure and fibrant replacement}

We formalize and describe here only an interpretation of HoTT; we do not build an associated Quillen model
structure. Boulier and Tabareau \cite{Boulier20} have extended Orton and Pitts' type-theoretic model of HoTT
\cite{OrtonPitts:2018af,LOPS18} (which axiomatizes cubical sets with connections) to obtain a model
structure on the category of types in the universe $\Flat \Univx$ of global presheaves (see \S\ref{ssec:tt-universes} below for a discussion of the $\Flat$ modality). We conjecture that their work adapts to the equivariant cartesian case.

One difference is in the definition of a fibrant replacement, or more generally the factorization for the
(trivial cofibration, fibration) factorization system. In our external development, this is obtained via the algebraic small object argument from a generating category transferred
from an algebraic weak factorization system on cubical species (\S
\ref{ssec:cylindrical-premodel-equivariant}). Boulier and Tabareau derive their fibrant replacement from a
postulated \emph{quotient inductive type (QIT)} \cite{AltenkirchCDKF18}. In our formalization we postulate a
similar QIT for (trivial cofibration, fibration) factorization (\formalmodule{axiom.fibrant-replacement}) and
derive a universal property (\formalmodule{fibration.fibrant-replacement}), though we do not need
this construction for the interpretation of HoTT. It is worth noting that unlike fibrant replacement in
non-equivariant cubical models, equivariant fibrant replacement does not seem to be expressible as a
\emph{W-type with reductions} in the sense of Swan \cite{swan18wtypes}. The construction of fibrant
replacement as a subset of an upper approximation in Coquand, Huber, and M\"ortberg's constructions of higher
inductive types \cite{CHM18} has to be replaced by a quotient by a partial equivalence relation on this
upper approximation.\footnote{We note that, in the case of the cartesian model, the upper approximation can be
described by a {\em finitary} inductive definition, so choice is not needed for proving the required property
of this quotient.}

\subsection{Judgments of the homotopical interpretation}
\label{ssec:tt-judgments}

From here forward, we work inside an extensional type theory, which we will call the \emph{ambient}
theory. We will introduce the necessary axioms as we go along, but first we can set up the judgmental
structure of the homotopical interpretation.  A \emph{context} in the homotopical model is a type of the
ambient theory. A \emph{substitution} between contexts is a function between types. The unit type $1$ serves
as the empty context. A \emph{type} of the homotopical model over a context $\Gamma$ is a family over $\Gamma$
paired with an equivariant filling structure, which we will define in \S\ref{ssec:app-fibration} below.
The \emph{terms} of a type over $\Gamma$ in the model are the global sections of the family $A$ underlying the
type, i.e.\ the elements of $\Elem~\Gamma~A \coloneq \Pi_{\gamma:\Gamma}A~\gamma$.

If $A$ is a family of types over $\Gamma$, we write $\Gamma.A$ for the type $\Sigma_{\gamma:\Gamma}A~\gamma$.
Thus an element of $\Gamma.A$ is a pair $\gamma,a$ with $\gamma$ in $\Gamma$ and $a$ in $A~\gamma$. If $A$ is
the underlying family of a type in the model, then we take $\Gamma.A$ as the interpretation of the context
extension of $\Gamma$ by that type.

\subsection{Cubes and cofibrations}

We assume as given two special types: an \emph{interval} type $\Ival$ with two distinct elements $0\neq 1$ (\formaldef{axiom.shape}{$\mathbbe{i}$}{\%f0\%9d\%95\%9a}) and a type of \emph{cofibrations} $\UnivCof$ (\formalmodule{axiom.cofibration}). These types are in all universes: we have $\UnivCof:\Univx$ and $\Ival:\Univx$. For
each $n : \NN$, the $n$-cube $\Ival^n : \Univx$ is then the cartesian product of $n$ copies of the interval. To
each cofibration $\psi$ is associated a proposition $[\psi] : \Univx$, where a type $A$ is a {\em proposition} if
it is a subsingleton, i.e.\ we have $a_0 = a_1 : A$ for any $a_0$ and $a_1$ in $A$.

We assume that $[-]$ embeds $\UnivCof$ as a sublattice of the lattice of propositions. In particular, for
$\psi,\phi : \UnivCof$ we have $\psi \lor \phi : \UnivCof$ such that $[\psi \lor \phi]$ is the union of $[\psi]$ and
$[\phi]$, and we have a true cofibration $\top : \UnivCof$ inhabited by some $\toptt : [\top]$. In this summary, we
assume \emph{cofibration extensionality}: if $[\psi]$ and $[\phi]$ are logically equivalent then
$\psi = \phi$.\footnote{In the formal development, we assume cofibration extensionality only to define Swan's
  strict identity types \cite[\S9.1]{CCHM:2018ctt} (\formalmodule{type-former.swan-identity}).} In particular,
given $x : [\psi]$ we have $\psi = \top$ and thus $x = \toptt$.

The model in cartesian cubical sets described in the main article corresponds to taking the representable
1-cube as the interval and the subobject classifier $\Omega$ as the type of cofibrations; the decoding
function $[-] : \Omega \to \Univx$ is the classifying map $\top \colon 1 \to \Omega$ regarded as a type family over
$\Omega$. If working constructively, we can instead take the classifier for {\em levelwise decidable}
subobjects, those monomorphisms $m \colon A \rightarrowtail B$ in $\cSet$ such that the component
$m_k \colon A_k \rightarrowtail B_k$ is decidable for each $k \in \NN$.

\begin{rmk}[\formalmodule{axiom.shape}]
  In the formal development, we do not work with cubes defined explicitly as products of an interval. Instead,
  we assume an abstract type $\mathsf{Shape}$ and a decoding function giving $\langle S \rangle : \Univx$ for
  each $S : \mathsf{Shape}$. We require that the interval $\Ival$ is coded by a shape, but not that every shape
  is a power of $\Ival$, nor that $\Ival^n$ is coded by a shape for $n \neq 1$. To obtain the equivariant fibration model, we would instantiate with
  $\mathsf{Shape} \coloneq \NN$ and $\langle n \rangle \coloneq \Ival^n$. We can also recover the
  non-equivariant model by taking $\Ival$ to be the only shape.
\end{rmk}

\subsection{Partial elements and contractible types}
\label{sec:tt-trivial-fibration}

The notion of {\em partial} elements and {\em contractible} types
play a crucial role in this internal description. Both definitions use only the type of cofibrations $\UnivCof$
and not the interval type $\Ival$.

\begin{defn}[\formaldef{cofibration}{$\_^+$}{_\%e2\%81\%ba}]
  To each type $A$ we associate a type $A^+ \coloneq \Sigma_{\psi:\UnivCof}A^{[\psi]}$ of partial elements of $A$.
  A partial element of $A$ is thus a pair $\psi,u$ where $u$ is in $A^{[\psi]}$.  The operation $A\mapsto A^+$
  on types is reflected in all universes and so defines a function $\Univx\rightarrow\Univx$.
\end{defn}

There is a canonical injection $i_A:A\rightarrow A^+$ which to any $a : A$ associates the element
$\top,u$ in $A^+$ with $u~x \coloneq a$. Viewed externally, $i_A$ is the partial map classifier introduced in
\S\ref{ssec:trivial-fibrations}, taken relative to the ambient context.

\begin{defn}[\formaldef{fibration.trivial}{Contr}{Contr}]
  For any type $A$, we can consider the type $\Contr(A)$ of {\em contractibility structures} on $A$. This is
  the type of operations $c_A$ which take a partial element $\psi,u$ in $A^+$ and build an element
  $c_A\,(\psi,u)$ in $A$ such that $[\psi]$ implies $c_A\,(\psi,u) = u~\toptt$.
\end{defn}

\begin{rmk}
  Any contractibility structure $c_A$ is a left inverse of $i_A$: we have $c_A\,(i_A~a) = a$ for any $a$ in
  $A$. Maybe surprisingly, the converse also holds: any left inverse $c_A$ of $i_A$ is in $\Contr(A)$, because
  if $c_A$ is a left inverse of $i_A$ then for any $\psi,u$ in $A^+$ we have that $[\psi]$ implies
  $(\psi,u) = i_A\,(u~\toptt)$ and thus $c_A\,(\psi,u) = c_A\,(i_A\,(u~\toptt)) = u~\toptt$.
\end{rmk}

\begin{defn}[\formaldef{fibration.trivial}{TFibStr}{TFibStr}]
  A \emph{trivial fibration structure} on a family of types $A$ over $\Gamma$ then consists of a family of
  contractibility structures on $A~\gamma$ for each $\gamma : \Gamma$.
\end{defn}

Viewed externally, such a family corresponds to a uniform trivial fibration structure in the sense of
Definition \ref{defn:uniform-tfib-structure}.

\subsection{Filling and equivariant filling}
\label{ssec:app-fibration}

Next we finish defining the interpretation of types by defining equivariant filling structures. We first
generalize the definition of fibration used by Angiuli et al.\ \cite{ABCFHL}, replacing the interval by an
arbitrary type.

\begin{defn}[\formaldef{fibration.fibration}{LocalFillStr}{LocalFillStr}]
  \label{defn:local-fill}
  Let $S$ be a type and $A$ be a family of types over $S$; we define the type $\LocalFillStr_S~A$ of {\em
    local $S$-filling structures} on $A$. These are operations $c_A$ which take as argument $r_0:S$ and
  $a_0:A~r_0$ and a partial section $\psi,u:(\Pi_{r:S}A~r)^{+}$ compatible with $a_0$, i.e.\ such that
  $[\psi]$ implies $u~\toptt~r_0 = a_0$, and produce an element $c_A~r_0~a_0~(\psi,u)$ in $\Pi_{r:S}A~r$ which
  extends $\psi,u$ and such that $c_A~r_0~a_0~(\psi,u)~r_0 = a_0$.
\end{defn}

\begin{defn}[\formaldef{fibration.fibration}{FillStr}{FillStr}]
  Let $S$ be a type and $A$ be a family of types over $\Gamma$. An \emph{$S$-filling structure} $c_A$ on $A$
  consists of a local $S$-filling structure $c_A~\gamma : \LocalFillStr_S~(A \circ \gamma)$ for every
  $\gamma : \Gamma^S$. We write $\Comp_S~A$ for the type of $S$-filling structures on $A$.
\end{defn}

In the cartesian cubical set model of Angiuli et al.\ \cite{ABCFHL}, a type is a family paired with an
$\Ival$-filling structure. To define {\em equivariant} filling structures, we use the case where $S = \Ival^n$ for
some $n : \NN$. In this case the symmetric group $\Sigma_n$ acts in a canonical way on $S$.  It then acts on
$\Gamma^S$ by precomposition, with $\gamma\sigma \coloneq \gamma\circ \sigma$ for $\gamma : \Gamma^S$ and
$\sigma : \Sigma_n$. We likewise have an action on partial elements: given $(\psi,u) : (\Pi_{r:S}A~r)^+$
define $(\psi,u)\sigma : (\Pi_{r:S}A~(\sigma~r))^+$ by $(\psi,u)\sigma \coloneq (\psi,u')$ where
$u'~x~r \coloneq u~x~(\sigma~r)$ for $x : [\psi]$ and $r : S$.

\begin{defn}[\formaldef{fibration.fibration}{FibStr}{FibStr}]
  An \emph{equivariant filling structure} on a family of types $A$ over $\Gamma$ is a family of operations
  $c^n_A$ in $\Comp_{\Ival^n}~\Gamma~A$ for $n : \NN$ each of which is {\em equivariant}, meaning that for any
  $\sigma$ in $\Sigma_n$, we have the {\em equivariance equation}
  \begin{equation}\label{eq:tt-equivariance}
    c^n_A~\gamma~(\sigma~r_0)~a_0~a~(\sigma~r_1) = c^n_{A}~(\gamma\sigma)~r_0~a_0~(a\sigma)~r_1
  \end{equation}
  for every $\gamma : \Gamma^{S}$, $r_0:S$, $a_0:A~(\sigma~(\gamma~r_0))$, compatible partial section $a:(\Pi_{r:S}A~(\gamma~r))^{+}$, and $r_1 : S$.

  We write $\Comp~\Gamma~A$ for the type of all equivariant filling structures on $A$.  These types of
  structure are reflected in each universe, so we have e.g.\
  $\Comp : \Pi_{\Gamma:\Univx}~(\Gamma\rightarrow\Univx)\rightarrow\Univx$.
\end{defn}

\begin{defn}[\formaldef{fibration.fibration}{$\_{\vdash}^{\mathtt{f}}\mathtt{Type}\_$}{_\%e2\%8a\%a2\%e1\%b6\%a0Type_}]
  An \emph{equivariant fibration} over $\Gamma$ is a family of types $A$ over $\Gamma$ paired with an
  equivariant filling structure.
\end{defn}

In this setting, building the model of HoTT consists in showing that each operator on types lifts to an
operator on equivariant filling structures, checking in each case that the output structure satisfies the
equivariance equation (\ref{eq:tt-equivariance}). Let us check for instance that we can interpret substitution
in types; the corresponding property in the external development is the stability of equivariant fibration
structures under pullback.

\begin{defn}[\formaldef{fibration.fibration}{$\_{\circ^{\mathtt{fs}}}\_$}{_\%e2\%88\%98\%e1\%b6\%a0\%cb\%a2_}]
  Let $A$ be a family of types over $\Gamma$ and let $\alpha:\Delta\rightarrow\Gamma$. Given $c_A$ in
  $\Comp~\Gamma~A$, we define $c_A\circ\alpha$ in $\Comp~\Delta~(A\circ\alpha)$ by
  $$
  (c_A\circ\alpha)^n~\gamma~r_0~a_0~a~r_1 \coloneq c^n_A~(\alpha\circ\gamma)~r_0~a_0~a~r_1
  $$
  and it is then clear that $c_A\circ \alpha$ is equivariant if $c_A$ is equivariant.
\end{defn}

\subsection{The Frobenius condition}
\label{ssec:tt-frobenius}

Proving the Frobenius condition, Definition \ref{defn:frobenius}, amounts to defining the interpretation of
$\Pi$-types. The corresponding result in the external, equivariant development is Proposition \ref{prop:equivariant-frobenius}.
A more detailed comparison between external and type-theoretic proofs of the Frobenius condition can be found in \cite[Appendix B]{HazratpourRiehl:2022tc}.

\begin{defn}
  Given a type family $A$ over $\Gamma$ and a type family $B$ over $\Gamma.A$, write $\Pi_AB$ for the
  family of types over $\Gamma$ defined by
  $$
  (\Pi_AB)\gamma \coloneq \Pi_{a:A\gamma}B(\gamma,a) \rlap{.}
  $$
\end{defn}

To prove the Frobenius condition in this setting is to show that, given filling structures on $A$ and $B$, we
have a filling structure on $\Pi_AB$. In fact the hypothesis of a filling structure on $A$ can be weakened: we
only need a \emph{transport structure} in the following sense.

\begin{defn}[\formaldef{fibration.transport}{TranspStr}{TranspStr}]
  Given a type $S$ and family of types $A$ over $\Gamma$, the type $\Transp_S~\Gamma~A$ of {\em $S$-transport
    structures} on $A$ is the type of operations $t_A$ which take $r_0:S$ and $a_0:A~(\gamma~r_0)$ and produce
  an element $t_A~\gamma~r_0~a_0$ in $\Pi_{r:S}A~(\gamma~r)$ such that $t_A~\gamma~r_0~a_0~r_0 = a_0$.

  An \emph{equivariant transport structure} on $A$ is a family of operations $t^n_A : \Transp_n~\Gamma~A$ for
  $n : \NN$ each of which satisfies the equivariance equation
  \[
    t^n_A~\gamma~(\sigma~r_0)~a_0~(\sigma~r_1)
    = t^n_{A}~(\gamma\sigma)~r_0~a_0~r_1
  \]
  for every $\gamma : \Gamma^{S}$, $r_0:S$, $a_0:A~(\sigma~(\gamma~r_0))$, and $r_1 : S$. We write
  $\Transp~\Gamma~A$ for the type of equivariant transport structures on $A$.
\end{defn}

\begin{rmk}[\formaldef{fibration.transport}{transpAndFiberwiseToFibStr}{transpAndFiberwiseToFibStr}]
  It is immediate that any (equivariant) filling structure on a type induces an (equivariant) transport
  structure by restricting to the partial section whose cofibration is $\bot$. As in \cite{ABCFHL}, one can
  conversely construct an equivariant filling structure on $A$ given an equivariant transport structure on $A$
  and an equivariant filling structure on the constant family $A~\gamma$ for every $\gamma : \Gamma$. This
  decomposition would be the key to interpreting higher inductive types following \cite{CHM18,CH19}, but we do
  not pursue this here.
\end{rmk}

\begin{prop}[Frobenius, \formaldef{type-former.pi}{$\Pi$FibStr}{\%ce\%a0FibStr}]
  \label{prop:internal-frobenius}
  Given a family of types $A$ over $\Gamma$ and $B$ over $\Gamma.A$, we have an operation
$$
\Transp~\Gamma~A\rightarrow\Comp~(\Gamma.A)~B\rightarrow \Comp~\Gamma~(\Pi_AB) \rlap{.}
$$
\end{prop}
\begin{proof}
  Let us write $T$ for $\Pi_AB$. Given $t_A$ in $\Transp~\Gamma~A$ and $c_B$ in $\Comp~(\Gamma.A)~B$,
  we define $c_T$ in $\Comp~\Gamma~T$ by
  \begin{equation}\label{eq:tt-frobenius}
    c^n_T~\gamma~r_0~f_0~(\psi,f)~r_1~a_1 \coloneq
    c^n_B ~\PAIR{\gamma,\tilde{a}}~r_0~b_0~(\psi,b)~r_1
  \end{equation}
  where
  \[
    \begin{array}{lclcl}
      \tilde{a} &\coloneq& t_A^n~\gamma~r_1~a_1 &\text{in}& \Pi_{r:S}A~(\gamma~r) \\
      \PAIR{\gamma,\tilde{a}}~r &\coloneq& (\gamma~r,\tilde{a}~r) &\text{in}& (\Gamma.A)^S \\
      b~x~r &\coloneq& f~x~r~(\tilde{a}~r) &\text{in}& (\Pi_{r:S}B(\gamma~r,\tilde{a}~r))^{[\psi]} \\
      b_0 &\coloneq& f_0~(\tilde{a}~r_0) &\text{in}& B~(\gamma~r_0,\tilde{a}~r_0) \rlap{.}
    \end{array}
  \]
  So far this is only a slight generalization of \cite{ABCFHL}, having replaced $\Ival$ by $S = \Ival^n$.

  It remains to check the equivariance equation (\ref{eq:tt-equivariance}) for the operation $c_T$, assuming
  that the operations $t_A$ and $c_B$ are equivariant. Let $\sigma$ be an element of $\Sigma_n$. Write
  $\tilde a, b, b_0$ for the auxiliary definitions associated to
  $c^n_T~\gamma~(\sigma~r_0)~t_0~(\psi,t)~(\sigma~r_1)$ and $\tilde a', b', b_0'$ for those associated to
  $c_T^n~\gamma\sigma~r_0~f_0~(\psi,f)\sigma~r_1$. Then we have
  \begin{align*}
    c^n_T~\gamma~(\sigma~r_0)~t_0~(\psi,t)~(\sigma~r_1)~a_1
    &\coloneq c^n_B ~\PAIR{\gamma,\tilde{a}}~(\sigma~r_0)~b_0~(\psi,b)~(\sigma~r_1) \\
    (\text{equivariance of $c_B$}) &= c^n_B ~\PAIR{\gamma,\tilde{a}}\sigma~r_0~b_0~(\psi,b)\sigma~r_1 \\
    (\text{equivariance of $t_A$}) &= c^n_B ~\PAIR{\gamma\sigma,\tilde{a}'}~r_0~b_0'~(\psi,b')~r_1 \\
    &=: c_T^n~\gamma\sigma~r_0~f_0~(\psi,f)\sigma~r_1~a_1
  \end{align*}
  where we use equivariance of $t_A$ to see that
  $\tilde{a}~(\sigma~r) = t_A~\gamma~(\sigma~r_1)~a_1~(\sigma~r) = t_A~\gamma\sigma~r_1~a_1~r = \tilde{a}'~r$.
\end{proof}

The core of the argument for Frobenius in this type-theoretic setting is thus the defining equation
(\ref{eq:tt-frobenius}).

\begin{rmk}
  \label{rmk:tt-subst-pi}
  To interpret the law $(\Pi_A B)[\rho] = \Pi_{A[\rho]} B[\rho.A]$ for computing a substitution applied to a
  $\Pi$-type, it is also necessary to check that the operation defined in Proposition
  \ref{prop:internal-frobenius} commutes with reindexing along any $\rho : \Delta \to \Gamma$; see
  \formaldef{type-former.pi}{reindex$\Pi$FibStr}{reindex\%ce\%a0FibStr} in the formalization.
\end{rmk}

\subsection{Other type formers}
\label{ssec:tt-type-formers}

We can follow the pattern of the proof of Proposition \ref{prop:internal-frobenius} to lift the other
type-theoretic operations to filling structures: take the corresponding definition of Angiuli et al.\
\cite{ABCFHL}, replace $\Ival$ by $S = \Ival^n$, and check the equivariance equation.

For instance (\formalmodule{type-former.sigma}), we define the $\Sigma$-type $\Sigma_AB$ of families $A$ over
$\Gamma$ and $B$ over $\Gamma.A$ by $(\Sigma_AB)\gamma = \Sigma_{a:A\gamma}B(\gamma,a)$ and build an element
of type
\[
  \Comp~\Gamma~A\rightarrow\Comp~(\Gamma.A)~B\rightarrow \Comp~\Gamma~(\Sigma_AB) \rlap{.}
\]
This corresponds to the closure of fibrations under composition in the external development.

To interpret identity types, we first define path types (\formalmodule{type-former.path}) as an instance of
extension types (\formalmodule{type-former.extension}) \`a la Riehl and Shulman
\cite{RiehlShulman17}. Extension types correspond externally to the closure of fibrations under Leibniz
exponentiation by cofibrations (Proposition \ref{prop:equivariant-sm7}). Path types suffice to interpret
identity types with a propositional computational rule for the eliminator. To interpret identity types with a
judgmental computation rule, we can apply a modification due to Swan to path types \cite[\S9.1]{CCHM:2018ctt}
(\formalmodule{type-former.swan-identity}).

We establish fibrancy and univalence of universes using the $\GlueTy$ types introduced in \cite[\S6]{CCHM:2018ctt} and adapted to cartesian cubical type theory in \cite[\S2.11]{ABCFHL} (\formalmodule{type-former.glue}).
Preliminary $\WeakGlueTy$ types correspond to the equivalence extension property for the equivariant premodel structure proven in Proposition \ref{prop:equivariant-EEP}.
The $\GlueTy$ types and associated properties (\formalmodule{universe.univalence}) correspond to univalence of the universe of equivariantly fibrant types proven in Proposition \ref{prop:equivariant-univalent}.
Mirroring the forward direction of Proposition \ref{prop:EEP-univalence}, this uses realignment for the universe of equivariantly fibrant types (\formalmodule{fibration.realignment}), which is deduced from realignment for the universe of the extensional type theory (\formalmodule{axiom.realignment}) and relative acyclicity of equivariant filling structures (\formaldef{fibration.realignment}{realignFibStr}{realignFibStr}); compare Proposition \ref{prop:has-universes}.

\subsection{Tiny interval and universes}
\label{ssec:tt-universes}

To interpret (univalent) universes, we follow Licata, Orton, Pitts, and Spitters \cite{LOPS18} and work in an
extension of extensional type theory by a modal type operator $\Flat$.  For the purposes of this summary, it
suffices to understand the motivating semantics in cubical sets: if $A$ is a presheaf, then $\Flat A$ is the
constant presheaf of global sections of $A$. We refer to the documentation of the formalization for a precise
description of this setting. We will sometimes refer to an element of $\Flat A$ as a ``global element of
$A$''. In particular, we read $\Flat \Univx$ as the type of external small presheaves. We leave the inclusion
$\Flat A \to A$ implicit in the following.

The use of this modality is to express internally that the interval is tiny, i.e.\  that exponentiation by the
interval $(-)^\Ival$ has a right adjoint \emph{root functor} $\sqrt[\Ival]{-}$ on (external) presheaves, as used in the proof of Lemma \ref{lem:symmetric-interval-tiny}.  Specifically, we require as an
axiom a functorial operator $\sqrt[\Ival]{-} : \Flat \Univx \to \Flat \Univx$ and an isomorphism
\[
  \Flat (A^\Ival \to B) \cong \Flat (A \to \sqrt[\Ival]{B})
\]
natural in $A,B : \Flat \Univx$, exhibiting $\sqrt[\Ival]{-}$ as right adjoint to exponentiation $(-)^\Ival$
(\formalmodule{axiom.tiny}). The restriction to global types is necessary for this axiom to be consistent
\cite[Theorem 5.1]{LOPS18}. By iterating, we also have a right adjoint $\sqrt[S]{-} : \Flat \Univx \to \Flat \Univx$ to
exponentiation by each cube $S = \Ival^n$.

\subsubsection{Dependent right adjoints (\formalmodule{tiny.dependent})}

To construct the universe, it is useful to observe that each right adjoint $\sqrt[S]{-}$ induces a \emph{dependent right adjoint} (spelled out in \cite[Lemma 2.2]{CoquandHuberSattler22}; see also Birkedal et al.\ \cite{Birkedal2020} and \cite[\S5]{LOPS18}).
Note the appearance of similar structure in Lemma \ref{lem:leibniz-pullback-application-loc-rep-fibred-structure} of the external development, which is likewise used to construct universes.

Briefly, for each
$\Gamma : \Flat \Univx$ and global type family $B : \Flat(\Gamma^S \to \Univx)$ we have a family $\sqrt[S]{B}$ over
$\Gamma$ and an isomorphism between \emph{dependent} function types
\[
  \draShut_S : \Flat (\Elem~\Gamma^S~B) \cong \Flat (\Elem~\Gamma~\sqrt[S]{B}) : \draOpen_S
\]
which is natural in $\Gamma$ and $B$ in an appropriate sense.

\begin{rmk}
  Riley \cite{Riley:2024} describes a type theory with a primitive dependent right adjoint of this kind and shows that this structure suffices to carry out the \cite{LOPS18} universe construction without relying on a $\flat$ modality \cite[\S5]{Riley:2024}.
  We use the same style of argument below; although our dependent right adjoint is not primitive, it remains a convenient abstraction, especially in the equivariant case where the universe construction is more involved than in \cite{LOPS18}.
\end{rmk}

\subsubsection{Universe of non-equivariant fibrations}

We now transpose the family $\LocalFillStr_S : \Flat(\Univx^S \to \Univx)$ from Definition \ref{defn:local-fill} to
obtain $\sqrt[S]{\LocalFillStr_S}$ over $\Univx$ with the property that for any global family
$A : \Flat(\Gamma \to \Univx)$ we have
\begin{equation}\label{eq:non-equivariant-fib}
  \Flat (\Elem~\Gamma~(\sqrt[S]{\LocalFillStr_S} \circ A))
  \cong \Flat (\Elem~\Gamma^S~(\LocalFillStr_S \circ A^S))
  = \Flat (\Comp_S~A) \rlap{.}
\end{equation}

\begin{defn}
  Define $\Univf_S \coloneq \Sigma_{A:\Univx}\sqrt[S]{\LocalFillStr_S}~A$.
\end{defn}

From (\ref{eq:non-equivariant-fib}), we have an isomorphism for $\Gamma : \Flat \Univx$ between global families
$\Gamma \to \Univf_S$ and global $\Univx$-small families over $\Gamma$ paired with $S$-filling structures. Note
that the type $\Univf_\Ival$ is exactly the universe defined in \cite{LOPS18}.
\begin{defn}
  \label{defn:get-local-filling}
  Leaving the first projection $\pi_1 : \Univf_S \to \Univx$ implicit, we transpose the projection
  $\pi_2 : \Pi_{A : \Univf_S} \sqrt[S]{\LocalFillStr_S}~A$ to yield a map
  $\draOpen_S~\pi_2 : \Pi_{A : (\Univf_S)^S} \LocalFillStr_S~A$ that associates a local filling structure
  to every $S$-cell in the universe.
\end{defn}

\subsubsection{Universe of equivariant fibrations}

To further restrict to universes of \emph{equivariant} fibrations, we introduce a another predicate on
elements of $\Univf_S$.

\begin{defn}[\formaldef{universe.core}{LocalEquivariance$\surd$}{LocalEquivariance\%e2\%88\%9a}]
  Fix $S = \Ival^n$ and let $A : (\Univf_S)^S$. Per Definition \ref{defn:get-local-filling}, we have
  $\draOpen_S~\pi_2~A : \LocalFillStr_S~A$. For each $\sigma$ in $\Sigma_n$, we also have
  $\draOpen_S~\pi_2~(A\sigma) : \LocalFillStr_S~(A\sigma)$. We say \emph{$A$ is equivariant} when for each
  $\sigma$ in $\Sigma_n$, we have
  \[
    \draOpen_S~\pi_2~A~(\sigma~r_0)~a_0~a~(\sigma~r_1) = \draOpen_S~\pi_2~(A\sigma)~r_0~a_0~(a\sigma)~r_1
  \]
  for all $r_0:S$, $a_0:A~(\sigma~r_0)$, partial sections $a:(\Pi_{r:S}A~r)^{+}$ compatible with $a_0$, and
  $r_1 : S$.

  We write $\Equivariant_S~A$ for the type of proofs that $A$ is equivariant.
\end{defn}

\begin{defn}[\formaldef{universe.core}{$\mathcal{U}$}{\%f0\%9d\%91\%bc}]
  Given $A : \Univx$, we define the type of \emph{equivariant fibration structures} on $A$ by
  \[
    \FibStr~A \coloneq \textstyle\prod_{n:\NN}\sum_{F:\sqrt[\Ival^n]{\LocalFillStr_{\Ival^n}}~A} \prod_{\sigma:\Sigma_n} \sqrt[\Ival^n]{\Equivariant_{\Ival^n}}~(A,F) \rlap{.}
  \]
  The \emph{universe of equivariant fibrations} is then $\Univf \coloneq \textstyle\sum_{A:\Univx} \FibStr~A$.
\end{defn}

\begin{prop}[\formalmodule{universe.core}]
  We have for each $\Gamma : \Flat \Univx$ an isomorphism
  \begin{equation}\label{eq:tt-fib-predicate}
    \Flat (\Elem~\Gamma~(\FibStr \circ A)) \cong \Flat (\Comp~\Gamma~A)
  \end{equation}
  and therefore an isomorphism between global families $\Gamma \to \Univf$ and global $\Univx$-small equivariant
  fibrations over $\Gamma$.
\end{prop}

The existence of such a predicate $\FibStr$ corresponds to the local representability of equivariant fibrations
(Lemma \ref{lem:equiv-loc-rep-fib}): for a family $A$ over $\Gamma$, the family $\FibStr \circ A$ over $\Gamma$
corresponds to the representing morphism $\psi_\pi$ for the projection $\pi \colon \Gamma.A \to \Gamma$. In
the external development, local representability of equivariant fibrations is derived from local
representability of fibrations in cubical species (Lemma \ref{lem:loc-rep-fib}), which uses the tininess of the
symmetric interval (Lemma \ref{lem:symmetric-interval-tiny}) and thus, like the construction here, ultimately
depends on the tininess of the cubes $\Ival^n$.

\subsubsection{Fibrancy of the universe (\formalmodule{universe.fibrant})}

As with the other type formers, we construct a fibrancy structure on the universe by generalizing the
definition of Angiuli et al.\ \cite[\S2.12]{ABCFHL} from $\Ival$ to $\Ival^n$ and checking that this
satisfies the equivariance equation. The construction relies on the $\GlueTy$ types mentioned in Section
\ref{ssec:tt-type-formers}. The corresponding argument in the external development is in \ref{ssec:Ufib},
based on the same definition of Angiuli et al.; there it is conducted in cubical species and then transferred
to cubical sets in Proposition \ref{prop:equivariant-FEP}.

When we have a larger universe $\Univx_1$ with $\Univx : \Univx_1$, we can repeat the definitions above to define a
predicate $\FibStr_1$ and universe of $\Univx_1$-small fibrations $\Univf_1 \coloneq \textstyle\sum_{A:\Univx} \FibStr_1~A$;
the fibrancy of $\Univx$ then implies that $\Univf_1$ contains a code for $\Univf$. More generally, a hierarchy of
universes $\Univx_n$ in the extensional type theory gives rise to a corresponding hierarchy of universes $\Univf_n$
in the homotopical interpretation.

\subsubsection{Type formers (\formalmodule{universe.type-former})}

Using the closure properties of the operation $\Comp$ established in Sections \ref{ssec:tt-frobenius} and
\ref{ssec:tt-type-formers} and the bijection (\ref{eq:tt-fib-predicate}), we can build operations of types

\begin{equation}
  \label{eq:tt-universe-type-formers}
  \begin{array}{l}
    \Pi_{A:\Univx}\Pi_{B:A\rightarrow\Univx}\FibStr~A\rightarrow (\Pi_{a:A}\FibStr~(B~a))\rightarrow \FibStr~(\Pi_AB)\\
    \Pi_{A:\Univx}\Pi_{B:A\rightarrow\Univx}\FibStr~A\rightarrow (\Pi_{a:A}\FibStr~(B~a))\rightarrow \FibStr~(\Sigma_AB)\\
    \Pi_{A:\Univx}\Pi_{a_0:A}\Pi_{a_1:A}\FibStr~A\rightarrow \FibStr~(\PathTy~A~a_0~a_1) \rlap{.}\\
  \end{array}
\end{equation}

From these, we deduce that $\Univf$ is closed under $\Pi$-types, $\Sigma$-types, and $\PathTy$-types. We also have
an alternative, isomorphic definition of the judgments of the homotopical interpretation: we can interpret
types over $\Gamma$ as maps $\Gamma \to \Univf$ rather than as families over $\Gamma$ with equivariant filling
structures. Because the type formers can then be defined pointwise by the operators shown in
(\ref{eq:tt-universe-type-formers}), the laws for computing substitutions such as mentioned in Remark
\ref{rmk:tt-subst-pi} become automatic; this is the same technique exploited by Voevodsky to solve the coherence problem in the simplicial model \cite{KapulkinLumsdaine:2021sm}.

\subsubsection{Univalence (\formalmodule{universe.univalence})}

Finally, the closure of the universe under $\GlueTy$-types provides an element of
\[
  \Pi_{A:\Univf}\Contr(\Sigma_{X:\Univf}\EquivTy~A~X)
\]
where $\EquivTy~A~X$ is the type of homotopy equivalences and $\Contr$ is the type of contractibility structures
from Section \ref{sec:tt-trivial-fibration}. This corresponds to the equivalence extension property for
$\Univx$-small fibrations. Any family with a trivial fibration structure is contractible as a type in the
homotopical interpretation, in the sense of \cite[\S3.11]{UF:2013}
(\formaldef{type-former.hlevels}{TFibToIsContr}{TFibToIsContr}). Thus the homotopical interpretation satisfies
the axiom
\[
  \Pi_{A:\Univf}\isContrTy(\Sigma_{X:\Univf}\EquivTy~A~X)
\]
where $\isContrTy~A \coloneq \Sigma_{a_0:A} \Pi_{a:A} (\PathTy~A~a_0~a)$. This is an equivalent formulation of the
univalence axiom (as observed by Escard\'o \cite{escardo-univalence14}); compare the derivation of univalence
from the equivalence extension property and Frobenius condition at the start of Section \ref{ssec:univalence}.


\raggedright
\printbibliography

\end{document}